\documentclass[final,onefignum,onetabnum]{siamart171218}

\usepackage{braket,amsfonts}
\usepackage{amsmath}
\usepackage{amssymb}
\usepackage{graphicx}
\usepackage{placeins}
\usepackage{longtable}
\usepackage[utf8]{inputenc}
\usepackage{babel}
\usepackage[font=small,labelfont=bf]{caption}
\usepackage{array}
\usepackage{longtable}

\usepackage[margin=1.0in,footskip=0.25in]{geometry}
\usepackage[normalem]{ulem}

\usepackage[caption=false]{subfig}
\usepackage{pgfplots}

\allowdisplaybreaks

\newcommand{\norm}[1]{\ensuremath{\left\|#1\right\|}}
\newcommand{\tnorm}[1]{{\left\vert\kern-0.25ex\left\vert\kern-0.25ex\left\vert #1 
		\right\vert\kern-0.25ex\right\vert\kern-0.25ex\right\vert}}
\newcommand{\avg}[1]{\{ \hspace{-0.1cm}\{#1\} \hspace{-0.1cm}\}}
\newcommand{\jump}[1]{\ensuremath{[\![#1]\!]} }

\newsiamthm{claim}{Claim}
\newsiamremark{remark}{Remark}
\newsiamremark{hypothesis}{Hypothesis}
\crefname{hypothesis}{Hypothesis}{Hypotheses}

\usepackage{algorithmic}

\usepackage{graphicx,epstopdf}

\crefname{ALC@unique}{Line}{Lines}

\usepackage{amsopn}

\DeclareMathOperator*{\esssup}{ess\,sup}

\usepackage{xspace}
\usepackage{bold-extra}
\usepackage[most]{tcolorbox}

\colorlet{texcscolor}{blue!50!black}
\colorlet{texemcolor}{red!70!black}
\colorlet{texpreamble}{red!70!black}
\colorlet{codebackground}{black!25!white!25}

\allowdisplaybreaks

\lstdefinestyle{siamlatex}{%
	style=tcblatex,
	texcsstyle=*\color{texcscolor},
	texcsstyle=[2]\color{texemcolor},
	keywordstyle=[2]\color{texemcolor},
	moretexcs={cref,cref,maketitle,mathcal,text,headers,email,url},
}

\tcbset{%
	colframe=black!75!white!75,
	coltitle=white,
	colback=codebackground, 
	colbacklower=white, 
	fonttitle=\bfseries,
	arc=0pt,outer arc=0pt,
	top=1pt,bottom=1pt,left=1mm,right=1mm,middle=1mm,boxsep=1mm,
	leftrule=0.3mm,rightrule=0.3mm,toprule=0.3mm,bottomrule=0.3mm,
	listing options={style=siamlatex}
}

\newtcblisting[use counter=example]{example}[2][]{%
	title={Example~\thetcbcounter: #2},#1}

\newtcbinputlisting[use counter=example]{\examplefile}[3][]{%
	title={Example~\thetcbcounter: #2},listing file={#3},#1}

\DeclareTotalTCBox{\code}{ v O{} }
{ 
	fontupper=\ttfamily\color{black},
	nobeforeafter,
	tcbox raise base,
	colback=codebackground,colframe=white,
	top=0pt,bottom=0pt,left=0mm,right=0mm,
	leftrule=0pt,rightrule=0pt,toprule=0mm,bottomrule=0mm,
	boxsep=0.5mm,
	#2}{#1}

		\makeatletter
		\newcommand*{\bdiv}{%
			\nonscript\mskip-\medmuskip\mkern5mu%
			\mathbin{\operator@font div}\penalty900\mkern5mu%
			\nonscript\mskip-\medmuskip
		}
		\makeatother

\patchcmd\newpage{\vfil}{}{}{}
\begin{tcbverbatimwrite}{tmp_\jobname_header.tex}
	\title{Non-conforming structure preserving finite element method for doubly diffusive flows on bounded Lipschitz domains \thanks{\funding{This work was supported by the SERB-CRG India (Grant Number : CRG/2021/002569).}}
	}
\author{Jai Tushar\thanks{Center for Computation and Technology, Louisiana State University, Baton Rouge, USA 
		(\email{jai.tushar@lsu.edu}).}
	\and Arbaz Khan\thanks{Department of Mathematics, Indian Institute of Technology, Roorkee, India 
		(\email{arbaz@ma.iitr.ac.in}, \email{maniltmohan@ma.iitr.ac.in}).}
	\and Manil T. Mohan\footnotemark[3]}

\headers{ Non-conforming method for doubly diffusive flows }{J. Tushar, A. Khan, and M. T. Mohan}
\end{tcbverbatimwrite}
\input{tmp_\jobname_header.tex}

\ifpdf
\hypersetup{pdftitle={Guide to Using  SIAM'S \LaTeX\ Style} }
\fi

\begin{document}
\maketitle
\date{\today}
\begin{abstract}
	We study a stationary model of doubly diffusive flows with temperature-dependent viscosity on bounded Lipschitz domains in two and three dimensions. A new well-posedness and regularity analysis of weak solutions under minimal assumptions on domain geometry and data regularity are established. A fully non-conforming finite element method based on Crouzeix-Raviart elements, which ensures locally exactly divergence-free velocity fields is explored. Unlike previously proposed schemes, this discretization enables to establish uniqueness of the discrete solutions. We prove the well-posedness of the discrete problem and derive a priori error estimates. An accuracy test is conducted to verify the theoretical error decay rates  in flow, Stokes and Darcy regimes on convex and non-convex domains, and a benchmark test of flow in a porous cavity is conducted, comparing the proposed method with existing literature.
\end{abstract}


\section{Introduction}\label{Intro}
Doubly diffusive flows describe single or multiphase fluid flows driven by two different density gradients with varying rates of diffusion and are important in understanding the evolution of various natural and artificial systems. For example, when particles settle through a stable temperature or salinity gradient, they can drive an instability known as sedimentary fingering convection, which explains salt fingering occurrences responsible for mixing and transporting nutrients and salts in the ocean and has significant implications in maintaining Earth’s climate (see for reference \cite{MR3621138, oschlies2003salt}). Such flows are also observed in applications such as petroleum extraction, crystal growth, food processing and many others (see \cite{balla2015soret, MR3342441, MR1004641}). In this article, we study a stationary model of such flows governed by a nonlinear system coupling the incompressible Navier-Stokes equations with temperature and concentration-dependent advection-diffusion equations:
\begin{align}\label{P:GE}
	\left\{
	\begin{aligned}
		\boldsymbol{K}^{-1} \boldsymbol{u} + (\boldsymbol{u} \cdot \nabla) \boldsymbol{u} - \boldsymbol{\bdiv}(\nu(T) \nabla \boldsymbol{u}) + \nabla p &= \boldsymbol{F}(\boldsymbol{y}) \;\; \mbox{in} \;\; \Omega, \\
		\bdiv\hspace{0.04cm}\boldsymbol{u} &= 0 \;\; \mbox{in} \;\; \Omega, \\
		-\boldsymbol{\bdiv}\hspace{-0.1cm}(\boldsymbol{D} \nabla \boldsymbol{y}) + (\boldsymbol{u} \cdot \nabla) \boldsymbol{y} &= 0 \;\; \mbox{in} \;\; \Omega, \\
		\boldsymbol{y} = \boldsymbol{y}^D, \;\; \boldsymbol{u} &= \boldsymbol{0} \;\; \mbox{on} \;\; \Gamma.
	\end{aligned}
	\right.
\end{align}
Here $\boldsymbol{u}$ denotes the fluid velocity, $p$ stands for the pressure field, $\boldsymbol{y}:= (T, S)^\top$, $S$ represents the concentration of a certain species within this fluid, and $T$ denotes the temperature. The domain $\Omega \subset \mathbb{R}^d \; (d = 2, 3)$ is  bounded with Lipschitz boundary $\Gamma$. The temperature-dependent viscosity function is denoted by $\nu>0$, $\boldsymbol{K}(x) > 0$ is the permeability matrix, $\boldsymbol{F}(\boldsymbol{y})$ is a given function modelling buoyancy, and $\boldsymbol{D}$ is a $2 \times 2$  constant matrix of the thermal conductivity and solutal diffusivity coefficients, possibly with cross-diffusion terms.

	\begin{remark}\label{rem:third-eqn-notation}
		For our convenience, in the third equation of \eqref{P:GE} we use the condensed vector notation
		$ \boldsymbol{y}:=(T,S)^\top \; \mbox{and} \; \boldsymbol{D}\in\mathbb{R}^{2\times 2}$. 
		The differential operators $\nabla$ and $\bdiv$ act component-wise on vector-valued functions and row-wise on matrix-valued fluxes.
		Therefore, the third equation of \eqref{P:GE} is equivalent to the coupled scalar advection--diffusion system
		$$
		\left\{
		\begin{aligned}
			-\bdiv\left(D_{11}\nabla T + D_{12}\nabla S \right) + \boldsymbol{u} \cdot \nabla T &= 0 \quad \mbox{in } \Omega,\\
			-\bdiv\left(D_{21}\nabla T + D_{22}\nabla S \right) + \boldsymbol{u} \cdot \nabla S &= 0 \quad \mbox{in} \Omega,
		\end{aligned}
		\right.
		$$
		with the boundary condition $\boldsymbol{y}=\boldsymbol{y}^D$ on $\Gamma$, i.e., $T=T^D$ and $S=S^D$ on $\Gamma$.
\end{remark}

\subsection{Assumptions}\label{Assum}
Throughout this article, we make the following assumptions on the governing equation \eqref{P:GE}:

	\smallskip
	
	\noindent\emph{Boundary data regularity.} The boundary data satisfies $\boldsymbol{y}^D = (T^D, S^D)^\top \in [H^{1/2}(\Gamma)]^2$.
	
	\smallskip
	
	\noindent\emph{Lipschitz continuity and uniform boundedness of the kinematic viscosity.} The viscosity $\nu(\cdot)$ satisfies
	$$
	|\nu(T_1) - \nu(T_2)| \leq \gamma_{\nu} |T_1 - T_2|
	\quad \mbox{and} \quad
	\nu_1 \leq \nu(T) \leq \nu_2,
	$$
	for all $T_1,T_2,T\in\mathbb{R}$, where $\gamma_{\nu}\geq 0$, $\nu_1,\nu_2$ are positive constants and $|\cdot|$ denotes the Euclidean norm in $\mathbb{R}^d$.
	
	\smallskip
	
	\noindent\emph{Lipschitz continuity of the buoyancy term.} There exist positive constants $\gamma_F, C_F$ such that
	$$
	|\boldsymbol{F}(\boldsymbol{y}_1) - \boldsymbol{F}(\boldsymbol{y}_2)| \leq \gamma_F |\boldsymbol{y}_1 - \boldsymbol{y}_2|
	\quad \mbox{and} \quad
	|\boldsymbol{F}(\boldsymbol{y})| \leq C_F |\boldsymbol{y}|,
	$$
	for all $\boldsymbol{y}_1,\boldsymbol{y}_2,\boldsymbol{y}\in\mathbb{R}^2$.
	
	\smallskip
	
	\noindent\emph{Uniform positive definiteness of permeability matrix.} The permeability matrix $\boldsymbol{K}$ is a $d\times d$ matrix with measurable coefficients which is assumed to be symmetric and uniformly positive definite; hence, its inverse satisfies
	$\boldsymbol{v}^\top \boldsymbol{K}^{-1}(x) \boldsymbol{v} \geq \alpha_1 |\boldsymbol{v}|^2$
	for all $\boldsymbol{v} \in \mathbb{R}^d$ and $x \in \Omega$, for a constant $\alpha_1 > 0$.
	
	\smallskip
	
	\noindent\emph{Positive definiteness of diffusion matrix.} The constant matrix $\boldsymbol{D}$ is assumed to satisfy the positive definiteness condition (though not necessarily symmetric), that is,
	$\boldsymbol{s}^\top \boldsymbol{D} \boldsymbol{s} \geq \alpha_2 |\boldsymbol{s}|^2$
	for all $\boldsymbol{s} \in \mathbb{R}^2$, for a constant $\alpha_2 > 0$.

\subsection{Related works} To analyse the proposed problem, the first step is to study the well-posedness of  \eqref{P:GE} under suitable assumptions.
The solvability of stationary Navier-Stokes equations on bounded domains with smooth boundary has been studied in the classical works of \cite{temam2001navier,Temam_1995}, and the case of bounded polygonal domains with Lipschitz boundary has been addressed in the works of \cite{Nicaise, MD_LectureNotes} and references therein.   Since \eqref{P:GE} is an extension of the classical Bousinessq-type equation, we now discuss the results in the literature regarding them. The solvability of both the stationary and transient Bousinessq PDEs with temperature-dependent viscosity and thermal conductivity is credited to the works of \cite{lorca1996stationary, MR1675260}. More specifically in the context of \eqref{P:GE},  a nonlinear energy stability theory is established  in \cite{guo1995double} and more recently, the existence of a weak solution, and a uniqueness result under strong regularity assumptions on the velocity, temperature and concentration was shown in \cite{burger2019h}. 

\subsection{Main contributions and outline}\label{Main_Contributions}
The well-posedness of the system in domains with smooth boundaries is relatively well-understood. However, challenges arise when addressing minimal regularity assumptions on both data and geometry. In order to prove the uniqueness of weak solutions (a kind of weak-strong uniqueness), \cite{burger2019h} used  a strong assumption that the divergence free fluid velocity $\boldsymbol{u} \in \boldsymbol{W}^{1,\infty}(\Omega)$ and $\boldsymbol{y}\in[L^{\infty}(\Omega)]^2$.  This is valid under the regularity assumptions (see Chapter II, Section 1.3, Proposition 1.1 in \cite{temam2001navier}) on the domain and the data $\boldsymbol{F}(\boldsymbol{y}) \in \boldsymbol{H}^1(\Omega)$, which dictate that the maximum regularity of strong solution $\boldsymbol{u}$ is $\boldsymbol{H}^3(\Omega)$. However, when the domain has a Lipschitz boundary only ( or when the data is in a weaker space (for example, $\boldsymbol{L}^2(\Omega)$), the analysis in \cite{burger2019h} is not applicable. We deal with the existence of solutions  by using a Faedo-Galerkin approximation technique in Section \ref{Sec:StateExistence}, and the uniqueness of weak solution of the governing equation is proved using minimal regularity of $\boldsymbol{u}$ on  certain class of bounded Lipschitz domains under a small data assumption (see Section \ref{Sec:RegularityState}, Theorem \ref{StateUniqueness} and Remark \ref{smalldata}).

\smallskip

We prove regularity results (see Theorem \ref{Regularity}) for the weak solutions of \eqref{P:GE} on a certain class of bounded Lipschitz domains in two and three dimensions (see Section \ref{Sec:RegularityState}). The significant difficulties lie in estimating the nonlinear terms $(\boldsymbol{u} \cdot \nabla) \boldsymbol{u}$, $\boldsymbol{\bdiv}\hspace{-0.04cm}(\nu(T) \nabla \boldsymbol{u}) $ and $(\boldsymbol{u} \cdot \nabla) \boldsymbol{y}$. They are handled by proving the regularity in a less regular space and then using it as a stepping stone to reach the desired regularity. The techniques and results presented are of independent interest to other problems with a similar nature of nonlinearity. For example Bousinessq model problems. 

\smallskip

Owing to their attractive features a popular choice of finite elements for this class of problems in the literature is the H(div)-conforming family \cite{oyarzua2014exactly}. However, proving the uniqueness of the corresponding discrete system  remains open \cite{burger2019h,oyarzua2014exactly}, which is due to the difficulty in controlling an augmented penalty dependent-norm. To overcome this we explore a fully non-conforming finite element discretization of \eqref{P:GE} based on the lowest order Crouzeix-Raviart (CR) finite element \cite{SB_CRreview} and piecewise constant spaces in Section \ref{Sec:nc:S}.  This discretization is ``strucutre preserving" in the sense that it produces locally exactly divergence-free velocity approximations, which are of particular importance in ensuring that solutions remain locally conservative as well as energy stable \cite{DS_LDG}. We note that, under higher regularity assumption on the continuous solution as mentioned above,  the choice of CR-DG finite element pair to discretize the problem is sufficient to prove the uniqueness of the discrete problem. However, the subsequent discrete analysis will inherit this assumption. On the contrary, our minimal-regularity framework allows to establish  discrete results under milder assumptions on data, geometry and continuous solution. The well-posedness of the discrete equation is studied using discrete lifting \cite{m2an_discretelifting} and fixed point  arguments in Theorems \ref{ncExistenceState} and \ref{ncUniqueness}. In Section \ref{NA:apriori:State}, we derive  \textit{a priori} error estimates under minimal regularity (see Theorem \ref{NA_State}) for the proposed equation. Finally in Section \ref{NE}, we perform a numerical experiment to test the convergence of the proposed scheme in the flow, Stokes and Darcy regimes, and conduct a benchmark test comparing the proposed scheme with the existing literature.

\subsection{Notations}
Let $\Omega \subset \mathbb{R}^d \; (d = 2, 3)$ be a bounded domain with Lipschitz boundary $\Gamma$. We use the following notations throughout this article. The usual $r^{\mathrm{th}}$ integrable Lebesgue measurable function spaces are denoted by $L^{r}(\Omega)$ with the norm $\|f\|_{L^r(\Omega)}=\left(\int_{\Omega}|f(x)|^rdx\right)^{1/r},$ for $f\in L^r(\Omega), \; r \in [1,\infty)$. For $r = \infty$, $L^{\infty}(\Omega)$ is the space of all essentially bounded measurable functions on $\Omega$ with the norm $\norm{f}_{\infty,\Omega} = \esssup\limits_{x \in \Omega} (|f(x)|), $ for $f \in L^{\infty}(\Omega)$ and for $r = 2$ we denote the norm by $\norm{\cdot}_{0,\Omega}$. The space of square integrable Lebesgue measurable functions with zero mean is defined as $$L_0^2(\Omega) := \left\{f\in L^2(\Omega): \int_{\Omega} f(x) \; dx = 0\right\},$$
and we define $(f,g):=\int_{\Omega}f(x)g(x)dx,$ for $f,g\in L^2(\Omega)$,  the usual inner product in $L^2(\Omega)$. The space of $C^{\infty}$-functions with compact support contained in $\Omega$ is denoted by $\left[C_0^{\infty}(\Omega)\right]^d$. Sobolev spaces are denoted by the standard notation $W^{k,r}(\Omega)$  with the norm
\begin{align*}
	\|u\|_{W^{k,r}(\Omega)}&=\left(\sum\limits_{|\alpha|\leq k}\int_{\Omega}|D^{\alpha}u(x)|^rdx\right)^{1/r},\;\mbox{if}\;  r \in [1,\infty),\\
	\norm{u}_{W^{k,\infty}(\Omega)} &= \max_{|\alpha| \leq k} \left\{\esssup\limits_{x \in \Omega} |D^{\alpha} u(x)|\right\}, \; \mbox{if} \; r = \infty,
\end{align*}
for $u\in W^{k,r}(\Omega)$, $k\in\mathbb{N}$ and $\alpha$ is the multi-index. Specifically, for $r = 2,$ we use the notation $W^{k,2}(\Omega) = H^{k}(\Omega)$ with the corresponding norm denoted by $\norm{\cdot}_{k,\Omega}.$ The fractional Sobolev spaces will also be denoted in an analogous way (see \cite{HitchikersGuideFractionalSobolev} for more details).  The vector valued functions in dimension $d$ are denoted by bold face and we define 
$$\boldsymbol{H}_0^1(\Omega) := \left\{\boldsymbol{v} \in \boldsymbol{H}^1(\Omega) : \gamma_0(\boldsymbol{v}) = \boldsymbol{0} \; \mbox{ on } \; \Gamma \right\},$$ where $\gamma_0 : \boldsymbol{H}^1(\Omega) \longrightarrow \boldsymbol{H}^{1/2}(\Gamma)$ is the trace operator, and the norm on $\boldsymbol{H}^{1/2}(\Gamma)$ is defined as,
$$\norm{\boldsymbol{g}}_{1/2,\Gamma} = \inf \big\{\norm{\boldsymbol{v}}_{1,\Omega} : \boldsymbol{v} \in \boldsymbol{H}^1(\Omega) \;\; \mbox{and} \;\; \gamma_0(\boldsymbol{v}) = \boldsymbol{g}\big\}.$$
The dual space of $\boldsymbol{H}_0^1(\Omega)$ is denoted by $\boldsymbol{H}^{-1}(\Omega)$ with the following norm:
$$\norm{\boldsymbol{u}}_{-1,\Omega} := \sup_{0 \neq \boldsymbol{v} \in \boldsymbol{H}_0^1(\Omega)} \frac{\langle \boldsymbol{u}, \boldsymbol{v} \rangle}{~~\norm{\boldsymbol{v}}_{1,\Omega}},$$
where $\langle \cdot, \cdot \rangle$  denotes the duality pairing between $\boldsymbol{H}_0^1(\Omega)$ and $\boldsymbol{H}^{-1}(\Omega).$	
We will also need the following vector-valued Hilbert space
\begin{align*}
	\boldsymbol{H}(\bdiv; \Omega) &:= \left\{\boldsymbol{v} \in \boldsymbol{L}^2(\Omega) : \bdiv\hspace{0.04cm}\boldsymbol{v} \in L^2(\Omega)\right\},
\end{align*}
endowed with the norm $\norm{\boldsymbol{w}}^2_{\bdiv, \Omega} := \norm{\boldsymbol{w}}^2_{0,\Omega} + \norm{\bdiv\hspace{0.04cm}\boldsymbol{w}}^2_{0,\Omega};$ the outward normal on $\Gamma$ is denoted by $\boldsymbol{n}_{\Gamma}$.
The duality product between a space $\boldsymbol{V}$ and its dual $\boldsymbol{V}'$ is denoted by $\langle \boldsymbol{f},\boldsymbol{v} \rangle$, where $\boldsymbol{f} \in \boldsymbol{V}'$ and $\boldsymbol{v} \in \boldsymbol{V}$. Throughout the article $C$ will denote a generic positive constant. 
\section{The functional setting}\label{CtsStateEq}
The variational formulation of the governing equation \eqref{P:GE} is obtained by testing against suitable functions and integrating by parts, and can be formulated as follows:\\

Find $(\boldsymbol{u},p,\boldsymbol{y}) \in \boldsymbol{H}_0^1(\Omega) \times L_0^2(\Omega) \times [H^1(\Omega)]^2$ satisfying $\boldsymbol{y} = \boldsymbol{y}^D$ on $\Gamma$ and 
\begin{align}\label{P:S}
	\left\{
	\begin{aligned}
		a(\boldsymbol{y}; \boldsymbol{u}, \boldsymbol{v}) + c(\boldsymbol{u}, \boldsymbol{u}, \boldsymbol{v}) + b(\boldsymbol{v},p) &= d(\boldsymbol{y},\boldsymbol{v}) \;\; \forall \;\; \boldsymbol{v} \in \boldsymbol{H}_0^1(\Omega), \\
		b(\boldsymbol{u},q) &= 0 \;\; \forall \;\; q \in L_0^2(\Omega), \\
		a_{\boldsymbol{y}}(\boldsymbol{y},\boldsymbol{s}) + c_{\boldsymbol{y}}(\boldsymbol{u},\boldsymbol{y},\boldsymbol{s}) &= 0 \;\; \forall \;\; \boldsymbol{s} \in [H_0^1(\Omega)]^2. 
	\end{aligned}
	\right.
\end{align}

The linear and nonlinear forms in \eqref{P:S} are defined as follows. 
	The bilinear form $a(\cdot;\cdot,\cdot): [H^1(\Omega)]^2 \times \boldsymbol{H}_0^1(\Omega) \times \boldsymbol{H}_0^1(\Omega) \longrightarrow \mathbb{R}$ is defined by
	$$
	a(\boldsymbol{y}; \boldsymbol{u}, \boldsymbol{v}) := (\boldsymbol{K}^{-1} \boldsymbol{u}, \boldsymbol{v}) + (\nu(\boldsymbol{y}) \nabla \boldsymbol{u}, \nabla \boldsymbol{v}),
	$$
	where $\nu(\boldsymbol{y})$ is understood as the kinematic viscosity depending only on the first component of the vector $\boldsymbol{y}$.
	The trilinear form $c(\cdot,\cdot,\cdot): \boldsymbol{H}_0^1(\Omega) \times \boldsymbol{H}_0^1(\Omega) \times \boldsymbol{H}_0^1(\Omega) \longrightarrow \mathbb{R}$ is defined by
	$$
	c(\boldsymbol{w},\boldsymbol{u},\boldsymbol{v}) := ((\boldsymbol{w} \cdot \nabla) \boldsymbol{u}, \boldsymbol{v}).
	$$
	The bilinear form $b(\cdot,\cdot): \boldsymbol{H}_0^1(\Omega) \times L_0^2(\Omega) \longrightarrow \mathbb{R}$ is defined by
	$$
	b(\boldsymbol{v},q) := - (q, \bdiv\hspace{0.04cm} \boldsymbol{v}).
	$$
	The linear form $d(\cdot,\cdot): \boldsymbol{H}^1(\Omega) \times \boldsymbol{H}_0^1(\Omega) \longrightarrow \mathbb{R}$ is defined by
	$$
	d(\boldsymbol{s},\boldsymbol{v}) := (\boldsymbol{F}(\boldsymbol{s}),\boldsymbol{v}).
	$$
	The bilinear form $a_{\boldsymbol{y}}(\cdot,\cdot): [H^1(\Omega)]^2 \times [H^1(\Omega)]^2 \longrightarrow \mathbb{R}$ is defined by
	$$
	a_{\boldsymbol{y}}(\boldsymbol{y},\boldsymbol{s}) := (\boldsymbol{D} \nabla\boldsymbol{y}, \nabla \boldsymbol{s}).
	$$
	The trilinear form $c_{\boldsymbol{y}}(\cdot,\cdot,\cdot) : \boldsymbol{H}_0^1(\Omega) \times [H^1(\Omega)]^2 \times [H^1(\Omega)]^2 \longrightarrow \mathbb{R}$ is defined by
	$$
	c_{\boldsymbol{y}}(\boldsymbol{v},\boldsymbol{y},\boldsymbol{s}) := ((\boldsymbol{v}\cdot\nabla)\boldsymbol{y},\boldsymbol{s}).
	$$

\subsection{Properties of the operators}\label{Sec:SP}
Due to the assumptions on the governing equation, the following boundedness properties hold for all $\boldsymbol{u}, \boldsymbol{v} \in \boldsymbol{H}_0^1(\Omega)$, $q \in L^2(\Omega)$, and $\boldsymbol{y}, \boldsymbol{s} \in [H^1(\Omega)]^2:$
\begin{align}
	|a(\boldsymbol{y};\boldsymbol{u},\boldsymbol{v})| &\leq \norm{\boldsymbol{K}^{-1}}_{\infty,\Omega} \norm{\boldsymbol{u}}_{0,\Omega} \norm{\boldsymbol{v}}_{0,\Omega} + \norm{\nu(\boldsymbol{y})}_{\infty,\Omega} \norm{ \nabla \boldsymbol{u}}_{0,\Omega} \norm{ \nabla \boldsymbol{v}}_{0,\Omega} \nonumber\\
	&\leq \max\left\{\norm{\boldsymbol{K}^{-1}}_{\infty,\Omega}, \nu_2\right\} \left(\norm{\boldsymbol{u}}_{0,\Omega} \norm{\boldsymbol{v}}_{0,\Omega} + \norm{ \nabla \boldsymbol{u}}_{0,\Omega} \norm{ \nabla \boldsymbol{v}}_{0,\Omega}\right) \nonumber\\
	&\leq C_a\|\boldsymbol{u}\|_{1,\Omega}\|\boldsymbol{v}\|_{1,\Omega}, \label{Cty:a}\\
	|a_{\boldsymbol{y}}(\boldsymbol{y},\boldsymbol{s})| &\leq \hat{C}_a \norm{\nabla \boldsymbol{y}}_{0,\Omega} \norm{\nabla \boldsymbol{s}}_{0,\Omega}, \label{Cty:ay} \\
	|b(\boldsymbol{v},q)| &\leq \sqrt{d} \norm{q}_{0,\Omega} \norm{\nabla\boldsymbol{v}}_{0,\Omega},\label{Cty:b} \\
	|d(\boldsymbol{y},\boldsymbol{v})| &\leq C_{F} \norm{\boldsymbol{y}}_{0,\Omega} \norm{\boldsymbol{v}}_{0,\Omega}. \label{Cty:d} 
\end{align}

Standard Sobolev embeddings indicate:
\begin{align}\label{H1embedding}
	\begin{aligned}
			&\mbox{For}\; r' \in [1,\infty) \; \mbox{if} \; d = 2 \;  \mbox{and} \; r' \in [1, 6] \;  \mbox{if} \; d = 3, \; \mbox{there exists} \; C_{r'_d} > 0 \\ &\mbox{such that} \; \norm{\boldsymbol{w}}_{L^{r'}(\Omega)} \leq C_{r'_d} \norm{\boldsymbol{w}}_{1,\Omega}, \; \mbox{for all} \; \boldsymbol{w} \in \boldsymbol{H}^1(\Omega).
		\end{aligned}
\end{align}

Then taking $\boldsymbol{u}, \boldsymbol{v}, \boldsymbol{w} \in \boldsymbol{H}^1(\Omega)$ and $\boldsymbol{y},\boldsymbol{s} \in [H^1(\Omega)]^2$, and applying the Sobolev embedding \eqref{H1embedding} along with H\"older's inequality, provide 
\begin{align}
	|c(\boldsymbol{w},\boldsymbol{u},\boldsymbol{v})| &\leq \norm{\boldsymbol{w}}_{L^6(\Omega)} \norm{\nabla \boldsymbol{u}}_{0,\Omega} \norm{\boldsymbol{v}}_{L^3(\Omega)} \leq C_{6_d}  C_{3_d} \norm{\boldsymbol{w}}_{1,\Omega} \norm{\nabla \boldsymbol{u}}_{0,\Omega} \norm{\boldsymbol{v}}_{1,\Omega}, \label{Cty:c}\\
	|c_{\boldsymbol{y}}(\boldsymbol{w},\boldsymbol{y},\boldsymbol{s})| &\leq \norm{\boldsymbol{w}}_{L^6(\Omega)} \norm{\nabla \boldsymbol{y}}_{0,\Omega} \norm{\boldsymbol{s}}_{L^3(\Omega)}\nonumber \\
	&\leq C_{6_d} \norm{\boldsymbol{w}}_{1,\Omega}  \norm{\nabla \boldsymbol{y}}_{0,\Omega} \norm{\boldsymbol{s}}_{L^3(\Omega)} \leq C_{6_d}C_{3_d} \norm{\boldsymbol{w}}_{1,\Omega}  \norm{\nabla \boldsymbol{y}}_{0,\Omega} \norm{\boldsymbol{s}}_{1, \Omega}.\label{Cty:cy}  
\end{align}

The following  fractional form of the Gagliardo-Nirenberg  inequality (see (2.6) in \cite{kumar2021large}, \cite{FractionalGagliardoNirenberg}) is used in the sequel:
Fix $1 \leq r_2, r_1 \leq \infty $ and $d \in \mathbb{N}. $ Let $\theta \in \mathbb{R} $ and $l \in \mathbb{R}^+ $ such that
\begin{align*}
	\begin{aligned}
		\frac{1}{p_1} = \frac{l}{d} + \theta \left(\frac{1}{r_1} - \frac{k}{d}\right) + \frac{1-\theta}{r_2}, \  \frac{l}{k} \leq \theta \leq 1, 
	\end{aligned}
\end{align*}
then we have
\begin{align}\label{FractionalGagliardoNirenberg}
	\begin{aligned}
		\|D^{l} \boldsymbol{u}\|_{L^{p_1}(\Omega)} \leq C_{gn} \|D^{k} \boldsymbol{u}\|^{\theta}_{L^{r_1}(\Omega)} \norm{\boldsymbol{u}}_{L^{r_2}(\Omega)}^{1-\theta} \; \forall \; \boldsymbol{u} \in \boldsymbol{W}^{k,r_1}(\Omega).
	\end{aligned}
\end{align}
For all $0 < \delta < 1/2$ and $\theta = 1$, the above inequality gives 
\begin{align}
	\begin{aligned}
		\mbox{for} \; d = 2,\;  \norm{\nabla \boldsymbol{u}}_{
			L^{4/(1-2\delta)}(\Omega)}  &\leq C_{gn} \norm{\boldsymbol{u}}_{3/2 + \delta,\Omega}, \label{GN}\\
		\mbox{for} \; d = 3,\;  \norm{\nabla \boldsymbol{u}}_{L^{3/(1-\delta)}(\Omega)}  &\leq C_{gn} \norm{\boldsymbol{u}}_{3/2 + \delta,\Omega}. 
	\end{aligned}
\end{align}

Using \eqref{GN}, one can get the following bound for $\frac{1}{p_1} + \frac{1}{p_2} = \frac{1}{2}$ ($p_1 = \frac{4}{1-2\delta}, 2<p_2=\frac{4}{1+2\delta}<4$ for $d=2$ and $p_1 = \frac{3}{1-\delta}, 3<p_2=\frac{6}{1+2\delta}<6$ for $d=3$),
\begin{align}
	|a(\boldsymbol{y}_a;\boldsymbol{u},\boldsymbol{u}) - a(\boldsymbol{y}_b;\boldsymbol{u},\boldsymbol{u})| &= |((\nu(\boldsymbol{y}_a)-\nu(\boldsymbol{y}_b)) \nabla \boldsymbol{u}, \nabla \boldsymbol{u})| \nonumber\\
	&\leq \gamma_{\nu} \norm{\boldsymbol{y}_a-\boldsymbol{y}_b}_{L^{p_2}(\Omega)} \norm{\nabla \boldsymbol{u}}_{L^{p_1}(\Omega)} \norm{\nabla \boldsymbol{u}}_{0,\Omega} \nonumber\\
	&\leq \gamma_{\nu} C_{{p_2}_{d}} \norm{\boldsymbol{y}_a-\boldsymbol{y}_b}_{1,\Omega} \norm{\nabla \boldsymbol{u}}_{L^{p_1}(\Omega)} \norm{\nabla \boldsymbol{u}}_{0,\Omega} \nonumber\\
	&\leq \gamma_{\nu} C_{{p_2}_{d}} C_{gn} \norm{\boldsymbol{y}_a-\boldsymbol{y}_b}_{1,\Omega} \norm{ \boldsymbol{u}}_{3/2 + \delta,\Omega} \norm{\nabla \boldsymbol{u}}_{0,\Omega}. \label{Cty:a1ma2}
\end{align}
Next, the  bilinear forms $a(\cdot;\cdot,\cdot)$ (for a fixed temperature) and $a_{\boldsymbol{y}}(\cdot,\cdot)$ are coercive, that is,
\begin{align}
	a(\cdot;\boldsymbol{v},\boldsymbol{v}) &\geq \min \left\{\alpha_1,\nu_1\right\} \left(\norm{\boldsymbol{v}}^2_{0,\Omega} + \norm{\nabla \boldsymbol{v}}^2_{0,\Omega}\right)
	=\alpha_a \norm{\boldsymbol{v}}^2_{1,\Omega} \;\; \forall \;\; \boldsymbol{v} \in \boldsymbol{H}_0^1(\Omega),\label{Coer:a}\\
	a_{\boldsymbol{y}}(\boldsymbol{y},\boldsymbol{y}) &\geq \alpha_2 \norm{\nabla \boldsymbol{y}}^2_{0,\Omega} \;\; \forall \;\; \boldsymbol{y} \in [H^1(\Omega)]^2. \label{Coer:ay}
\end{align}
Let $\boldsymbol{X}$ denote the kernel of $b(\cdot,\cdot)$, namely,
$$\boldsymbol{X} := \left\{\boldsymbol{v} \in \boldsymbol{H}_0^1(\Omega): b(\boldsymbol{v},q) =  0\; \forall \; q \in L_0^2(\Omega)\right\} = \left\{\boldsymbol{v} \in \boldsymbol{H}_0^1(\Omega) : \bdiv\hspace{0.04cm}\boldsymbol{v} = 0 \;\; \mbox{in} \;\; \Omega\right\}.$$
For all $\boldsymbol{w} \in \boldsymbol{X}$ and $\boldsymbol{v} \in \boldsymbol{H}_0^1(\Omega)$ an application of the integration by parts leads to
\begin{align}
	c(\boldsymbol{w},\boldsymbol{v},\boldsymbol{v}) &:= \sum_{i,j = 1}^{d} \int_{\Omega} w_i \frac{\partial v_j}{\partial x_i} v_j \; dx = \frac{1}{2} \sum_{i,j=1}^{d} \int_{\Omega} w_i \frac{\partial (v_j)^2}{\partial x_i} \; dx \nonumber\\
	&= \frac{1}{2} \sum_{j=1}^{d} \left[ - \int_{\Omega} \bdiv\hspace{0.04cm}\boldsymbol{w} (v_j)^2 \; dx + \int_{\Gamma} \boldsymbol{w}\cdot\boldsymbol{n}_{\Gamma} (v_j)^2 \; ds\right] = 0. \label{Prop:c1}
\end{align}
Replacing $\boldsymbol{v}$ by $\boldsymbol{u} + \boldsymbol{v}$ and using \eqref{Prop:c1} leads to
\begin{eqnarray}
	c(\boldsymbol{w}, \boldsymbol{u}, \boldsymbol{v}) = -c(\boldsymbol{w}, \boldsymbol{v}, \boldsymbol{u}). \label{Prop:c2}
\end{eqnarray}
Moreover, 
\begin{eqnarray}\label{Prop:c3}
	c(\boldsymbol{u}, \boldsymbol{v}, \boldsymbol{w}) = ((\nabla \boldsymbol{v})^\top \boldsymbol{w}, \boldsymbol{u}),
\end{eqnarray}
for all $\boldsymbol{w}\in \boldsymbol{X},\boldsymbol{u}, \boldsymbol{v}\in \boldsymbol{H}_0^1(\Omega)$.
For all $\boldsymbol{s} \in [H_0^1(\Omega)]^2$, results analogous to \eqref{Prop:c1}, \eqref{Prop:c2} and \eqref{Prop:c3} hold for $c_{\boldsymbol{y}}(\boldsymbol{w}, \boldsymbol{s},\boldsymbol{s})$. 
\section{Well-posedness and regularity}\label{Sec:Well-posedness}
\subsection{Existence of solutions}\label{Sec:StateExistence}
We first state the ``de Rham Theorem" (Theorem IV.2.4, \cite{boyer2012mathematical} and \cite{temam2001navier}), which will be used to prove an equivalence result required in the well-posedness analysis of the governing equation.
\begin{theorem}\label{deRham}
	Let $\Omega$ be a connected, bounded, Lipschitz domain in $\mathbb{R}^d$. Let $f$ be an element in $(H^{-1}(\Omega))^d$, such that for any function $\phi \in (C^{\infty}_0(\Omega))^d$ satisfying $\bdiv\hspace{0.04cm}\phi = 0$, we have $\langle f, \phi \rangle$ $= 0$. Then, there exists a unique function $p \in L_0^2(\Omega)$ such that $f = \nabla p$. 
\end{theorem}
\begin{lemma}\label{State:equivalence}
	If $(\boldsymbol{u},p,\boldsymbol{y}) \in \boldsymbol{H}_0^1(\Omega) \times L_0^2(\Omega) \times [H^1(\Omega)]^2$ solves \eqref{P:S}, then $(\boldsymbol{u},\boldsymbol{y}) \in \boldsymbol{X} \times [H^1(\Omega)]^2$ satisfies $\boldsymbol{y}|_{\Gamma} = \boldsymbol{y}^D$ and
	\begin{align}\label{P:Sred}
		\left\{
		\begin{aligned}
			a(\boldsymbol{y};\boldsymbol{u},\boldsymbol{v}) + c(\boldsymbol{u}, \boldsymbol{u}, \boldsymbol{v}) - d(\boldsymbol{y}, \boldsymbol{v}) &= 0 \;\; \forall \;\; \boldsymbol{v} \in \boldsymbol{X},\\
			a_{\boldsymbol{y}}(\boldsymbol{y},\boldsymbol{s}) + c_{\boldsymbol{y}}(\boldsymbol{u}, \boldsymbol{y}, \boldsymbol{s}) &= 0 \;\; \forall \;\; \boldsymbol{s} \in [H_0^1(\Omega)]^2.
		\end{aligned}
		\right.
	\end{align}
	Conversely, if $(\boldsymbol{u},\boldsymbol{y}) \in \boldsymbol{X} \times [H^1(\Omega)]^2$ is a solution of the reduced problem, then there exists a $p \in L_0^2(\Omega)$ such that $(\boldsymbol{u},p,\boldsymbol{y})$ is a solution of \eqref{P:S}.
\end{lemma}
\begin{proof}
	One way is trivial using the definition of $\boldsymbol{X}$. Conversely, using Theorem \ref{deRham}, there exists  a unique function $p$ such that $\boldsymbol{f} = \nabla p$. In other words, there exists a unique $p$ such that $\langle \boldsymbol{f},\boldsymbol{v} \rangle =  \langle \nabla p, \boldsymbol{v} \rangle = - (p, \bdiv\hspace{0.04cm}\boldsymbol{v})$ for all $\boldsymbol{v} \in \boldsymbol{H}_0^1(\Omega)$. Now, choosing $\langle \boldsymbol{f},\boldsymbol{v} \rangle = -a(\boldsymbol{y};\boldsymbol{u},\boldsymbol{v}) - c(\boldsymbol{u}, \boldsymbol{u}, \boldsymbol{v}) + d(\boldsymbol{y}, \boldsymbol{v})$ and applying the above consequence leads to \eqref{P:S}.    
\end{proof}

\vspace{0.25cm}
A lifting argument is used to deal with the non-homogeneous Dirichlet data appearing in the advection-diffusion equations. We write $\boldsymbol{y}$ as $\boldsymbol{y} = \boldsymbol{y}_0 + \boldsymbol{y}_1$, where $\boldsymbol{y}_0 \in [H^1_0(\Omega)]^2$ and $\boldsymbol{y}_1$ is such that
\begin{eqnarray}\label{lifting}
	\boldsymbol{y}_1 \in [H^1(\Omega)]^2 \;\; \mbox{with} \;\; \boldsymbol{y}_1|_{\Gamma} = \boldsymbol{y}^D.
\end{eqnarray}
We need the following intermediate result to prove the energy estimates (see \cite{lorca1996stationary, burger2019h}).
\begin{lemma}\label{IntermediateLemma}
	If $\boldsymbol{y}^D \in [H^{1/2}(\Gamma)]^2$, then for any $\epsilon > 0$ 	and $1 \leq r \leq 6$ if $d = 3$ and $1 \leq r < \infty$ if $d = 2$, there exists an extension $\boldsymbol{y}_1 \in [H^1(\Omega)]^2$ of $\boldsymbol{y}^D$ with $\norm{\boldsymbol{y}_1}_{L^r(\Omega)} < \epsilon$.
\end{lemma}

\begin{lemma}\label{EnergyEst:State}
	Let $(\boldsymbol{u},\boldsymbol{y})$ be a solution to the reduced problem \eqref{P:Sred}. Then, there exists positive constants $C_{\boldsymbol{u}}$ and $C_{\boldsymbol{y}}$ such that 
	\begin{align}
		\norm{\boldsymbol{u}}_{1,\Omega} &\leq C_{\boldsymbol{u}} \norm{\boldsymbol{y}^D}_{1/2,\Gamma} := M_{\boldsymbol{u}}, \label{M_u}\\
		\norm{\nabla \boldsymbol{y}_0}_{0,\Omega} &\leq C_{\boldsymbol{y}} \norm{\boldsymbol{y}^D}_{1/2,\Gamma}, \nonumber\\
		\norm{\boldsymbol{y}}_{1,\Omega} &\leq C_{\boldsymbol{y}} \norm{\boldsymbol{y}^D}_{1/2,\Gamma} := M_{\boldsymbol{y}}. \label{M_y}
	\end{align}
\end{lemma}
\begin{proof}
	In \eqref{P:Sred}, take $\boldsymbol{v} = \boldsymbol{u}$ and $\boldsymbol{s} = \boldsymbol{y}_0$,
	\begin{align}
		a(\boldsymbol{y}, \boldsymbol{u}, \boldsymbol{u}) + c(\boldsymbol{u}, \boldsymbol{u}, \boldsymbol{u}) - d(\boldsymbol{y}, \boldsymbol{u}) &= 0, \label{EnergyEstState:eq1}\\
		a_{\boldsymbol{y}}(\boldsymbol{y}_0+\boldsymbol{y}_1, \boldsymbol{y}_0) + c_{\boldsymbol{y}}(\boldsymbol{u}, \boldsymbol{y}_0+\boldsymbol{y}_1, \boldsymbol{y}_0) &= 0. \label{EnergyEstState:eq2}
	\end{align} 
	Using the coercivity properties \eqref{Coer:a}, \eqref{Coer:ay},  properties of the trilinear form and the boundedness  estimates \eqref{Cty:a}-\eqref{Cty:d} in \eqref{EnergyEstState:eq1} and \eqref{EnergyEstState:eq2}, we get
	\begin{align}
		\alpha_a \norm{\boldsymbol{u}}_{1,\Omega} &\leq C_F  \left(\norm{\boldsymbol{y}_0}_{0,\Omega} + \norm{\boldsymbol{y}_1}_{0,\Omega}\right) \label{EnergyEstState:eq3}\\
		&\leq  C_F \left( C_P \norm{\nabla \boldsymbol{y}_0}_{0,\Omega} + \norm{\boldsymbol{y}_1}_{0,\Omega}\right), \nonumber	\\
		\alpha_2 \norm{\nabla \boldsymbol{y}_0}_{0,\Omega} &\leq C_{6_d}\norm{\boldsymbol{u}}_{1,\Omega} \norm{\boldsymbol{y}_1}_{L^3(\Omega)} + \hat{C}_a \norm{\nabla\boldsymbol{y}_1}_{0,\Omega}. \label{EnergyEstState:eq4}
	\end{align}
	Putting \eqref{EnergyEstState:eq4} in \eqref{EnergyEstState:eq3}, we  derive 
	\begin{eqnarray*}
		\norm{\boldsymbol{u}}_{1,\Omega} \left(1 - \frac{C_F C_P C_{6_d}}{\alpha_a \alpha_2} \norm{\boldsymbol{y}_1}_{L^3(\Omega)}\right) \leq \frac{C_F }{\alpha_a}\max\left\{\frac{ \hat{C}_a C_P}{ \alpha_2},1 \right\}\norm{\boldsymbol{y}_1}_{1,\Omega}.
	\end{eqnarray*}
	Now choosing $\frac{C_F C_P C_{6_d}}{\alpha_a \alpha_2} \norm{\boldsymbol{y}_1}_{L^3(\Omega)} < C_{\epsilon} < 1$ and applying Lemma \ref{IntermediateLemma}, we obtain 
	\begin{align}
		\norm{\boldsymbol{u}}_{1,\Omega} &\leq \frac{C_F}{\alpha_a (1-C_{\epsilon})}\max\left\{\frac{ \hat{C}_a C_P}{ \alpha_2},1 \right\} \norm{\boldsymbol{y}_1}_{1,\Omega}  \nonumber\\
		&\leq \max\left\{\frac{C_F}{\alpha_a (1-C_{\epsilon})}\max\left\{\frac{ \hat{C}_a C_P}{ \alpha_2},1 \right\} \right\} \norm{\boldsymbol{y}_1}_{1,\Omega} \nonumber\\
		&\leq C_{\boldsymbol{u}} \norm{\boldsymbol{y}^D}_{1/2, \Gamma}. \nonumber
	\end{align}\
	Putting the bounds of $\boldsymbol{u}$ in \eqref{EnergyEstState:eq4} and simplifying, we deduce the following:
	\begin{align*}
		\norm{\nabla \boldsymbol{y}_0}_{0,\Omega} &\leq \max \left\{\frac{1}{\alpha_2}\left(\frac{1}{1-C_{\epsilon}}\right), \frac{C_{\epsilon}}{C_P \alpha_2}\right\} \norm{\boldsymbol{y}_1}_{1,\Omega} \leq C_{\boldsymbol{y}}\norm{\boldsymbol{y}^D}_{1/2,\Gamma},
	\end{align*}
	which completes the proof. 
\end{proof}

\begin{definition}[Weak solution]
	The pair  of functions $(\boldsymbol{u}, \boldsymbol{y}) \in \boldsymbol{X} \times \left[H^1(\Omega)\right]^2$ is called \emph{weak solution} to the system \eqref{P:Sred}, if for $\boldsymbol{y}^D \in \left[H^{1/2}(\Gamma)\right]^2$ and $(\boldsymbol{v}, \boldsymbol{s}) \in \boldsymbol{X} \times \left[H_0^1(\Omega)\right]^2$, $(\boldsymbol{u}, \boldsymbol{y})$ satisfies \eqref{P:Sred} with $\boldsymbol{y}|_{\Gamma} = \boldsymbol{y}^D.$
\end{definition}

Now we prove the existence of a weak solution. 
\begin{theorem}\label{StateExistence}
	Let the assumptions in Section \ref{Assum} on the governing equation hold. For every $\boldsymbol{y}^D \in [H^{1/2}(\Gamma)]^2,$ there exists a lifting $\boldsymbol{y}_1 \in [H^1(\Omega)]^2$ of $\boldsymbol{y}^D$ satisfying \eqref{lifting} such that the problem \eqref{P:Sred} has a \emph{weak solution} $(\boldsymbol{u},\boldsymbol{y}=\boldsymbol{y}_0 + \boldsymbol{y}_1) \in \boldsymbol{X} \times [H^1(\Omega)]^2$.
\end{theorem}
\begin{proof}
	Let $\boldsymbol{v}^k = \boldsymbol{v}^k(\cdot),$ $k\in\mathbb{Z}^+$ (resp. $\boldsymbol{s}^k = \boldsymbol{s}^k(\cdot)$) be smooth functions and  the set $\left\{\boldsymbol{v}^k\right\}_{k=1}^{\infty}$ (resp.  $\left\{\boldsymbol{s}^k\right\}_{k=1}^{\infty}$) be an orthogonal basis of $\boldsymbol{H}_0^1(\Omega)$ (resp. $[H_0^1(\Omega)]^2$) and orthonormal basis of $\boldsymbol{L}^2(\Omega)$ (resp. $[L^2(\Omega)]^2$). Let us also consider the finite-dimensional spaces $\boldsymbol{X}_n = \mbox{span}\left\{\boldsymbol{v}^1, \ldots , \boldsymbol{v}^n \right\}$ and $\boldsymbol{Y}_n = \mbox{span}\left\{\boldsymbol{s}^1, \ldots , \boldsymbol{s}^n\right\},$ for a fixed $n \in \mathbb{Z}^+$ with the following $n^{\mathrm{th}}$ approximate solution of \eqref{P:Sred},
	$$ \boldsymbol{u}^n(x) = \sum_{k=1}^{n} c_{n,k} \; \boldsymbol{v}^k(x) \;\;\; \mbox{and} \;\;\; \boldsymbol{y}_0^n(x) = \sum_{k=1}^{n} d_{n,k} \; \boldsymbol{s}^k(x),$$ 
	such that
	\begin{eqnarray}
		a(\boldsymbol{y}_0^n+\boldsymbol{y}_{1_n}; \boldsymbol{u}^n, \boldsymbol{v}^j) + c(\boldsymbol{u}^n,\boldsymbol{u}^n,\boldsymbol{v}^j) - d(\boldsymbol{y}_0^n+\boldsymbol{y}_{1_n}, \boldsymbol{v}^j) = 0, \label{FGStateEq1} \\
		a_{\boldsymbol{y}}(\boldsymbol{y}_0^n+\boldsymbol{y}_{1_n},\boldsymbol{s}^j) + c_{\boldsymbol{y}}(\boldsymbol{u}^n,\boldsymbol{y}_0^n+\boldsymbol{y}_{1_n},\boldsymbol{s}^j) = 0, \label{FGStateEq2}
	\end{eqnarray}	
	for $1 \leq j \leq n$, where ${\boldsymbol{y}_1}_n=\boldsymbol{\pi}_n\boldsymbol{y}_1:=\sum\limits_{k=1}^n(\boldsymbol{y}_1,\boldsymbol{s}^k)\boldsymbol{s}^k$. Moreover, 
	\begin{align}\label{y1_proj}
		\|{\boldsymbol{y}_1}_n\|_{1,\Omega}\leq \|{\boldsymbol{y}_1}\|_{1,\Omega} \;\; \mbox{and} \;\; \|{\boldsymbol{y}_1}_n-{\boldsymbol{y}_1}\|_{1,\Omega}\to 0 \; \mbox{as} \; n \to\infty.
	\end{align}
	The above equations denote a system of $2n$ nonlinear equations with unknowns $c_{n,k}$ and $d_{n,k}$. We denote the scalar product associated with $\boldsymbol{X}_n$ and $\boldsymbol{Y}_n$ by $\left[\cdot, \cdot\right] := \left(\nabla \cdot, \nabla \cdot\right)$. Moreover, the notation
	$\left[\cdot\right]$ denotes the norm on $\boldsymbol{X}_n$ and $\boldsymbol{Y}_n$, which is induced by $\boldsymbol{H}_0^1(\Omega)$ and $[H_0^1(\Omega)]^2$, respectively. Now we define the map $\boldsymbol{P}_n := \boldsymbol{P}_n^X \times \boldsymbol{P}_n^Y$ such that
	\begin{align}
		\left[\boldsymbol{P}_n^X(\boldsymbol{u}), \boldsymbol{v}\right] &= \left(\nabla \boldsymbol{P}_n^X(\boldsymbol{u}), \nabla \boldsymbol{v}\right)
		= a(\boldsymbol{y}; \boldsymbol{u},\boldsymbol{v}) + c(\boldsymbol{u}, \boldsymbol{u}, \boldsymbol{v}) - d(\boldsymbol{y}, \boldsymbol{v}), \nonumber\\
		\left[\boldsymbol{P}_n^Y(\boldsymbol{y}_0), \boldsymbol{s}\right] &= \left(\nabla \boldsymbol{P}_n^Y(\boldsymbol{y}_0), \nabla \boldsymbol{s}\right) 
		= a_{\boldsymbol{y}}(\boldsymbol{y}_0, \boldsymbol{s}) + c_{\boldsymbol{y}}(\boldsymbol{u}, \boldsymbol{y}_0, \boldsymbol{s}) - a_{\boldsymbol{y}}(\boldsymbol{y}_1, \boldsymbol{s}) - c_{\boldsymbol{y}}(\boldsymbol{u}, \boldsymbol{y}_1, \boldsymbol{s}). \nonumber
	\end{align}
	The norm on $\boldsymbol{X}_n \times \boldsymbol{Y}_n$ is defined by $\norm{(\boldsymbol{u}, \boldsymbol{y}_0)}_{1,\Omega} = \norm{\nabla\boldsymbol{u}}_{0,\Omega} + \norm{\nabla\boldsymbol{y}_0}_{0,\Omega}$. Now we show the boundedness  of $\boldsymbol{P}_n$ using the boundedness estimates \eqref{Cty:a}-\eqref{Cty:d}, Sobolev's embedding \eqref{H1embedding} and the consequence of Poincar\'e's inequality as follows:
	\begin{align*}
		\left[\boldsymbol{P}_n^X(\boldsymbol{u}), \boldsymbol{v}\right] &\leq  \left(C_a \norm{\boldsymbol{u}}_{1,\Omega} + C_{6_d} C_{3_d} \norm{\boldsymbol{u}}_{1,\Omega} \norm{\nabla \boldsymbol{u}}_{0,\Omega} \right.\nonumber + C_F \left( \norm{\boldsymbol{y}_0}_{0,\Omega}+\norm{\boldsymbol{y}_1}_{0,\Omega} \right) \left. \right) \norm{\nabla \boldsymbol{v}}_{0,\Omega}, \\	\left[\boldsymbol{P}_n^Y(\boldsymbol{y}_0), \boldsymbol{s}\right] &\leq \left(\hat{C}_a \norm{\nabla \boldsymbol{y}_0}_{0,\Omega} + C_{6_d} C_{3_d} \norm{\boldsymbol{u}}_{1,\Omega} \norm{\nabla \boldsymbol{y}_0}_{0,\Omega} 
		+ \hat{C}_a \norm{\nabla \boldsymbol{y}_1}_{0,\Omega} \right. \\& \left.\quad + C_{6_d} C_{3_d} \norm{\boldsymbol{u}}_{1,\Omega} \norm{\nabla \boldsymbol{y}_1}_{0,\Omega}\right) \norm{\nabla \boldsymbol{s}}_{0,\Omega} . \nonumber
	\end{align*}
	Thus, for $\boldsymbol{\tilde{z}}=(\boldsymbol{u},\boldsymbol{y}_0)$ and $\boldsymbol{\tilde{w}}=(\boldsymbol{v},\boldsymbol{s}),$ the uniform boundedness of $\boldsymbol{P}_n$ follows, since 
	\begin{align*}
		|\left(\boldsymbol{P}_n(\boldsymbol{\tilde{z}}), \boldsymbol{\tilde{w}}\right)| &\leq \left(C_a \norm{\boldsymbol{u}}_{1,\Omega} + C_{6_d} C_{3_d} \norm{\boldsymbol{u}}_{1,\Omega} \norm{\nabla \boldsymbol{u}}_{0,\Omega} \right.\nonumber + C_F \left( \norm{\boldsymbol{y}_0}_{0,\Omega} +\norm{\boldsymbol{y}_1}_{0,\Omega}\right) \nonumber\\&\quad\left. + \hat{C}_a \norm{\nabla \boldsymbol{y}_0}_{0,\Omega} + C_{6_d} C_{3_d} \norm{\boldsymbol{u}}_{1,\Omega} \norm{\nabla \boldsymbol{y}_0}_{0,\Omega}  \right. \nonumber\\&\quad\left.+ \hat{C}_a \norm{\nabla \boldsymbol{y}_1}_{0,\Omega} + C_{6_d} C_{3_d} \norm{\boldsymbol{u}}_{1,\Omega} \norm{\nabla \boldsymbol{y}_1}_{0,\Omega}\right) (\norm{\nabla \boldsymbol{s}}_{0,\Omega} + \norm{\nabla \boldsymbol{v}}_{0,\Omega}).
	\end{align*}
	We now prove the coercivity of the map by using Lemma \ref{EnergyEst:State} as follows:
	\begin{align*}
		\left[\boldsymbol{P}_n^X(\boldsymbol{u}), \boldsymbol{u}\right] &\geq \alpha_a \norm{\boldsymbol{u}}^2_{1,\Omega} - C_F \norm{\boldsymbol{y}}_{1,\Omega} \norm{\boldsymbol{u}}_{1,\Omega} \nonumber\\
		&\geq \alpha_a \norm{\boldsymbol{u}}^2_{1,\Omega} - C_F M_{\boldsymbol{y}} M_{\boldsymbol{u}}.
	\end{align*}
	It follows that $\left[\boldsymbol{P}_n^X(\boldsymbol{u}), \boldsymbol{u}\right] > 0$ for  $\left[\boldsymbol{u}\right] = \norm{\boldsymbol{u}}^2_{1,\Omega} = \kappa_1$ , and $\kappa_1$ sufficiently large: more precisely,
	$$\kappa_1 > \left\{\frac{1}{\alpha_a} C_F M_{\boldsymbol{y}} M_{\boldsymbol{u}} \right\}.$$ Similarly, we have
	\begin{align*}
		\left[\boldsymbol{P}_n^Y(\boldsymbol{y}_0), \boldsymbol{y}_0\right] &\geq \alpha_2 \norm{\nabla \boldsymbol{y}_0}^2_{0,\Omega} - \hat{C}_a \norm{\boldsymbol{y}_1}_{1,\Omega} \norm{\nabla \boldsymbol{y}_0}_{0,\Omega} - \hat{C}_v \norm{\boldsymbol{u}}_{1,\Omega} \norm{\nabla \boldsymbol{y}_1}_{1,\Omega} \norm{\boldsymbol{y}_0}_{1,\Omega} \\
		&\geq \alpha_2 \norm{\nabla \boldsymbol{y}_0}^2_{0,\Omega} - \hat{C}_a \norm{\boldsymbol{y}^D}_{1/2,\Gamma} C_{\boldsymbol{y}} \norm{\boldsymbol{y}^D}_{1/2,\Gamma} - C_{6_d} C_{3_d} M_{\boldsymbol{u}} \norm{\boldsymbol{y}^D}_{1/2,\Gamma} \norm{\nabla \boldsymbol{y}_0}_{0,\Omega},
	\end{align*}
	which implies that $[\boldsymbol{P}_n^Y(\boldsymbol{y}_0), \boldsymbol{y}_0] > 0$ for $\left[\boldsymbol{y}_0\right] = \norm{\nabla \boldsymbol{y}_0}^2_{0,\Omega} = \kappa_2$, and $\kappa_2$ sufficiently large: more precisely
	$$\kappa_2 > \left\{\frac{1}{\alpha_2} \left(\norm{\boldsymbol{y}^D}_{1/2,\Gamma} C_{\boldsymbol{y}} \norm{\boldsymbol{y}^D}_{1/2,\Gamma}  \left(C_{6_d} C_{3_d} M_{\boldsymbol{u}} + \hat{C}_a\right)\right)\right\}.$$ 
	
	From the coercivity of $\boldsymbol{P}_n^X$ and $\boldsymbol{P}_n^Y$ and the definition of $\boldsymbol{P}_n$, we get that $\left[\boldsymbol{P}_n(\tilde{\boldsymbol{w}}), \tilde{\boldsymbol{w}}\right] > 0$ for sufficiently large  $\left[\tilde{\boldsymbol{w}}\right] = \kappa_1 + \kappa_2 > 0$. Then there exists a solution $(\boldsymbol{u}^n, \boldsymbol{y}_0^n + {\boldsymbol{y}_1}_n)$ of \eqref{FGStateEq1}-\eqref{FGStateEq2} such that $\boldsymbol{P}_n(\boldsymbol{u}^n, \boldsymbol{y}_0^n + {\boldsymbol{y}_1}_n) = 0$ (see Lemma 1.4 of Chapter II in \cite{temam2001navier}). 
	
	Now in what follows we show the passage to the limit. Multiply \eqref{FGStateEq1} and \eqref{FGStateEq2} by $c_{n,k}$ and $d_{n,k}$, respectively and sum over $1 \leq k \leq n$ to obtain 
	\begin{eqnarray*}
		a(\boldsymbol{y}_0^n+\boldsymbol{y}_{1_n}; \boldsymbol{u}^n, \boldsymbol{u}^n) - d(\boldsymbol{y}_0^n+\boldsymbol{y}_{1_n}, \boldsymbol{u}^n) = 0, \\
		a_{\boldsymbol{y}}(\boldsymbol{y}_0^n+\boldsymbol{y}_{1_n},\boldsymbol{y}_0^n) + c_{\boldsymbol{y}}(\boldsymbol{u}^n,\boldsymbol{y}_0^n+\boldsymbol{y}_{1_n},\boldsymbol{y}_0^n) = 0.
	\end{eqnarray*}		
	Following the same steps as in the proof of Lemma \ref{EnergyEst:State}, we can get analogous uniform bounds on $\norm{\boldsymbol{u}^n}_{1,\Omega}$ and $\norm{\nabla \boldsymbol{y}_0^n}_{0,\Omega}$. Thus, using the Banach-Alaoglu Theorem (see for reference \cite[Theorem A.51]{leoni2017first}), we can extract subsequences $(\boldsymbol{u}^{n_k})_{k\in\mathbb{Z}^+}$ and $(\boldsymbol{y}_0^{n_k})_{k\in\mathbb{Z}^+}$ such that
	\begin{align}\label{weak convergence}
		\boldsymbol{u}^{n_k} \rightharpoonup \boldsymbol{u} \;\; \mbox{ in } \;\; \boldsymbol{H}_0^1(\Omega)   \;\; \mbox{ and } \;\;  \boldsymbol{y}_0^{n_k} \rightharpoonup \boldsymbol{y} \;\; \mbox{ in } \;\; [{H}_0^1(\Omega)]^2 \;\; \mbox{ as } \;\; k \rightarrow \infty.
	\end{align}
	Moreover, since $\boldsymbol{H}_0^1(\Omega)$ is compactly embedded in $\boldsymbol{L}^2(\Omega)$, we can extract subsequences $(\boldsymbol{u}^{n_{k_j}})_{j\in\mathbb{Z}^+}$ and $(\boldsymbol{y}_0^{n_{k_j}})_{j\in\mathbb{Z}^+}$ such that
	\begin{align}\label{strong convergence}
		\boldsymbol{u}^{n_{k_j}} \longrightarrow \boldsymbol{u} \;\; \mbox{ in } \;\; \boldsymbol{L}^2(\Omega) \mbox{ and } \;\; \boldsymbol{y}_0^{n_{k_j}} \longrightarrow \boldsymbol{y}_0 \;\; \mbox{ in } \;\; [{L}^2(\Omega)]^2 \;\; \mbox{ as } \;\; j \rightarrow \infty,
	\end{align}
	and a.e. convergence along a further subsequence. 
	Now using the above convergences, we first show that $c(\boldsymbol{u}^{n_{k_j}}, \boldsymbol{u}^{n_{k_j}}, \boldsymbol{v}) \longrightarrow c(\boldsymbol{u}, \boldsymbol{u}, \boldsymbol{v})$ for all $\boldsymbol{v}\in \boldsymbol{H}_0^1(\Omega)$ as follows:
	\begin{align*}
		&	|c(\boldsymbol{u}^{n_{k_j}}, \boldsymbol{u}^{n_{k_j}}, \boldsymbol{v}) - c(\boldsymbol{u}, \boldsymbol{u}, \boldsymbol{v})| \nonumber\\&= |c(\boldsymbol{u}^{n_{k_j}}, \boldsymbol{v}, \boldsymbol{u}^{n_{k_j}} - \boldsymbol{u}) - c(\boldsymbol{u}^{n_{k_j}} - \boldsymbol{u}, \boldsymbol{v},\boldsymbol{u})|  \nonumber\\
		&\leq  \norm{\boldsymbol{u}^{n_{k_j}}}_{L^4(\Omega)} \norm{\nabla \boldsymbol{v}}_{0,\Omega} \norm{\boldsymbol{u}^{n_{k_j}}-\boldsymbol{u}}_{L^4(\Omega)} + \norm{\boldsymbol{u}^{n_{k_j}}-\boldsymbol{u}}_{L^4(\Omega)}  \norm{\nabla \boldsymbol{v}}_{0,\Omega} \norm{\boldsymbol{u}}_{L^4(\Omega)} \nonumber\\
		&\leq \begin{cases}
			2^{1/4} \; C_{4_d} \norm{\boldsymbol{u}^{n_{k_j}}}_{1,\Omega} \norm{\nabla \boldsymbol{v}}_{0,\Omega} \left(\norm{\boldsymbol{u}^{n_{k_j}} - \boldsymbol{u}}^{1/2}_{0,\Omega} \norm{\nabla (\boldsymbol{u}^{n_{k_j}} - \boldsymbol{u})}^{1/2}_{0,\Omega} \right)\\
			~~~~~+ 2^{1/4} \; C_{4_d} \norm{\boldsymbol{u}}_{1,\Omega} \norm{\nabla \boldsymbol{v}}_{0,\Omega} \left(\norm{\boldsymbol{u}^{n_{k_j}} - \boldsymbol{u}}^{1/2}_{0,\Omega} \norm{\nabla (\boldsymbol{u}^{n_{k_j}} - \boldsymbol{u})}^{1/2}_{0,\Omega} \right) \quad \mbox{when}\; d=2, \\
			2^{1/2} \; C_{4_d} \norm{\boldsymbol{u}^{n_{k_j}}}_{1,\Omega} \norm{\nabla \boldsymbol{v}}_{0,\Omega} \left(\norm{\boldsymbol{u}^{n_{k_j}} - \boldsymbol{u}}^{1/4}_{0,\Omega} \norm{\nabla (\boldsymbol{u}^{n_{k_j}} - \boldsymbol{u})}^{3/4}_{0,\Omega} \right)\\
			~~~~~+ 2^{1/2} \;C_{4_d} \norm{\boldsymbol{u}}_{1,\Omega} \norm{\nabla \boldsymbol{v}}_{0,\Omega} \left(\norm{\boldsymbol{u}^{n_{k_j}} - \boldsymbol{u}}^{1/4}_{0,\Omega} \norm{\nabla (\boldsymbol{u}^{n_{k_j}} - \boldsymbol{u})}^{3/4}_{0,\Omega} \right) \quad \mbox{when}\; d=3,
		\end{cases}\\ 		 
		&\longrightarrow 0 \;\; \mbox{ as } \;\; j \rightarrow \infty.
	\end{align*}
	Secondly, the following also holds for all  $\boldsymbol{v} \in \boldsymbol{C}_0^{\infty}(\Omega)$,
	\begin{align*}
		&|(\nu(\boldsymbol{y}_0^{n_{k_j}} + \boldsymbol{y}_{1_n}^{k_j}) \nabla \boldsymbol{u}^{n_{k_j}}, \nabla \boldsymbol{v}) - (\nabla \boldsymbol{u}, \nu(\boldsymbol{y}_0 + \boldsymbol{y}_1), \nabla \boldsymbol{v})| \\
		&\quad\leq  |(\nabla \boldsymbol{u}^{n_{k_j}}, (\nu(\boldsymbol{y}_0^{n_{k_j}} + \boldsymbol{y}_{1_n}^{k_j})- \nu (\boldsymbol{y}_0 + \boldsymbol{y}_1)) \nabla \boldsymbol{v})| + |(\nabla (\boldsymbol{u}^{n_{k_j}} - \boldsymbol{u}), \nu (\boldsymbol{y}_0 + \boldsymbol{y}_1) \nabla \boldsymbol{v})| \longrightarrow 0.
	\end{align*}

	Indeed, the second term goes to zero by the weak convergence given in \eqref{weak convergence} and the first term goes to zero since $\left(\nabla \boldsymbol{u}^{n_{k_j}}\right)_{j\in\mathbb{Z}^+}$ is uniformly bounded,  \eqref{y1_proj} and the following convergence:
	\begin{align*}
		\norm{\left(\nu (\boldsymbol{y}_0^{n_{k_j}} + \boldsymbol{y}_{1_n}^{k_j}) - \nu (\boldsymbol{y}_0 + \boldsymbol{y}_1) \right) \nabla \boldsymbol{v}}_{0,\Omega} &\leq \norm{\left(\nu (\boldsymbol{y}_0^{n_{k_j}} + \boldsymbol{y}_{1_n}^{k_j}) - \nu (\boldsymbol{y}_0 + \boldsymbol{y}_1) \right)}_{0,\Omega} \norm{\nabla \boldsymbol{v}}_{\infty,\Omega} \\
		&\leq \gamma_\nu \big(\|\boldsymbol{y}_0^{n_{k_j}} - \boldsymbol{y}_0\|_{0,\Omega} + \|\boldsymbol{y}_{1_n}^{k_j} - \boldsymbol{y}_1\|_{0,\Omega}\big) \norm{\nabla \boldsymbol{v}}_{\infty,\Omega} \longrightarrow 0.
	\end{align*}
	Similar logic holds for passing the limit to other terms of \eqref{FGStateEq1}-\eqref{FGStateEq2} through the subsequence $(n_{k_j})$ as $j \rightarrow \infty$.
\end{proof}

\subsection{Regularity of solutions}\label{Sec:RegularityState}
In this section, we discuss the regularity of solutions of the governing equation \eqref{P:GE} on a class of Lipschitz domains. 
Following section 2 in \cite{MoniqueDaugeRegularity},  we denote by $\mathcal{O}_2(\mathbb{R}^2)$ and $\mathcal{O}_3(\mathbb{R}^3)$ for some  class of domains in two and three dimensions, respectively. The domain  $\mathcal{O}_2(\mathbb{R}^2)$ consists of all curvilinear polygons, possibly with cracks but without turning points: $\Omega$ belongs to $\mathcal{O}_2(\mathbb{R}^2)$ if and only if it satisfies
\begin{itemize}
	\item[(i)] $\Omega$ is bounded and connected.
	\item[(ii)] The boundary of $\Omega$ consists of finite number of smooth closed arcs $\Gamma_1,\ldots , \Gamma_N,\ \Gamma_{N+1} = \Gamma_1$.
	\item[(iii)] $\Omega$ is locally diffeomorphic to a neighbourhood of zero in a plane sector $\Gamma_{A_j},$ where $A_j$ and $A_{j+1}$ are the ends of $\Gamma_j$. 
\end{itemize}
$\Omega$ belongs to $\mathcal{O}_3(\mathbb{R}^3)$ if and only if 
\begin{itemize}
	\item[(i)] $\Omega$ is bounded and connected.
	\item[(ii)] At each point $x$ of its ``stretched boundary'', $\Omega$ is locally diffeomorphic to a neighbourhood of zero in one of the following three kind of domains:
	\subitem (a) A half-space: then $x$ is a regular point.
	\subitem (b) A dihedron isomorphic to $\mathbb{R} \times \Gamma_x$, with $\Gamma_x$ a plane sector with an opening $\omega_x$ different from $\pi$: then $x$ belongs to an edge.
	\subitem (c) A cone $\Gamma_x$ with vertex zero, such that its intersection $G_x$ with $\mathcal{S}^2$ belongs to $\mathcal{O}_2(\mathcal{S}^2)$, the class of curvilinear polygons on unit sphere $\mathcal{S}^2$ of $\mathbb{R}^3$: then $x$ is a vertex.   
\end{itemize}

Now, invoking Theorem 5.5 and Theorem 9.20 in \cite{MoniqueDaugeRegularity}, we state  in Lemma \ref{RegularityStokes}, the regularity result proved for the following Stokes problem:
\begin{align}\label{StokesProblem}
	\left\{
	\begin{aligned}
		- \Delta \boldsymbol{u} + \nabla p &= \boldsymbol{\mathcal{F}} \;\;  \;\; \mbox{ in } \;\; \Omega,  \\ \bdiv\hspace{0.04cm}\boldsymbol{u} &= 0 \;\; \mbox{ in } \;\; \Omega, \\ \boldsymbol{u} &= \boldsymbol{0} \;\; \mbox{ on } \;\; \Gamma,
	\end{aligned}\right.
\end{align}
where $\boldsymbol{\mathcal{F}}$ is an external forcing. 
\begin{lemma}\label{RegularityStokes}
	Let $\Omega$ be of the class $\mathcal{O}^2(\mathbb{R}^2)$ or $\mathcal{O}^3(\mathbb{R}^3)$, the weak solution to the Stokes problem \eqref{StokesProblem} satisfies $(\boldsymbol{u}, p) \in \left[\boldsymbol{H}_0^1(\Omega) \cap \boldsymbol{H}^{3/2 + \varkappa}(\Omega)\right] \times H^{1/2 + \varkappa}(\Omega)$ for $\varkappa \in (-\frac{1}{2}, 0) \cup (0, \frac{1}{2})$ given the data $\boldsymbol{\mathcal{F}} \in \boldsymbol{H}^{-1/2 + \varkappa}(\Omega)$. 	
\end{lemma}  
The regularity result in Lemma \ref{RegularityStokes} holds for different examples of domains depending on $\varkappa$, which in turn depends on the interior angle of the re-entrant corners (see for reference \cite{MD_LectureNotes}). In three dimensions, in particular, $\varkappa \in \big(-\frac{1}{2},0\big)$ covers (see (1.6) in \cite{MoniqueDaugeRegularity}),
\begin{align}
	&\mbox{any domain in $\mathcal{O}_3(\mathbb{R}^3)$ with corners,} \label{ExDomain1}
\end{align}
and, $\varkappa \in \big(-\frac{1}{2}, 0\big) \cup (0, 0.044]$ holds for
\begin{align}\label{Exdomain2}
	\begin{aligned}
		&\Omega = Q_1 \backslash Q_2 \; \mbox{where} \; Q_1 \; \mbox{and} \; Q_2 \; \mbox{are two rectangular parallelepipeds} \\
		&\mbox{with the same axes (see (1.7), \cite{MoniqueDaugeRegularity})}. 
	\end{aligned}
\end{align}
Furthermore, we also state a regularity result for elliptic equations on bounded Lipschitz domains in $\mathbb{R}^d$ in the following lemma (see for reference  \cite{EllipticEquationsOnPolyhedralDomains}, \cite{Nicaise} and \cite{SavareRegularity}).

\begin{lemma}\label{RegularityElliptic}
	Let $\Omega$ be of  class $\mathcal{O}^2(\mathbb{R}^2)$ or $\mathcal{O}^3(\mathbb{R}^3)$. Then under the assumption of real valued, uniformly positive definite and Lipschitz continuous $A$, for every $\varkappa \in (-\frac{1}{2},0) \cup (0,\frac{1}{2})$, if 
	$$f \in H^{-1/2 + \varkappa} (\Omega),  \ g \in H^{1 + \varkappa}(\Gamma),$$
	the non-homogeneous Dirichlet problem
	$$- \bdiv\hspace{-0.06cm}(A \nabla y )= f \;\; \mbox{in} \;\; \Omega, \;\; y = g \;\; \mbox{on} \;\; \Gamma,$$
	admits a unique solution $y \in H^{3/2 + \varkappa}(\Omega)$ and the following estimate holds:
	\begin{align}\label{EllipticReg_Est}
		\|y\|_{3/2+\varkappa,\Omega} \leq C \left(\|f\|_{-1/2+\varkappa,\Omega} + \|g\|_{1+\varkappa,\Gamma}\right). 
	\end{align}
\end{lemma}

Now we are in a position to prove the main regularity result for the system  \eqref{P:GE}.  As $\nu(\cdot)$ is globally Lipschitz continuous, $\nu(\cdot)$ is differentiable a.e. by Rademacher's theorem (see for reference Chapter 5, Theorem 6, \cite{evans2022partial})and we denote the derivative by $\nu_T(\cdot)$ which is a.e. bounded by $\gamma_{\nu}$.  
\begin{theorem}\label{Regularity}
	Let $\boldsymbol{y}^D \in \left[H^{1 + \delta}(\Gamma)\right]^2, $  for $\delta \in (0,\frac{1}{2})$ be given.  Then for any domain $\Omega$ of the class $\mathcal{O}^2(\mathbb{R}^2)$ or $\mathcal{O}^3(\mathbb{R}^3)$, the weak solution to \eqref{P:GE} satisfies  $$(\boldsymbol{u}, p, \boldsymbol{y}) \in \left[\boldsymbol{H}_0^1(\Omega) \cap \boldsymbol{H}^{3/2+\delta}(\Omega)\right] \times \left[L^2_0(\Omega) \cap H^{1/2+\delta}(\Omega)\right] \times \left[H^{3/2+ \delta}(\Omega)\right]^2$$ and 
	\begin{align}\label{M}
		\norm{\boldsymbol{u}}_{3/2+\delta,\Omega} + \norm{\boldsymbol{y}}_{3/2+\delta,\Omega} \leq M,
	\end{align}
	where, $M$ depends on the data $\|\boldsymbol{y}^D\|_{1+\delta,\Gamma}$ and is defined in \eqref{explicitbound_u_2D} and \eqref{exlplicitbound_u_3D_B} for two and three dimensions, respectively.
\end{theorem}
\begin{proof}
	\underline{Idea of the proof:} One can rewrite the system \eqref{P:GE} in $\Omega$ as 
	\begin{align}\label{Reg_1}
		\left\{
		\begin{aligned}
			- \nu(T) \Delta\boldsymbol{u}+ \nabla p &= \boldsymbol{F}(\boldsymbol{y}) - \boldsymbol{K}^{-1} \boldsymbol{u} - (\boldsymbol{u} \cdot \nabla) \boldsymbol{u} + (\nabla \nu(T) \cdot \nabla) \boldsymbol{u}=: \tilde{\boldsymbol{\mathcal{F}}} ,\\
			\bdiv\hspace{0.04cm}\boldsymbol{u} &= 0 , \\
			-\boldsymbol{\bdiv}\hspace{-0.1cm}(\boldsymbol{D} \nabla \boldsymbol{y}) &= -(\boldsymbol{u} \cdot \nabla) \boldsymbol{y}=:\tilde{\boldsymbol{f}}, \\
			\boldsymbol{y} = \boldsymbol{y}^D, \;\; \boldsymbol{u} &= \boldsymbol{0} \;\; \mbox{on} \;\; \Gamma. 
		\end{aligned}
		\right.
	\end{align}
	The idea is to show that $\tilde{\boldsymbol{\mathcal{F}}} \in \boldsymbol{H}^{-1/2 + \delta}(\Omega)$, $\tilde{\boldsymbol{f}} \in \left[H^{-1/2 + \delta} (\Omega)\right]^2$, so that we can use the regularity results available in Lemmas \ref{RegularityStokes} and \ref{RegularityElliptic} to reach our desired regularity result. Observe that in Lemma \ref{RegularityStokes}, the authors have developed the regularity of the Stokes problem on Lipschitz domains for $\nu(T) = 1,$ and we were not able to find any relevant reference in the literature which addresses the case of variable coefficients. To address this we use $\nu_1 \leq \nu(T) \leq \nu_2$ and rewrite \eqref{Reg_1} as follows:
	\begin{align}\label{Reg_2}
		\left\{
		\begin{aligned}
			-\Delta\boldsymbol{u}+ \nabla \tilde{p} &= \left(\frac{1}{\nu(T)}\right)\left(\boldsymbol{F}(\boldsymbol{y}) - \boldsymbol{K}^{-1} \boldsymbol{u} - (\boldsymbol{u} \cdot \nabla) \boldsymbol{u} + (\nabla \nu(T) \cdot \nabla) \boldsymbol{u} - p \frac{\nu_T(T) \nabla T}{\nu(T)}\right)=: {\boldsymbol{\mathcal{F}}} ,\\
			\bdiv\hspace{0.04cm}\boldsymbol{u} &= 0 , \\
			-\boldsymbol{\bdiv}\hspace{-0.1cm}(\boldsymbol{D} \nabla \boldsymbol{y}) &= -(\boldsymbol{u} \cdot \nabla) \boldsymbol{y}=:{\boldsymbol{f}}, \\
			\boldsymbol{y} = \boldsymbol{y}^D, \;\; \boldsymbol{u} &= \boldsymbol{0} \;\; \mbox{on} \;\; \Gamma,
		\end{aligned}
		\right.
	\end{align}		
	where $\tilde{p} := \frac{p}{\nu(T)} - m,$ with $m$ being the mean value of $\frac{p}{\nu(T)} $  so that $\tilde{p} \in L_0^2(\Omega)$.  
	\smallskip
	In two dimensions, we first prove that $\boldsymbol{f} \in [H^{-1/2 + \delta}(\Omega)]^2$ in \textit{step 1} followed by $\boldsymbol{\mathcal{F}} \in \boldsymbol{H}^{-1/2 + \delta}(\Omega)$ in \textit{step 2} which uses the regularity result derived in the previous step to handle the term $(\nabla \nu(T) \cdot \nabla) \boldsymbol{u}.$ The same approach fails in three dimensions. In this case, the major difficulties lie in estimating the nonlinear terms $(\boldsymbol{u} \cdot \nabla) \boldsymbol{y}, (\boldsymbol{u} \cdot \nabla) \boldsymbol{u}$ and $(\nabla \nu(T) \cdot \nabla) \boldsymbol{u}$ and the main idea to deal with them is to prove the regularity of the solution  in a less regular space $(\boldsymbol{u}, p, \boldsymbol{y}) \in [\boldsymbol{H}_0^1(\Omega) \cap \boldsymbol{H}^{3/2-\delta}(\Omega)] \times L^2_0(\Omega)\cap  H^{1/2-\delta}(\Omega) \times [ H^{3/2-\delta}(\Omega)]^2$  by showing that $\boldsymbol{\mathcal{F}} \in \boldsymbol{H}^{-1/2 - \delta}(\Omega)$, $\boldsymbol{f} \in [H^{-1/2 - \delta} (\Omega)]^2$ and $\boldsymbol{y}^D \in [H^{1 - \delta}(\Gamma)]^2$. Then we use this as a stepping stone to reach our goal. We also observe that in order to prove the regularity: $\boldsymbol{u} \in \boldsymbol{H}^{3/2-\delta}(\Omega)$ in three dimensions, more than $[H^{3/2 - \delta}(\Omega)]^2$ regularity is needed for $\boldsymbol{y}$, this is due to the nonlinearity in the diffusion coefficient.  Keeping these observations in mind, we divide the proof into two parts, providing a  clear exposition in two and three dimensions, separately. In the first step, we prove the regularity of the solution variable $\boldsymbol{y}$  followed by the regularity of $\boldsymbol{u}$.
	
	\smallskip
	
	\underline{Part I:} Two dimensions. 
	
	\smallskip
	
	\underline{Step 1} [Regularity of $\boldsymbol{y}$ ($[ H^{3/2+\delta}(\Omega)]^2$).]  To begin we recall the  fractional Leibniz rule \cite{AnconaLiebnitz}	(see Remark  \ref{rem:Leibniz}) with $\frac{1}{p_1} + \frac{1}{p_2} = \frac{1}{2}$ and $\frac{1}{q_1} + \frac{1}{q_2} = \frac{1}{2}$, we find 
	
	\begin{align}
		\norm{\left(\boldsymbol{u} \cdot \nabla\right) \boldsymbol{y}}_{{-1/2 + \delta},\Omega} &= \norm{\nabla \cdot \left(\boldsymbol{u} \otimes \boldsymbol{y}\right)}_{{-1/2 + \delta},\Omega}\leq  C \norm{\boldsymbol{u} \otimes \boldsymbol{y}}_{{1/2 + \delta},\Omega} \nonumber\\
		&\leq  C \left(\norm{\boldsymbol{u}}_{W^{1/2 + \delta, p_1}(\Omega)} \norm{\boldsymbol{y}}_{L^{p_2}(\Omega)} + \norm{\boldsymbol{u}}_{L^{q_1}(\Omega)}  \norm{\boldsymbol{y}}_{W^{1/2+\delta,q_2}(\Omega)}\right). \label{RegEq1}
	\end{align}		
	
	Now we apply the fractional Gagliardo-Nirenberg inequality \eqref{FractionalGagliardoNirenberg} for $\theta =k= 1, p_1 = 4/(1+2\delta) = q_2$ along with \eqref{H1embedding} for $p_2 = 4/(1-2\delta) = q_1$, to get
	\begin{align}
		\norm{\left(\boldsymbol{u} \cdot \nabla\right) \boldsymbol{y}}_{{-1/2 + \delta},\Omega} \leq C \norm{\boldsymbol{u}}_{1,\Omega} \norm{\boldsymbol{y}}_{1,\Omega}. \label{RegEq2_2d}
	\end{align}
	
	Next invoking Lemma \ref{EnergyEst:State} and Theorem \ref{StateExistence}, we infer that $\norm{\boldsymbol{f}}_{-1/2+\delta,\Omega}$ is finite.  Thus an application of Lemma \ref{RegularityElliptic} yields  $\boldsymbol{y} \in \left[H^{3/2 + \delta}(\Omega)\right]^2$.  Let us denote by $M_{\boldsymbol{y}^D, \boldsymbol{U}} := \norm{\boldsymbol{y}^D}_{1/2,\Gamma} + \norm{\boldsymbol{y}^D}_{1+\delta,\Gamma}$, then using the above bounds, Lemma \ref{EnergyEst:State} and \eqref{EllipticReg_Est}, we obtain
	\begin{align}\label{explicitbound_y_2D}
		\norm{\boldsymbol{y}}_{3/2+\delta,\Omega} \leq \norm{\boldsymbol{f}}_{-1/2+\delta,\Omega} \leq C_{l_1}\; M^2_{\boldsymbol{y}^D,\boldsymbol{U}}, 
	\end{align}
	where the positive constant $C_{l_1}$ depends on $C_{\boldsymbol{u}}, C_{\boldsymbol{y}}$ and the embedding constants used. 
	\smallskip

	\underline{Step 2} [Regularity of $\boldsymbol{u}$ ($\boldsymbol{H}^{3/2+\delta}(\Omega)$).] Navier-Stokes nonlinearity can be dealt with analogous to Step 1 and $\nu_1 \leq \nu(T) \leq \nu_2$ yielding
	$$\norm{(1/\nu(T))(\boldsymbol{u} \cdot \nabla) \boldsymbol{u}}_{-1/2 + \delta, \Omega} \leq C \norm{\boldsymbol{u}}^2_{1,\Omega}.$$
	The embedding $\boldsymbol{L}^2(\Omega) \hookrightarrow {\boldsymbol{H}^{-1/2 + \delta}(\Omega)}$ and  assumption on the buoyancy term gives
	\begin{align*}
		&\norm{(1/\nu(T))\boldsymbol{K}^{-1} \boldsymbol{u}}_{-1/2+\delta,\Omega} \leq \frac{C}{\nu_1} \norm{\boldsymbol{K}^{-1}}_{L^{\infty}(\Omega)} \norm{\boldsymbol{u}}_{0,\Omega} \leq C \norm{\boldsymbol{u}}_{1,\Omega},\\
		&\norm{(1/\nu(T))F(\boldsymbol{y})}_{-1/2+\delta,\Omega} \leq  \frac{C_F}{\nu_1}  \norm{\boldsymbol{y}}_{0,\Omega} \leq C \norm{\boldsymbol{y}}_{1,\Omega}.
	\end{align*}
	Using the Lipschitz continuity of $\nu(T)$,  \eqref{FractionalGagliardoNirenberg} and H\"{o}lder's inequality, one can derive
	\begin{align}
		\norm{(1/\nu(T))(\nabla \nu(T) \cdot \nabla) \boldsymbol{u}}_{-1/2 + \delta, \Omega} &\leq \frac{1}{\nu_1} \norm{\nu_{T}(T) (\nabla T \cdot \nabla) \boldsymbol{u}}_{-1/2+\delta,\Omega} \\
		&= \frac{1}{\nu_1} \norm{D^{-1/2+\delta} (\nu_{T}(T)(\nabla T \cdot \nabla) \boldsymbol{u})}_{0,\Omega}\nonumber\\
		&\leq C \norm{ D^{1/2-\delta} (D^{-1/2+\delta} (\nu_{T}(T)(\nabla T \cdot \nabla) \boldsymbol{u}))}_{L^{4/(3-2\delta)}(\Omega)} \nonumber\\&\leq C \norm{\nu_T(T)}_{L^{\infty}(\Omega)} \norm{(\nabla T \cdot \nabla) \boldsymbol{u}}_{L^{4/(3-2\delta)}(\Omega)}\nonumber\\
		&\leq C \norm{\boldsymbol{u}}_{W^{1,p_1}(\Omega)} \norm{T}_{W^{1,p_2}(\Omega)}, \label{eqReg_4_2d}
	\end{align}
	with $ \frac{1}{p_1} + \frac{1}{p_2} = \frac{3-2\delta}{4}.$ Now choosing $p_1 = 2$ and applying \eqref{FractionalGagliardoNirenberg} yields
	\begin{align*}
		\norm{(1/\nu(T))(\nabla \nu(T) \cdot \nabla) \boldsymbol{u}}_{-1/2 + \delta, \Omega} &\leq C \norm{\boldsymbol{u}}_{1,\Omega} \|T\|_{W^{1,4/(1-2\delta)}(\Omega)} \leq C \norm{\boldsymbol{u}}_{1,\Omega} \norm{T}_{3/2+\delta,\Omega}  \\
		&\leq C \norm{\boldsymbol{u}}_{1,\Omega} \norm{\boldsymbol{y}}_{3/2+\delta,\Omega},  
	\end{align*}
	and the right hand side is finite by Step 1. Furthermore, similar to \eqref{eqReg_4_2d} we have the following:
	\begin{align*}
		\norm{p \frac{\nu_T(T) \nabla T}{(\nu(T))^2}}_{-1/2+\delta,\Omega} \leq \frac{\norm{\nu_T(T)}_{L^{\infty}(\Omega)}}{\nu_1^2} \norm{p \nabla T}_{-1/2+\delta,\Omega} \leq C \norm{p}_{0,\Omega} \norm{T}_{3/2+\delta,\Omega}.
	\end{align*}
	Now invoking   Lemma \ref{EnergyEst:State} and \eqref{explicitbound_y_2D} gives
	\begin{align}\label{explicitbound_u_2D}
		\norm{\boldsymbol{u}}_{3/2+\delta,\Omega} \leq \norm{\boldsymbol{\mathcal{F}}}_{-1/2+\delta,\Omega} \leq C_{l_3} \left(M_{\boldsymbol{y}^D} + M_{\boldsymbol{y}^D}^2 + M_{\boldsymbol{y}^D}^3\right).
	\end{align}
	which completes the proof in two dimensions. 
	
	\smallskip
	
	\underline{Part II:} Three dimensions.
	
	\smallskip
	
	\underline{Step 1} [Regularity of $\boldsymbol{y}$ ($[ H^{3/2+\delta}(\Omega)]^2$).] Choosing $p_1 = 3/(1-\delta) = q_2$ and $p_2 = 6/(1+2\delta) = q_1$ in \eqref{RegEq1} and applying \eqref{FractionalGagliardoNirenberg} with $l=0$,  $r_1=2$ and $\theta = 1$ yields \eqref{RegEq2_2d}, which is finite due to Lemma \ref{EnergyEst:State} and Theorem \ref{StateExistence}. Thus, $\boldsymbol{y} \in [H^{3/2-\delta}(\Omega)]^2.$ Let us fist consider  $0<\delta\leq\frac{1}{4}$. Then following steps used to obtain \eqref{RegEq1}
	\begin{align}
		\norm{\left(\boldsymbol{u} \cdot \nabla\right) \boldsymbol{y}}_{{-1/2 + \delta},\Omega} \leq  C \left(\norm{\boldsymbol{u}}_{W^{1/2 + \delta, p_1}(\Omega)} \norm{\boldsymbol{y}}_{L^{p_2}(\Omega)} + \norm{\boldsymbol{u}}_{L^{q_1}(\Omega)} \norm{\boldsymbol{y}}_{W^{1/2+\delta,q_2}(\Omega)}\right), \label{RegEq2_3D}
	\end{align}
	where, $\frac{1}{p_1} + \frac{1}{p_2} = \frac{1}{2}$ and $\frac{1}{q_1} + \frac{1}{q_2} = \frac{1}{2}$. Choosing $p_1 = \frac{3}{1+\delta}, p_2 = \frac{6}{1-2\delta}, q_1 = 6, q_2 = 3$ in \eqref{RegEq2_3D} and invoking the fractional Sobolev embedding (see for reference Theorem 4.57 in \cite{FractionalSobolevEmbeddings}),
	\begin{align}\label{FractionalSobolevEmbedding}
		\begin{aligned}
			\mbox{ if } \; l r_1 < d, \; \mbox{ then } \;  \boldsymbol{W}^{l, r_1}(\Omega) \hookrightarrow \boldsymbol{L}^{p_2}(\Omega) \;  \mbox{ for every} \; p_2 \leq d r_1 / (d - l r_1),
		\end{aligned}
	\end{align}
	with $l = \frac{3}{2} - \delta$ and $r_1 = 2$, we deduce that the embedding holds for $p_2 \in (0, \infty)$ and 
	\begin{align}\label{40}
		\norm{\left(\boldsymbol{u} \cdot \nabla\right) \boldsymbol{y}}_{-1/2+\delta,\Omega} &\leq C \left( \norm{\boldsymbol{u}}_{1,\Omega} \norm{\boldsymbol{y}}_{L^{6/(1-2\delta)}(\Omega)} +  \norm{\boldsymbol{u}}_{1,\Omega}  \norm{\boldsymbol{y}}_{1+\delta,\Omega}\right) \nonumber\\
		&\leq C \left( \norm{\boldsymbol{u}}_{1,\Omega} \norm{\boldsymbol{y}}_{3/2-\delta,\Omega} +  \norm{\boldsymbol{u}}_{1,\Omega}  \norm{\boldsymbol{y}}_{3/2-\delta,\Omega}\right),
	\end{align}
	since $0<\delta\leq\frac{1}{4}$. The right hand side of the above inequality 	is finite, so that  $\boldsymbol{y} \in [H^{3/2+\delta}(\Omega)]^2$. 
	Let us use the previous result to  establish  regularity results for $\delta\in(\frac{1}{4},\frac{1}{2})$. By the previous result, we know that $\boldsymbol{y} \in [H^{3/2+\delta_a}(\Omega)]^2$, for all $\delta_a\in(0,\frac{1}{4}]$. For $\delta\in(0,\frac{1}{2})$ and  $\delta_b\in(\frac{1}{4},\frac{1}{2})$,  we consider 
	\begin{align*}
		\norm{\left(\boldsymbol{u} \cdot \nabla\right) \boldsymbol{y}}_{-1/2+\delta_b,\Omega} &\leq  C \big(\norm{\boldsymbol{u}}_{W^{1/2+\delta_b, 3/(1+\delta_b)}(\Omega)} \norm{\boldsymbol{y}}_{L^{6/(1-2\delta_b)}(\Omega)} \\
		&\quad+  \norm{\boldsymbol{u}}_{L^{3/(1-(\delta_b - \delta_a))}(\Omega)}  \norm{\boldsymbol{y}}_{W^{1/2+\delta_b,6/(1+2(\delta_b-\delta_a))}(\Omega)}\big) \nonumber\\
		&\leq C \left(\norm{\boldsymbol{u}}_{1,\Omega} \norm{\boldsymbol{y}}_{3/2-\delta,\Omega} + \norm{\boldsymbol{u}}_{1,\Omega}  \norm{\boldsymbol{y}}_{3/2+\delta_a, \Omega}\right).
	\end{align*}	
	Thus, combining the above bounds, we obtain $\boldsymbol{y} \in [H^{3/2 + \delta}(\Omega)],$ for $\delta\in(0,\frac{1}{2})$. Furthermore, similar to Step 1 of Part I, we have the following dependence on data in three dimensions:
	\begin{align}
		\norm{\boldsymbol{y}}_{3/2-\delta,\Omega} \leq C_{l_5}\; M^2_{\boldsymbol{y}^D} \; \mbox{and} \;
		\norm{\boldsymbol{y}}_{3/2+\delta,\Omega} \leq C_{l_6}  \left(M^3_{\boldsymbol{y}^D} + M^4_{\boldsymbol{y}^D}\right) := \hat{M}_{\boldsymbol{y}^D}. \label{explicitbound_y_3D}
	\end{align}
   \smallskip
	\underline{Step 2} [Regularity of $\boldsymbol{u}$ ($\boldsymbol{H}^{3/2-\delta}(\Omega)$).] The Navier-Stokes nonlinearity can be dealt by using the same exponents as in Step 1 of Part II yielding
	$$\norm{(1/\nu(T))(\boldsymbol{u} \cdot \nabla) \boldsymbol{u}}_{-1/2- \delta, \Omega} \leq C\|\boldsymbol{u}\otimes\boldsymbol{u}\|_{1/2-\delta,\Omega}\leq C\|\boldsymbol{u}\|_{W^{1/2-\delta,3/(1-\delta)}(\Omega)}\|\boldsymbol{u}\|_{L^{6/(1+2\delta)}(\Omega)}\leq C \norm{\boldsymbol{u}}^2_{1,\Omega}.$$
	In order to estimate $\norm{(1/\nu(T))(\nabla \nu(T) \cdot \nabla) \boldsymbol{u}}_{-1/2 - \delta, \Omega}$, we use an iterative technique. We first consider $\frac{1}{2(k+1)}\leq\delta<\frac{1}{2}$,  and  show that $\boldsymbol{u}\in\boldsymbol{H}^{3/2-k\delta}(\Omega),$	for $k=1,2,\ldots.$  Then by using this regularity, we prove that $\boldsymbol{u}\in\boldsymbol{H}^{3/2-(k-1)\delta}(\Omega),$ and we continue this procedure for $(k-1)$ times to finally obtain $\boldsymbol{u}\in\boldsymbol{H}^{3/2-\delta}(\Omega).$ Then by taking sufficiently large $k$, one can obtain this result for $0<\delta<\frac{1}{2}$.   Similar to \eqref{eqReg_4_2d}, one can derive the following for $k = 1, 2, 3, \ldots,$ in three dimensions
	\begin{align}\label{eq*_1}
		\norm{(1/\nu(T))(\nabla \nu(T) \cdot \nabla) \boldsymbol{u}}_{-1/2 - k\delta, \Omega} \leq C \norm{\boldsymbol{u}}_{W^{1,p_1}(\Omega)} \norm{T}_{W^{1,p_2}(\Omega)}, \;\; \mbox{with} \;\; \frac{1}{p_1} + \frac{1}{p_2} = \frac{2+k\delta}{3}.
	\end{align}
	Choosing $p_1 = 2$ which implies $p_2 = 6/(1+2 k \delta)$ and applying \eqref{FractionalGagliardoNirenberg} gives
	\begin{align}\label{eq!!}
		\norm{(1/\nu(T)) (\nabla \nu(T) \cdot \nabla) \boldsymbol{u}}_{-1/2 - k \delta, \Omega} \leq C \norm{\boldsymbol{u}}_{1,\Omega} \norm{T}_{2-k\delta,\Omega} \leq C \norm{\boldsymbol{u}}_{1,\Omega} \norm{T}_{3/2+\delta,\Omega},
	\end{align}
	provided $\frac{1}{2}\leq(k+1)\delta$. 	Handling the other bounds as in Step 2 of Part I and using Lemma \ref{EnergyEst:State} and additional regularity of $\boldsymbol{y}$ obtained in Step 1, we infer that $\norm{\boldsymbol{\mathcal{F}}}_{-1/2-\delta,\Omega}$ is finite and thus $\boldsymbol{u} \in \boldsymbol{H}^{3/2-k\delta}(\Omega),$ for $k = 1, 2, \dots.$ Using this regularity of $\boldsymbol{u}$ and following the same steps, we obtain
	\begin{align}\label{eq*_2}
		\norm{(1/\nu(T)) (\nabla \nu(T) \cdot \nabla) \boldsymbol{u}}_{-1/2 - (k-1)\delta, \Omega} \leq C \norm{\boldsymbol{u}}_{W^{1,p_1}(\Omega)} \norm{T}_{W^{1,p_2}(\Omega)}, \;\; \mbox{with} \;\; \frac{1}{p_1} + \frac{1}{p_2} = \frac{2+(k-1)\delta}{3}.
	\end{align}
	Choosing $p_1 = \frac{3}{1+k\delta}$ which implies $p_2 = \frac{3}{1-\delta}$ and applying \eqref{FractionalGagliardoNirenberg} provides 
	\begin{align}\label{eq!!_1}
		\norm{(1/\nu(T)) (\nabla \nu(T) \cdot \nabla) \boldsymbol{u}}_{-1/2 - (k-1)\delta, \Omega} \leq C \norm{\boldsymbol{u}}_{3/2-k\delta,\Omega} \norm{T}_{3/2+\delta,\Omega}.
	\end{align}
	Tackling the other bounds analogously, we can readily obtain $\boldsymbol{u} \in \boldsymbol{H}^{3/2 - (k-1) \delta}(\Omega)$. Now we claim that $\boldsymbol{u} \in \boldsymbol{H}^{3/2-(k-2)\delta}(\Omega)$ using the regularity of $\boldsymbol{u}$ that we have  just obtained. Indeed, similar to \eqref{eq*_2}, we have
	\begin{align}\label{eq*_3}
		\norm{(1/\nu(T)) (\nabla \nu(T) \cdot \nabla) \boldsymbol{u}}_{-1/2 - (k-2)\delta, \Omega} \leq C \norm{\boldsymbol{u}}_{W^{1,p_1}(\Omega)} \norm{T}_{W^{1,p_2}(\Omega)}, \;\; \mbox{with} \;\; \frac{1}{p_1} + \frac{1}{p_2} = \frac{2+(k-2)\delta}{3}.
	\end{align}
	Choosing $p_1 = \frac{3}{1+(k-1)\delta}$ which implies $p_2 = \frac{3}{1-\delta}$ and applying \eqref{FractionalGagliardoNirenberg} leads to
	$$\norm{(1/\nu(T)) (\nabla \nu(T) \cdot \nabla) \boldsymbol{u}}_{-1/2 - (k-2)\delta, \Omega} \leq C \norm{\boldsymbol{u}}_{3/2-(k-1)\delta,\Omega} \norm{T}_{3/2+\delta,\Omega}.$$
	Repeating this procedure $(k-1)$ times will subsequently lead the desired regularity of $\boldsymbol{u}$, that is, $\boldsymbol{u} \in \boldsymbol{H}^{3/2-\delta}(\Omega)$. Furthermore, 
	\begin{align*}
		\norm{p \frac{\nu_T(T) \nabla T}{(\nu(T))^2}}_{-1/2-\delta,\Omega} \leq \frac{\norm{\nu_T(T)}_{L^{\infty}(\Omega)}}{\nu_1^2} \norm{p \nabla T}_{-1/2-\delta,\Omega},
	\end{align*}
	where $\norm{p \nabla T}_{-1/2-\delta,\Omega}$ is handled analogously to $\norm{(\nabla \nu(T) \cdot \nabla) \boldsymbol{u}}_{-1/2 - \delta, \Omega}$ using the iterative technique.
	Rest of the bounds can be handled similar to Step 2 of Part I. Moreover, the following also holds due to Lemma \ref{EnergyEst:State} and \eqref{explicitbound_y_3D}
	\begin{align}\label{explicitbound_u_3D}
		\norm{\boldsymbol{u}}_{3/2-\delta,\Omega} \leq \norm{\boldsymbol{\mathcal{F}}}_{-1/2-\delta,\Omega} \leq C_{l_7} \left(M_{\boldsymbol{y}^D} + M_{\boldsymbol{y}^D}^2 + M_{\boldsymbol{y}^D} \hat{M}^k_{\boldsymbol{y}^D}\right) =: \tilde{M}_{\boldsymbol{y}^D}.
	\end{align}
    \smallskip
	\underline{Step 3} [Regularity of $\boldsymbol{u}$  ($\boldsymbol{H}^{3/2+\delta}(\Omega)$).] The Navier-Stokes nonlinearity is handled in a similar manner to Step 1 of Part II (see the first term of \eqref{40}) by using the extra regularity of $\boldsymbol{u}$ derived in the previous step
	$$\norm{(1/\nu(T))(\boldsymbol{u} \cdot \nabla) \boldsymbol{u}}_{-1/2 + \delta, \Omega} \leq C \norm{\boldsymbol{u}}_{1,\Omega} \norm{\boldsymbol{u}}_{3/2-\delta,\Omega},$$ for $\delta\in(0,\frac{1}{2})$. For $\delta_a\in(0,\frac{1}{4})$, similar to \eqref{eqReg_4_2d}, we can derive the following in three dimensions
	\begin{align}\label{3d_deltaA}
		\norm{(1/\nu(T)) (\nabla \nu(T) \cdot \nabla) \boldsymbol{u}}_{-1/2 + \delta_a, \Omega} \leq C \norm{\boldsymbol{u}}_{W^{1,p_1}(\Omega)} \norm{T}_{W^{1,p_2}(\Omega)}, \;\; \mbox{with} \;\; \frac{1}{p_1} + \frac{1}{p_2} = \frac{2-\delta_a}{3}.
	\end{align}
	Choosing $p_1 = 3/(1+ \delta_a)$ and applying \eqref{FractionalGagliardoNirenberg} with $l = 1, \theta = 1$ yields
	\begin{align}\label{42}
		\norm{\nabla \boldsymbol{u}}_{L^{3/(1+\delta_a)}(\Omega)} \leq C \norm{\boldsymbol{u}}_{3/2 - \delta_a, \Omega}.
	\end{align}
	Now using \eqref{FractionalGagliardoNirenberg} with $l=1$, $\theta = 1$ and $p_2 = 3/(1-2\delta_a)$, we get
	\begin{align}\label{43}
		\norm{\nabla T}_{L^{3/(1-2\delta_a)}(\Omega)} \leq C \norm{T}_{3/2 + 2\delta_a,\Omega}, 
	\end{align}
	where $2 \delta_a$ belongs to $(0,\frac{1}{2})$. Putting the bounds \eqref{42} and \eqref{43} in \eqref{3d_deltaA} results to 
	\begin{align}
		\norm{(1/\nu(T)) (\nabla \nu(T) \cdot \nabla) \boldsymbol{u}}_{-1/2 + \delta_a, \Omega} &\leq C \norm{\boldsymbol{u}}_{3/2 - \delta_a, \Omega} \norm{T}_{3/2 + 2 \delta_a,\Omega}<+\infty.
	\end{align}
	From the above bounds, \eqref{explicitbound_y_3D}, and \eqref{explicitbound_u_3D} analogous to the previous steps, one can conclude that $\boldsymbol{u} \in \boldsymbol{H}^{3/2 + \delta_a}(\Omega),$ for $\delta_a\in(0,\frac{1}{4})$ and 
	\begin{align}\label{explicitbound_u_3D_A}
		\norm{\boldsymbol{u}}_{3/2+\delta_a,\Omega}  \leq C_{l_8} \left(M_{\boldsymbol{y}^D} \tilde{M}_{\boldsymbol{y}^D}+\tilde{M}_{\boldsymbol{y}^D} \hat{M}_{\boldsymbol{y}^D} + M_{\boldsymbol{y}^D}\right) =: \bar{M}_{\boldsymbol{y}^D}.
	\end{align}
	Now, for $\delta_b\in[\frac{1}{4},\frac{1}{2})$, proceeding similarly results in
	\begin{align}\label{3d_deltaB}
		\norm{(1/\nu(T))  (\nabla \nu(T) \cdot \nabla) \boldsymbol{u}}_{-1/2 + \delta_b, \Omega} \leq C \norm{\boldsymbol{u}}_{W^{1,p_1}(\Omega)} \norm{T}_{W^{1,p_2}(\Omega)}, \;\; \mbox{with} \;\; \frac{1}{p_1} + \frac{1}{p_2} = \frac{2-\delta_b}{3}.
	\end{align}
	Choosing $p_1 = 3/(1-\delta_a), p_2 = 3/(1-(\delta_b - \delta_a))$, and applying \eqref{FractionalGagliardoNirenberg} with $l = 1, \theta = 1$ gives
	\begin{align*}
		\norm{(1/\nu(T)) (\nabla \nu(T) \cdot \nabla) \boldsymbol{u}}_{-1/2 + \delta_b, \Omega} &\leq C \norm{\nabla \boldsymbol{u}}_{L^{3/(1-\delta_a)}(\Omega)} \norm{\nabla T}_{L^{3/(1-(\delta_b - \delta_a))}(\Omega)}	\\
		&\leq C \norm{\boldsymbol{u}}_{3/2 + \delta_a,\Omega} \norm{T}_{3/2+(\delta_b - \delta_a), \Omega},
	\end{align*}
	where $(\delta_b - \delta_a)\in(0,\frac{1}{2})$. This is finite, since $\boldsymbol{u}\in\boldsymbol{H}^{3/2 + \delta_a}(\Omega),$ for $\delta_a\in(0,\frac{1}{4})$ and $T\in H^{3/2+\delta}(\Omega)$ for $\delta\in(0,\frac{1}{2})$. Thus, $\boldsymbol{u} \in \boldsymbol{H}^{3/2 + \delta}(\Omega),$ for any $\delta\in(0,\frac{1}{2})$ and 
	\begin{align}\label{exlplicitbound_u_3D_B}
		\norm{\boldsymbol{u}}_{3/2+\delta,\Omega}  \leq C_{l_8} \left(M_{\boldsymbol{y}^D} \tilde{M}_{\boldsymbol{y}^D}+\bar{M}_{\boldsymbol{y}^D,\boldsymbol{U}} \hat{M}_{\boldsymbol{y}^D} + M_{\boldsymbol{y}^D}\right),
	\end{align}
	which completes the proof in three dimensions. 
\end{proof}

	\begin{remark}\label{rem:Leibniz}
		Note that  we have used the fractional Leibniz rule (1.1) from \cite{AnconaLiebnitz} to estimate the terms \eqref{RegEq1} and \eqref{RegEq2_3D}. For fractional Sobolve spaces defined on bounded domains (see Section 2, \cite{HitchikersGuideFractionalSobolev}), using the null-expansion of functions defined over $\Omega$ to $\mathbb{R}^n$,  one can show that  the estimate (1.1) in  \cite{AnconaLiebnitz}  remains valid for bounded domains also.
		We recall the fractional Leibniz rule in this context.
			Let $\Omega \subset \mathbb{R}^d$ be a bounded Lipschitz domain.
			Let $1< p_1,p_2,q_1,q_2 \le \infty$ satisfy
			$$
			\frac{1}{r}
			=
			\frac{1}{p_1}+\frac{1}{p_2}
			=
			\frac{1}{q_1}+\frac{1}{q_2},
			\quad
			s > \max \left\{0,\frac{d}{r}-d\right\}
			\quad\text{or}\quad
			s \in 2\mathbb{Z}_+ .
		   $$
			Then there exists a constant $C>0$, depending only on
			$\Omega,d,s$ and the indices, such that for all
			$
			u \in W^{s,p_1}(\Omega)\cap L^{q_1}(\Omega),
			v \in W^{s,q_2}(\Omega)\cap L^{p_2}(\Omega),
			$
			the product $uv$ belongs to $W^{s,r}(\Omega)$ and
		    $$
			\|uv\|_{W^{s,r}(\Omega)}
			\le
			C\Big(
			\|u\|_{W^{s,p_1}(\Omega)}\|v\|_{L^{p_2}(\Omega)}
			+
			\|u\|_{L^{q_1}(\Omega)}\|v\|_{W^{s,q_2}(\Omega)}
			\Big).
			$$	
		 A similar reasoning holds for \eqref{FractionalGagliardoNirenberg} given in \cite{FractionalGagliardoNirenberg}. Furthermore, since we are working with the full norms for $\boldsymbol{y}$, the nonhomogeneous Dirichlet boundary  condition for $\boldsymbol{y}$ does not affect the application of fractional Leibniz rule, see for example \eqref{RegEq1}.
	\end{remark}

Moreover, if the domain under  consideration is convex, then using the regularity result in Theorem 3.2.1.2 in \cite{grisvard2011elliptic} and Theorem 1.8 in \cite{girault2012finite} for elliptic equation, \cite{cattabriga1961problema}  (domains with $C^2$-boundary), Theorem 2 in \cite{kellogg1976regularity} (2D convex domain) and Theorem 6.3 in \cite{MoniqueDaugeRegularity} (3D convex domain) for the Stokes equation and following the same ideas developed in previous theorem, the regularity of solutions can be improved further.    

\begin{theorem}\label{H2regularity}
	Let $\boldsymbol{U} \in \boldsymbol{L}^r(\Omega),\; \boldsymbol{y}^D \in [H^{3/2}(\Gamma)]^2$ be given. Then for any convex domain $\Omega$ in $\mathbb{R}^d$, the weak solution to \eqref{P:GE} satisfies  $$(\boldsymbol{u}, p, \boldsymbol{y}) \in \left[\boldsymbol{H}_0^1(\Omega) \cap \boldsymbol{H}^{2}(\Omega)\right] \times [H^{1}(\Omega)\cap L^2_0(\Omega)] \times \left[H^{2}(\Omega)\right]^2.$$
\end{theorem}

\subsection{Uniqueness of solutions}
We use the regularity result derived in Theorem \ref{Regularity} to discuss the uniqueness of a regular weak solution to \eqref{P:Sred} in Theorem \ref{StateUniqueness}.

\begin{theorem}\label{StateUniqueness}
	Given $\boldsymbol{y}^D  \in \left[\boldsymbol{H}^{1+\delta}(\Gamma)\right]^2$. Let $(\boldsymbol{u},\boldsymbol{y}) \in [\boldsymbol{X} \cap \boldsymbol{H}^{3/2 + \delta}(\Omega)] \times [H^1(\Omega) \cap H^{3/2 + \delta}(\Omega)]^2$ be a solution of the reduced problem \eqref{P:Sred}, where $\delta\in(0,\frac{1}{2})$ and $\Omega$ be the domain as defined in Theorem \ref{Regularity}. Then if the condition 
	\begin{align}\label{StateUniquenessCondition}
		\begin{aligned}
			&\alpha_a > C_{6_d} C_{3_d} \left(\frac{\gamma_{\nu} C_{p_{2_d}} C_{gn} M M_{\boldsymbol{y}}}{\hat{\alpha}_a} + M_{\boldsymbol{u}} + \frac{\gamma_F M_{\boldsymbol{y}}}{\hat{\alpha}_a}\right),
		\end{aligned}
	\end{align}
	 is satisfied then the solution of \eqref{P:Sred} is unique. Here, the quantities $M_{\boldsymbol{u}}, M_{\boldsymbol{y}}$ and $M$ depend on the data and are defined in \eqref{M_u}, \eqref{M_y} and \eqref{M}, respectively.  
\end{theorem}
\begin{proof}
	Let $(\boldsymbol{u}_1, \boldsymbol{y}_1)$ and $(\boldsymbol{u}_2,\boldsymbol{y}_2)$ solve \eqref{P:Sred} and let $(\boldsymbol{\alpha},\boldsymbol{\beta}) := (\boldsymbol{u}_1-\boldsymbol{u}_2, \boldsymbol{y}_1-\boldsymbol{y}_2)$. Then subtracting the corresponding variational formulations for all $\boldsymbol{v} \in \boldsymbol{X}$, $\boldsymbol{s} \in [H_0^1(\Omega)]^2$, we have	
	\begin{align}\label{StateUniquenessEq1}
		\begin{aligned}
			a(\boldsymbol{y}_1;\boldsymbol{u}_1,\boldsymbol{v}) - a(\boldsymbol{y}_2;\boldsymbol{u}_2,\boldsymbol{v}) + c(\boldsymbol{u}_1,\boldsymbol{u}_1,\boldsymbol{v}) - c(\boldsymbol{u}_2, \boldsymbol{u}_2,\boldsymbol{v}) &= d(\boldsymbol{y}_1,\boldsymbol{v}) - d(\boldsymbol{y}_2,\boldsymbol{v}), \\
			a_{\boldsymbol{y}}(\boldsymbol{y}_1,\boldsymbol{s}) - a_{\boldsymbol{y}}(\boldsymbol{y}_2,\boldsymbol{s}) + c_{\boldsymbol{y}}(\boldsymbol{u}_1,\boldsymbol{y}_1,\boldsymbol{s}) - c_{\boldsymbol{y}}(\boldsymbol{u}_2,\boldsymbol{y}_2,\boldsymbol{s}) &= 0.
		\end{aligned}
	\end{align}	
	Notice that in \eqref{StateUniquenessEq1}, we can rewrite the differences as follows:
	\begin{align*}
		a(\boldsymbol{y}_1;\boldsymbol{u}_1,\boldsymbol{v}) - a(\boldsymbol{y}_2;\boldsymbol{u}_2,\boldsymbol{v}) &= a(\boldsymbol{y}_1;\boldsymbol{\alpha},\boldsymbol{v}) + a(\boldsymbol{y}_1;\boldsymbol{u}_2,\boldsymbol{v}) - a(\boldsymbol{y}_2;\boldsymbol{u}_2,\boldsymbol{v}), \\
		c(\boldsymbol{u}_1,\boldsymbol{u}_1,\boldsymbol{v}) - c(\boldsymbol{u}_2,\boldsymbol{u}_2,\boldsymbol{v}) &= c(\boldsymbol{u}_1,\boldsymbol{\alpha},\boldsymbol{v}) + c(\boldsymbol{\alpha},\boldsymbol{u}_2,\boldsymbol{v}), \\
		c_{\boldsymbol{y}}(\boldsymbol{u}_1,\boldsymbol{y}_1,\boldsymbol{s}) - c_{\boldsymbol{y}}(\boldsymbol{u}_2,\boldsymbol{y}_2,\boldsymbol{s}) &= c_{\boldsymbol{y}}(\boldsymbol{u}_1,\boldsymbol{\beta},\boldsymbol{s}) + c_{\boldsymbol{y}}(\boldsymbol{\alpha},\boldsymbol{y}_2,\boldsymbol{s}),
	\end{align*}
	and then we can choose as test function $\boldsymbol{v} = \boldsymbol{\alpha} \in \boldsymbol{X}$ and exploit \eqref{Prop:c1} to find 
	\begin{eqnarray*}
		a(\boldsymbol{y}_1;\boldsymbol{\alpha},\boldsymbol{\alpha}) + (a(\boldsymbol{y}_1;\boldsymbol{u}_2,\boldsymbol{\alpha}) - a(\boldsymbol{y}_2;\boldsymbol{u}_2,\boldsymbol{\alpha}))  + c(\boldsymbol{\alpha},\boldsymbol{u}_2,\boldsymbol{\alpha}) = d(\boldsymbol{y}_1,\boldsymbol{\alpha}) - d(\boldsymbol{y}_2,\boldsymbol{\alpha}).
	\end{eqnarray*}
	Using \eqref{Coer:a} in the above equation, we get
	\begin{eqnarray}\label{StateUniquenessEq2_}
		\alpha_a \norm{\boldsymbol{\alpha}}^2_{1,\Omega} \leq |(a(\boldsymbol{y}_1;\boldsymbol{u}_2,\boldsymbol{\alpha}) - a(\boldsymbol{y}_2;\boldsymbol{u}_2,\boldsymbol{\alpha}))| + |c(\boldsymbol{\alpha},\boldsymbol{u}_2,\boldsymbol{\alpha})| + |d(\boldsymbol{y}_1,\boldsymbol{\alpha}) - d(\boldsymbol{y}_2,\boldsymbol{\alpha})|.
	\end{eqnarray}
	We can bound the first term on the right hand side of \eqref{StateUniquenessEq2_} using \eqref{Cty:a1ma2} to obtain 
	\begin{align*}
		|(a(\boldsymbol{y}_1;\boldsymbol{u}_2,\boldsymbol{\alpha}) - a(\boldsymbol{y}_2;\boldsymbol{u}_2,\boldsymbol{\alpha}))| \leq			\gamma_{\nu} C_{p_{2_d}} C_{gn} \norm{\boldsymbol{\beta}}_{1,\Omega} \norm{ \boldsymbol{u}_2}_{3/2 + \delta,\Omega} \norm{\boldsymbol{\alpha}}_{1,\Omega}. 			
	\end{align*}
	The remaining terms on the right hand side of \eqref{StateUniquenessEq2_} can be bounded as follows:
	\begin{align*}
		|c(\boldsymbol{\alpha},\boldsymbol{u}_2,\boldsymbol{\alpha})|  &\leq C_{6_d} C_{3_d} \norm{\boldsymbol{\alpha}}^2_{1,\Omega} \norm{\nabla \boldsymbol{u}_2}_{0,\Omega} \leq C_{6_d} C_{3_d} \norm{\boldsymbol{\alpha}}^2_{1,\Omega} \norm{\boldsymbol{u}_2}_{1,\Omega}, \nonumber\\
		|d(\boldsymbol{y}_1, \boldsymbol{\alpha}) - d(\boldsymbol{y}_2, \boldsymbol{\alpha})| &= | \left(\boldsymbol{F}(\boldsymbol{y}_1) - \boldsymbol{F}(\boldsymbol{y}_2), \boldsymbol{\alpha}\right)| \leq \norm{\boldsymbol{F}(\boldsymbol{y}_1) - \boldsymbol{F}(\boldsymbol{y}_2)}_{0,\Omega} \norm{\boldsymbol{\alpha}}_{0,\Omega} \nonumber\\
		&\leq \gamma_{F} \norm{\boldsymbol{\beta}}_{0,\Omega} \norm{\boldsymbol{\alpha}}_{0,\Omega} \leq  \gamma_{F} \norm{\boldsymbol{\beta}}_{1,\Omega} \norm{\boldsymbol{\alpha}}_{1,\Omega}.
	\end{align*}
	Putting the above bounds in \eqref{StateUniquenessEq2_} results to
	\begin{align}
		\alpha_a \norm{\boldsymbol{\alpha}}^2_{1,\Omega} &\leq \gamma_{\nu} C_{p_{2_d}} C_{gn} \norm{\boldsymbol{\beta}}_{1,\Omega} \norm{ \boldsymbol{u}_2}_{3/2+\delta,\Omega} \norm{\boldsymbol{\alpha}}_{1,\Omega} + C_{6_d} C_{3_d} \norm{\boldsymbol{\alpha}}^2_{1,\Omega} \norm{\boldsymbol{u}_2}_{1,\Omega}  + \gamma_{F} \norm{\boldsymbol{\beta}}_{1, \Omega} \norm{\boldsymbol{\alpha}}_{1,\Omega} \nonumber\\
		&\leq \gamma_{\nu} C_{p_{2_d}} C_{gn} M \norm{\boldsymbol{\beta}}_{1,\Omega} \norm{\boldsymbol{\alpha}}_{1,\Omega} + C_{6_d} C_{3_d} M_{\boldsymbol{u}} \norm{\boldsymbol{\alpha}}^2_{1,\Omega}  + \gamma_{F} \norm{\boldsymbol{\beta}}_{1, \Omega} \norm{\boldsymbol{\alpha}}_{1,\Omega}. \label{eq2}
	\end{align}
	Analogously, putting $\boldsymbol{s} = \boldsymbol{\beta} \in [H_0^1(\Omega)]^2$ in second equation of \eqref{StateUniquenessEq1} and applying \eqref{Coer:ay} gives
	\begin{align}\label{StateUniquenessEq5}
		\hat{\alpha}_a \norm{\boldsymbol{\beta}}_{1,\Omega} &\leq |c_{\boldsymbol{y}}(\boldsymbol{\alpha},\boldsymbol{y}_2,\boldsymbol{\beta})| \leq C_{6_d} C_{3_d} M_{\boldsymbol{y}} \norm{\boldsymbol{\alpha}}_{1,\Omega}.
	\end{align}
	Substituting \eqref{StateUniquenessEq5} in \eqref{eq2} leads to the assertion that uniqueness of problem \eqref{P:Sred} holds under the largeness condition given in \eqref{StateUniquenessCondition}. 
	This completes the proof. 
\end{proof}

\begin{remark}\label{smalldata}
	The condition \eqref{StateUniquenessCondition} in Theorem \ref{StateUniqueness} implies a smallness assumption, that is, we can guarantee the uniqueness of a regular weak solution to \eqref{P:Sred}, if the data $\boldsymbol{y}^D$ is ``sufficiently small" or $\alpha_a$ is ``sufficiently large'' so that \eqref{StateUniquenessCondition} holds. 
\end{remark}

	\begin{remark}
		In this remark, we compare the continuous analysis developed in \cite{Agroum2015} for a closely related model problem.
		The existence of at least one weak solution in the natural energy spaces for admissible data is established using a fixed point theorem. The uniqueness result \cite[Proposition 2.3]{Agroum2015} still assumes a strong assumption that the solution itself satisfies $\| \boldsymbol{u} \|_{W^{1,q}(\Omega)} < 1$.
		 In \cite[Proposition 2.4]{Agroum2015}, the authors improve the regularity to $\boldsymbol{u} \in \boldsymbol{W}^{1,q}(\Omega)$ (for $2\le q\le q_0$) and $T \in W^{2,q}(\Omega)$ given the data $\boldsymbol{f} \in \boldsymbol{W}^{-1,q}(\Omega)$ and $T_D \in W^{2-1/q,q}(\Omega)$, with the admissible exponent range depending on the geometry of $\Omega$ (and improving for convex $\Omega$). In our case, the regularity result improves also the velocity itself to $\boldsymbol{u} \in \boldsymbol{H}^{3/2+\delta}(\Omega)$ (and to $\boldsymbol{u} \in \boldsymbol{H}^2(\Omega)$ if $\Omega$ is convex). The discrete analysis however is based on a weaker assumption \cite[Assumption 3.2]{Agroum2015} and local-uniqueness arguments. 
	\end{remark}

\section{Discretization}
In this section we propose a discretization of the governing equation \eqref{P:S} using nonconforming finite elements and study the well-posedness of resulting discrete system. To this aim, let $\mathcal{T}_h$ denote a partition of the domain $\Omega$ into a finite collection of open subdomains $K_i, \; i = 1, ..., N^{E}$ (triangles in 2D and tetrahedra  in 3D) such that:
$$\bar{\Omega} = \bigcup_{K_i \in \mathcal{T}_h} \bar{K}_{i} \;\; \mbox{and} \;\; K_i \cap K_j = \phi, \; i \neq j.$$
The mesh parameter $h$ associated with $\mathcal{T}_h$ is defined as $h = \max_{K \in \mathcal{T}_h} h_K,$ where $h_K$ denotes the diameter of $K$. We assume that $\mathcal{T}_h$ is shape regular in the sense that there exists a real number $\bar{\rho} > 0$, independent of $h$ and $K$ such that as $h \rightarrow 0$,
$$\frac{h_K}{\rho_K} \leq \bar{\rho}, \;\; \forall \;\; K \in \mathcal{T}_h,$$
where $\rho_K := \sup \left\{diam(\mathcal{B}): \mathcal{B} \; \mbox{is a ball contained in K}\right\}.$  Let $\mathcal{E}(\mathcal{T}_h)$ be the set of all edges in (2D) or the set of all faces in (3D), which is divided into $\mathcal{E}_I(\mathcal{T}_h)$ and $\mathcal{E}_B(\mathcal{T}_h)$, the set of all interior and boundary edges, respectively. For any  $e \in \mathcal{E}_I(\mathcal{T}_h)$, we label by $K^{+}$ and $K^{-}$  the elements adjacent to $e$. Furthermore, $h_e$ denotes the edge length or the maximum diameter of the facet in (2D) and (3D), respectively. Let $\boldsymbol{v}, w$ represent smooth vector and scalar fields, respectively, defined over $\mathcal{T}_h$. Then, by $(\boldsymbol{v}^+, \boldsymbol{v}^-)$ and $(w^+, w^-)$ we will denote the traces of $\boldsymbol{v}$ and $w$ on $e$ taken from inside $K^+$ and $K^-$, respectively. Similarly, we denote by $\boldsymbol{n}_e^+$ and $\boldsymbol{n}_e^-$ the outward unit normal vectors to $e$ corresponding to $K^+$ and $K^-$, respectively. The averages and jumps of $\boldsymbol{v}$ and $w$ across $e$ are defined, respectively, as follows 
\begin{align*}
	\avg{\boldsymbol{v}} := \frac{(\boldsymbol{v}^+ + \boldsymbol{v}^-)}{2}, \;\; \jump{\boldsymbol{v}} := \left(\boldsymbol{v}^- \cdot \boldsymbol{n}^-_e + \boldsymbol{v}^+ \cdot \boldsymbol{n}^+_e\right),  
	\avg{w} := \frac{(w^+ + w^-)}{2}, \;\; \jump{w} := \left(w^- \boldsymbol{n}^-_e + w^+  \boldsymbol{n}^+_e\right).
\end{align*}
The respective notations for boundary jumps and averages are, $\avg{\boldsymbol{v}} = \boldsymbol{v},  \jump{\boldsymbol{v}} = \boldsymbol{v} \cdot \boldsymbol{n}_{\Gamma},  \avg{w} = w, \; \jump{w} = w \; \boldsymbol{n}_{\Gamma}.$ In addition, $\nabla_h$ and $\bdiv_h$ denote the broken gradient and divergence operators, that is, the operators $\nabla$ and $\bdiv$ taken element wise. For any natural number $k \geq 1$, $\mathbb{P}_k(K)$ denotes the space of piecewise polynomials of degree less than or equal to $k$.

\subsection{Nonconforming method}\label{Sec:nc:S}
Consider the following lowest order Crouzeix-Raviart space on the partition $\mathcal{T}_h$,
\begin{align*}
	&\boldsymbol{V}^{cr}_h := \left\{ \boldsymbol{v}_h \in \boldsymbol{L}^2(\Omega) : \boldsymbol{v}_h |_K \in \left[\mathbb{P}_1(K)\right]^d, \boldsymbol{v}_h \; \mbox{is continuous at midpoint} \; \forall \; e \in \mathcal{E}_I(\mathcal{T}_h) \right\}, \\
	&\boldsymbol{V}^{cr_0}_h := \left\{ \boldsymbol{v}_h \in \boldsymbol{V}^{cr}_h : \boldsymbol{v}_h  \; \mbox{at midpoint is zero} \; \forall \; e \in \mathcal{E}_B(\mathcal{T}_h) \right\},
\end{align*}
and let $\mathcal{M}_h^{cr}$ and $\mathcal{M}_h^{cr_0}$ denote their counterparts for all $\boldsymbol{s}_h \in \left[L^2(\Omega)\right]^2$. Moreover, we also define the following discontinuous finite element space
$$Q_h^{k} := \left\{q_h \in L_0^2(\Omega) : q_h|_K \in \mathbb{P}_{k}(K) \;\; \forall \;\; K \in \mathcal{T}_h\right\}.$$
Corresponding to these finite dimensional spaces, we consider the following nonconforming finite element method which is based on an upwind discretization to handle the convective terms (cf. \cite{DS_upwind}).  Find $(\boldsymbol{u}_h,p_h,\boldsymbol{y}_h) \in \boldsymbol{V}^{cr_0}_h \times Q^0_h \times \mathcal{M}^{cr}_h$ such that $\boldsymbol{y}_h |_{\Gamma} = \boldsymbol{y}_h^D$ and 
\begin{align}\label{ncPh:S}
	\left\{
	\begin{aligned}
		a^{h}(\boldsymbol{y}_h; \boldsymbol{u}_h, \boldsymbol{v}_h) + c^{h}(\boldsymbol{u}_h, \boldsymbol{u}_h, \boldsymbol{v}_h) + b^{h}(\boldsymbol{v}_h,p_h) &= d^h(\boldsymbol{y}_h,\boldsymbol{v}_h)  \;\; \forall \;\; \boldsymbol{v}_h \in \boldsymbol{V}^{cr_0}_h, \\
		b^h(\boldsymbol{u}_h,q_h) &= 0 \;\; \forall \;\; q_h \in Q^0_h, \\
		a^h_{\boldsymbol{y}}(\boldsymbol{y}_h,\boldsymbol{s}_h) + c^h_{\boldsymbol{y}}(\boldsymbol{u}_h,\boldsymbol{y}_h,\boldsymbol{s}_h) &= 0 \;\; \forall \;\; \boldsymbol{s}_h \in \mathcal{M}^{cr_0}_h, 
	\end{aligned}
	\right.
\end{align}
with
\begin{align*}
	a^h(\boldsymbol{s}_h; \boldsymbol{u}_h, \boldsymbol{v}_h) &:= \int_{\Omega} \left(\boldsymbol{K}^{-1} \boldsymbol{u}_h \cdot \boldsymbol{v}_h + \nu(\boldsymbol{s}_h) \nabla_h \boldsymbol{u}_h : \nabla_h \boldsymbol{v}_h \right)  dx,\;\; a^h_{\boldsymbol{y}}(\boldsymbol{y}_h, \boldsymbol{s}_h) := \int_{\Omega} \boldsymbol{D}~ \nabla_h\boldsymbol{y}_h : \nabla_h\boldsymbol{s}_h ~dx, \\
	c^h(\boldsymbol{w}_h, \boldsymbol{u}_h, \boldsymbol{v}_h) &:= \int_{\Omega} \left(\boldsymbol{w}_h \cdot \nabla_h \right) \boldsymbol{u}_h  \cdot \boldsymbol{v}_h ~dx + \sum_{K \in \mathcal{T}_h} \int_{\partial K \backslash \Gamma} \hat{\boldsymbol{w}}^{up}_{h} (\boldsymbol{u}_h) \cdot \boldsymbol{v}_h ~ds, \\
	c^h_{\boldsymbol{y}}(\boldsymbol{w}_h, \boldsymbol{y}_h, \boldsymbol{s}_h) &:=  \int_{\Omega} \left(\boldsymbol{w}_h \cdot \nabla_h \right) \boldsymbol{y}_h \cdot \boldsymbol{s}_h ~dx + \sum_{K \in {\mathcal{T}_h}} \int_{\partial K \backslash \Gamma} \hat{\boldsymbol{w}}^{up}_{h} (\boldsymbol{y}_h) \cdot \boldsymbol{s}_h ~ds,
\end{align*}
where $b^h(\cdot,\cdot), \; d^h(\cdot,\cdot)$ allows for discontinuous elements and are defined analogously to $b(\cdot,\cdot), \; d(\cdot,\cdot)$, respectively. The fluxes are defined as follows:
$$\hat{\boldsymbol{w}}^{up}_h(\boldsymbol{u}_h) := \frac{1}{2} \left(\boldsymbol{w}_h \cdot  \boldsymbol{n}_K - |\boldsymbol{w}_h \cdot \boldsymbol{n}_K|\right)(\boldsymbol{u}_h^e - \boldsymbol{u}_h) \;\; \mbox{and} \;\; \hat{\boldsymbol{w}}^{up}_h(\boldsymbol{y}_h) := \frac{1}{2} \left(\boldsymbol{w}_h \cdot  \boldsymbol{n}_K - |\boldsymbol{w}_h \cdot \boldsymbol{n}_K|\right)(\boldsymbol{y}_h^e - \boldsymbol{y}_h),$$
and $\boldsymbol{u}_h^e$ and $\boldsymbol{y}_h^e$ are the traces of $\boldsymbol{u}_h$ and $\boldsymbol{y}_h = (T_h,S_h)^\top$, respectively, taken from within the exterior of $K$ over the edge under consideration and is zero on the boundary.  Without loss of generality, by $\boldsymbol{\Pi}_{nc},$ we denote the nonconforming interpolation operator from $\left[H^1(\Omega)\right]^2$ to $\mathcal{M}_h^{cr}$ ($\boldsymbol{H}^1(\Omega)$ to $\boldsymbol{V}_h^{cr}$) such that
\begin{align}\label{nonConf_interpolation}
	\boldsymbol{\Pi}_{nc} \boldsymbol{s}(m_e) = \frac{1}{|e|} \int_e \boldsymbol{s} \; ds \;\; \forall \;\; e \in \mathcal{E}(\mathcal{T}_h),
\end{align}
where $m_e$ denotes the mid-point of the edge $e$ (cf. \cite{m2an_discretelifting,AinsworthNonConf}). Let $$\boldsymbol{E}_h := \left\{\boldsymbol{s} \in \left[L^2(\Gamma)\right]^2\;:\; \boldsymbol{s}|_e \in \mathbb{P}_0(e) \; \forall \; e \in \mathcal{E}_B(\mathcal{T}_h) \right\},$$ denote the space of Lagrange multipliers on $\Gamma$. Following (\cite[Section 2.3]{m2an_discretelifting}), we use an analogue of $\boldsymbol{\Pi}_{nc}$ restricted to the boundary in order to define $\boldsymbol{y}_h^D := \boldsymbol{\Pi}^{\partial}_{nc} \boldsymbol{y}^D$, where $\boldsymbol{\Pi}^{\partial}_{nc} : \left[L^2(\Gamma)\right]^2 \longrightarrow \boldsymbol{E}_h$ such that
$$\boldsymbol{\Pi}^{\partial}_{nc} \boldsymbol{s}(m_e) = \frac{1}{|e|} \int_e \boldsymbol{s} \; ds \;\; \forall \; e \in \mathcal{E}_B(\mathcal{T}_h),$$ where, $|
e|$ denotes the Lebesgue measure of the edge $e$. Using an integration by parts, mid-point continuity on interior facets and the boundary condition we have that  $\bdiv_h \boldsymbol{V}_h^{cr_0} \subset \mathbb{P}_0(\mathcal{T}_h) \cap L_0^2(\Omega) =  Q_h^0$. Thus, we have the following discrete kernel corresponding to $b^h(\cdot,\cdot):$
$$\boldsymbol{X}^{cr}_h := \left\{\boldsymbol{v}_h \in \boldsymbol{V}^{cr_0}_h : b^h(\boldsymbol{v}_h, q_h) = 0 \;\; \forall \;\; q_h \in Q^0_h\right\} = \left\{\boldsymbol{v}_h \in \boldsymbol{V}^{cr_0}_h : \bdiv_h \boldsymbol{v}_h = 0\right\}.$$
 Consequently the  pair $(\boldsymbol{V}_h^{cr_0},Q_h^0)$ are structure preserving in the sense that they preserve the incompressibility constraint locally (elementwise).
\subsection{Discrete boundedness properties}\label{Sec:DSP}
The space $\boldsymbol{V}_h^{cr_0}$ is a subset of the piecewise Sobolev space
$$\boldsymbol{H}^1(\mathcal{T}_h) := \left\{\boldsymbol{v} \in \boldsymbol{L}^2(\Omega) : \boldsymbol{v}|_{K} \in \boldsymbol{H}^1(K) \;\; \forall \;\; K \in \mathcal{T}_h\right\}.$$
Similarly, $\mathcal{M}_h^{cr}$ is also a subset of $\left[H^1(\mathcal{T}_h)\right]^2$ which is defined analogously for all $\boldsymbol{s}_h \in \left[L^2(\Omega)\right]^2$. Now we define the following mesh and parameter dependent norms over $\boldsymbol{V}_h^{cr_0}$ and $\mathcal{M}^{cr}_h$ as follows:
\begin{align*}
	\norm{\boldsymbol{v}_h}^2_{1,\mathcal{T}_h^{nc}} := \sum_{K \in {\mathcal{T}_h}} \sigma \norm{\boldsymbol{v}_h}^2_{0,K} + &\nu_2 \norm{\nabla \boldsymbol{v}_h}^2_{0,K}, \; \forall \; \boldsymbol{v}_h \in \boldsymbol{H}^1({\mathcal{T}_h}), \; 
	\norm{\boldsymbol{s}_h}^2_{1,\mathcal{T}_h^{nc}} := \sum_{K \in {\mathcal{T}_h}} \bar{\sigma} \norm{\nabla \boldsymbol{s}_h}^2_{0,K}, \; \forall \; \boldsymbol{s}_h \in \left[H^1(\mathcal{T}_h)\right]^2,
\end{align*}
where $\sigma = \norm{\boldsymbol{K}^{-1}}_{\infty,\Omega}$, $\bar{\sigma} = \norm{\boldsymbol{D}}_{\infty, \Omega}$ and $\nu_2$ is as defined in section \ref{Intro}.  In the following lemma, we state and prove the continuity properties of the discrete bilinear and trilinear forms.
\begin{lemma}\label{ncStateCty}
	There holds
	\begin{align*}
		|a^h(\boldsymbol{y}_h; \boldsymbol{u}_h, \boldsymbol{v}_h)| &\leq C_a^{nc} \norm{\boldsymbol{u}_h}_{1,\mathcal{T}_h^{nc}} \norm{\boldsymbol{v}_h}_{1,\mathcal{T}_h^{nc}} \;\; \forall \;\; \boldsymbol{u}_h, \boldsymbol{v}_h \in \boldsymbol{V}_h^{cr_0}, \\ 
		|a^h_{\boldsymbol{y}}(\boldsymbol{y}_h, \boldsymbol{s}_h)| &\leq \hat{C}_a^{nc} \norm{\boldsymbol{y}_h}_{1,\mathcal{T}_h^{nc}} \norm{\boldsymbol{s}_h}_{1,\mathcal{T}_h^{nc}} \;\; \forall \;\; \boldsymbol{y}_h \in \mathcal{M}_h^{cr}, \boldsymbol{s}_h \in \mathcal{M}^{cr_0}_h, \\
		|c^h(\boldsymbol{w}_h, \boldsymbol{u}_h, \boldsymbol{v}_h)| &\leq C_{e}^{nc} \norm{\boldsymbol{w}_h}_{1,\mathcal{T}_h^{nc}} \norm{\boldsymbol{u}_{h}}_{1,\mathcal{T}_h^{nc}} \norm{\boldsymbol{v}_h}_{1,\mathcal{T}_h^{nc}} \;\; \forall \;\; \boldsymbol{w}_h, \boldsymbol{u}_h, \boldsymbol{v}_h \in \boldsymbol{V}_h^{cr_0} \\
		|c^h_{\boldsymbol{y}}(\boldsymbol{w}_h, \boldsymbol{y}_h, \boldsymbol{s}_h)| &\leq C_{e}^{nc} \norm{\boldsymbol{w}_h}_{1,\mathcal{T}_h^{nc}} \norm{\boldsymbol{y}_{h}}_{[L^4(\Omega)]^2} \norm{\boldsymbol{s}_h}_{1,\mathcal{T}_h^{nc}} \;\; \forall \;\; \boldsymbol{w_h} \in \boldsymbol{V}_h^{cr_0}, \boldsymbol{y}_h \in \mathcal{M}_h^{cr}, \boldsymbol{s}_h \in \mathcal{M}_h^{cr_0}, \\
		|b^h(\boldsymbol{v}_h, q_h)| &\leq C_d \norm{\boldsymbol{v}_h}_{1,\mathcal{T}_h^{nc}} \norm{q_h}_{0,\Omega} \;\; \forall \;\; \boldsymbol{v}_h \in \boldsymbol{H}^1(\mathcal{T}_h), q_h \in L_0^2(\Omega).
	\end{align*}
\end{lemma}
\begin{proof}
	We begin by establishing the continuity properties of the discrete bilinear form $a^h$. The element-wise contributions can be bounded for all $\boldsymbol{v}_h \in \boldsymbol{V}_h^{cr_0}$ by the mesh-dependent norms as follows: 
	\begin{align*}
		a^h(\boldsymbol{y}_h; \boldsymbol{u}_h, \boldsymbol{v}_h) &\leq \sum_{K \in {\mathcal{T}_h}} \sigma \norm{\boldsymbol{u}_h}_{0,K} \norm{\boldsymbol{v}_h}_{0,K} + \sum_{K \in {\mathcal{T}_h}} \nu_2 \norm{\nabla_h \boldsymbol{u}_h}_{0,K} \norm{\nabla_h \boldsymbol{v}_h}_{0,K} \leq \hat{C}_a^{nc} \norm{\boldsymbol{u}_h}_{1,\mathcal{T}_h^{nc}} \norm{\boldsymbol{v}_h}_{1,\mathcal{T}_h^{nc}}.
	\end{align*}
	Similarly, we can derive the desired bound for $a^h_{\boldsymbol{y}}$. To bound the discrete trilinear forms $c^h$, we utilize the discrete Sobolev embedding for the nonconforming setting (cf. \cite{Buffa_IMA, Karakashian_SINUM}): For $r \in [1, \infty)$ if $d = 2$ and $r \in [1,6]$ if $d = 3$ there exists a constant $C_{e}^{nc} > 0$ such that
	\begin{align}\label{DiscSobEmbnc}
		\norm{\boldsymbol{v}}_{\boldsymbol{L}^r(\Omega)} \leq C_{e}^{nc} \norm{\boldsymbol{v}}_{1,\mathcal{T}_h^{nc}} \;\; \forall \;\; \boldsymbol{v} \in \boldsymbol{H}^1(\mathcal{T}_h).
	\end{align}
	Indeed, the elementwise contribution of the convective term can be bounded using the H\"older's inequality and \eqref{DiscSobEmbnc} as follows:
	\begin{align*}
		\Big|\sum_{K \in {\mathcal{T}_h}} \int_K (\boldsymbol{w}_h \cdot \nabla \boldsymbol{u}_h) \cdot \boldsymbol{v}_h \; dx\Big| &\leq \sum_{K \in {\mathcal{T}_h}}  \norm{\boldsymbol{w}_h}_{\boldsymbol{L}^4(K)} \norm{\nabla \boldsymbol{u}_h}_{\boldsymbol{L}^2(K)} \norm{\boldsymbol{v}_h}_{\boldsymbol{L}^4(K)} \\
		&\leq \Big\{\sum_{K \in {\mathcal{T}_h}} \norm{\boldsymbol{w}_h}_{\boldsymbol{L}^4(K)}^4\Big\}^{1/4} \Big\{\sum_{K \in {\mathcal{T}_h}} \norm{\boldsymbol{v}_h}_{\boldsymbol{L}^4(K)}^4\Big\}^{1/4} \Big\{\sum_{K \in {\mathcal{T}_h}} \norm{\nabla \boldsymbol{u}_h}_{\boldsymbol{L}^2(K)}^2\Big\}^{1/2}  \\
		&\leq \norm{\boldsymbol{w}_h}_{\boldsymbol{L}^4(\Omega)} \norm{\boldsymbol{v}_h}_{\boldsymbol{L}^4(\Omega)} \norm{\boldsymbol{u}_h}_{1,\mathcal{T}_h^{nc}}  \leq C_e^{nc} \norm{\boldsymbol{w}_h}_{1,\mathcal{T}_h^{nc}} \norm{\boldsymbol{v}_h}_{1,\mathcal{T}_h^{nc}} \norm{\boldsymbol{u}_h}_{1,\mathcal{T}_h^{nc}}.
	\end{align*}
	Using the inverse inequality,
	\begin{align}\label{inverseEst}
		\norm{\boldsymbol{v}}_{0,\partial K} \leq C_{T} h^{-\frac{1}{2}}_K \norm{\boldsymbol{v}}_{0,K} \;\; \forall \;\; \boldsymbol{v} \in \mathbb{P}_k(K),
	\end{align}
	with $C_T$ being independent of the mesh parameter $h$, and following similar steps as above, we can show the following bound on the discrete upwind term:
	\begin{align*}
		\Big|\sum_{K \in \mathcal{T}_h} \int_{\partial K \backslash \Gamma} \hat{\boldsymbol{w}}^{up}_{h} (\boldsymbol{u}_h) \cdot \boldsymbol{v}_h ~ds \Big| \leq C_e^{nc} \norm{\boldsymbol{w}_h}_{1,\mathcal{T}_h^{nc}} \norm{\boldsymbol{u}_h}_{1,\mathcal{T}_h^{nc}} \norm{\boldsymbol{v}_h}_{1,\mathcal{T}_h^{nc}}.
	\end{align*}
	Combining the above two  bounds leads to desired bound for $c^h$. To derive the bound for $c^h_{\boldsymbol{y}}$, we integrate it by parts 
	$$c^h_{\boldsymbol{y}}(\boldsymbol{w}_h, \boldsymbol{y}_h, \boldsymbol{s}_h) :=  -\sum_{K \in \mathcal{T}_h} \int_{K} \left(\boldsymbol{w}_h \cdot \nabla_h \right) \boldsymbol{s}_h \cdot \boldsymbol{y}_h ~dx + \sum_{K \in {\mathcal{T}_h}} \int_{\partial K \backslash \Gamma} \hat{\boldsymbol{w}}^{up}_{h} (\boldsymbol{s}_h) \cdot \boldsymbol{y}_h ~ds,$$
	and apply the same procedure used to get the bounds for $c_h$.
\end{proof}

In addition, for all $\boldsymbol{w}_h \in \boldsymbol{X}_h^{cr}$, an integration by parts of the convective term and the fact that the discrete velocities obtained are exactly divergence free, yield the following positivity properties,
\begin{align}
	c^h(\boldsymbol{w}_h, \boldsymbol{u}_h, \boldsymbol{u}_h) = \frac{1}{2} \sum_{e \in \mathcal{E}(\mathcal{T}_h)} \int_{e} |\boldsymbol{w}_h \cdot \boldsymbol{n}_e| \; |\jump{\boldsymbol{u}_h}|^2 \; ds \geq 0\;\; \forall \;\; \boldsymbol{u}_h \in \boldsymbol{V}_h^{cr_0}, \label{Positivitync:ch}\\
	c^h_{\boldsymbol{y}}(\boldsymbol{w}_h, \boldsymbol{s}_h, \boldsymbol{s}_h) = \frac{1}{2} \sum_{e \in \mathcal{E}(\mathcal{T}_h)} \int_{e} |\boldsymbol{w}_h \cdot \boldsymbol{n}_e| \; |\jump{\boldsymbol{s}_h}|^2 \; ds \geq 0 \;\; \forall \;\; \boldsymbol{s}_h \in \mathcal{M}_h^{cr_0}. \label{Positivitync:chy}
\end{align}
Moreover, the following coercivity properties of $a^h$ and $a^h_{\boldsymbol{y}}$ also hold:
\begin{align}\label{nc:Coercivity}
	a^h(\cdot; \boldsymbol{u}_h, \boldsymbol{u}_h) \geq \alpha_a^{nc} \norm{\boldsymbol{v}_h}^2_{1,\mathcal{T}_h^{nc}} \;\; \forall \;\; \boldsymbol{v}_h \in \boldsymbol{V}_h^{cr_0} \;\; \mbox{and} \;\;
	a^h_{\boldsymbol{y}}(\boldsymbol{s}_h, \boldsymbol{s}_h) \geq \alpha_2^{nc} \norm{\boldsymbol{s}_h}^2_{1,\mathcal{T}_h^{nc}} \;\; \forall \;\; \boldsymbol{s}_h \in \mathcal{M}_h^{cr_0},
\end{align}
and $b^h(\cdot,\cdot)$ satisfies the following discrete inf-sup condition for all $q_h \in Q_h^0:$
\begin{align}\label{discreteinfsup}
	\sup_{\boldsymbol{v}_h \in \boldsymbol{V}_h^{cr_0} \backslash \{\boldsymbol{0}\}} \frac{b^h(\boldsymbol{v}_h, q_h)}{\norm{\boldsymbol{v}_h}_{1,\mathcal{T}_h^{nc}}} \geq \tilde{\beta} \norm{q_h}_{0,\Omega}, \;\; \mbox{for some} \; \tilde{\beta} > 0 \; \mbox{which is independent of}\; h.
\end{align}
\section{Discrete well-posedness}
\subsection{Existence of discrete solutions}\label{Sec:ncExistenceState}
The following discrete analogue of Lemma \ref{State:equivalence} holds:
\begin{lemma}\label{Discrete:State:equivalence}
	If $(\boldsymbol{u}_h,p_h,\boldsymbol{y}_h) \in \boldsymbol{V}_h^{cr_0} \times Q_h^0 \times \mathcal{M}_h^{cr}$ solves \eqref{ncPh:S}, $(\boldsymbol{u}_h,\boldsymbol{y}_h) \in \boldsymbol{X}_h^{cr} \times \mathcal{M}_h^{cr}$ satisfies $\boldsymbol{y}_h|_{\Gamma} = \boldsymbol{y}_h^D$ and
	\begin{align}\label{ncPh:Sred}
		\left\{
		\begin{aligned}
			a^h(\boldsymbol{y}_h;\boldsymbol{u}_h,\boldsymbol{v}_h) + c^h(\boldsymbol{u}_h, \boldsymbol{u}_h, \boldsymbol{v}_h) - d^h(\boldsymbol{y}_h, \boldsymbol{v}_h) &= 0 \;\; \forall \;\; \boldsymbol{v}_h \in \boldsymbol{X}_h^{cr},\\
			a^h_{\boldsymbol{y}}(\boldsymbol{y}_h,\boldsymbol{s}_h) + c^h_{\boldsymbol{y}}(\boldsymbol{u}_h, \boldsymbol{y}_h, \boldsymbol{s}_h) &= 0 \;\; \forall \;\; \boldsymbol{s}_h \in \mathcal{M}_h^{cr_0}.
		\end{aligned}
		\right.
	\end{align}
	Conversely, if $(\boldsymbol{u}_h,\boldsymbol{y}_h) \in \boldsymbol{X}_h^{cr} \times \mathcal{M}_h^{cr}$ is a solution of the reduced problem \eqref{ncPh:Sred}, then there exists $p_h \in Q_h^0$ such that $(\boldsymbol{u}_h, p_h ,\boldsymbol{y}_h)$ is a solution of \eqref{ncPh:S}.
\end{lemma}

A discrete boundary lifting of $\boldsymbol{y}_h$ is performed by setting 
\begin{align}\label{discreteLifitingdefn}
	\boldsymbol{y}_h = \boldsymbol{y}_{h,0} + \boldsymbol{y}_{h,1} \; \mbox{with} \; \boldsymbol{y}_{h,0} \in \mathcal{M}_{h}^{cr_0} \; \mbox{such that} \; \boldsymbol{y}_{h,1} \in \mathcal{M}_h^{cr} \; \mbox{and} \; \boldsymbol{y}_{h,1}|_{\Gamma} = \boldsymbol{y}_h^D.
\end{align}
Since trace of a nonconforming function to the boundary does not necessarily belong to $[{H}^{1/2}(\Gamma)]^2$, a discrete lifting theorem was proved in \cite{m2an_discretelifting} using a discrete $[{H}^{1/2}(\Gamma)]^2$-norm combined with an enriching operator (see \cite[Appendix B]{Brenner_EnrichmentOperator}) to convert the nonconforming approximation to a conforming one. The combined results stating compatibility of discrete $[{H}^{1/2}(\Gamma)]^2$-norm with the usual $[{H}^{1/2}(\Gamma)]^2$-norm and its relation with the $\norm{\cdot}_{1,\mathcal{T}_h^{nc}}$ norm is stated  in the following lemma (see \cite[Lemmas 3.1 and 3.4]{m2an_discretelifting}):

\begin{lemma}\label{discreteLifiting}
	Let $\boldsymbol{g} \in \big[H^{\frac{1}{2}}(\Gamma)\big]^2$ then for $\boldsymbol{\Pi}_{nc}^{\partial} \boldsymbol{g} \in \boldsymbol{E}_h$, there exists $\boldsymbol{s}_h \in \mathcal{M}_h^{cr}$ satisfying $\boldsymbol{s}_h(m_e) = \boldsymbol{\Pi}_{nc}^{\partial} \boldsymbol{g}(m_e)$	and
	$$\|\boldsymbol{s}_h\|_{1,\mathcal{T}_h^{nc}} \leq C \|\boldsymbol{\Pi}_{nc}^{\partial} \boldsymbol{g}\|_{1/2, \boldsymbol{E}_h} \leq C \norm{\boldsymbol{g}}_{1/2, \Gamma},$$
	where $C$ denotes a generic positive constant  independent of the mesh parameter $h$.
\end{lemma}
\begin{lemma}\label{ncEnergyEstState}
	Let $(\boldsymbol{u}_h, \boldsymbol{y}_h)$ be a solution of \eqref{ncPh:Sred} with the discrete boundary lifting as defined in \eqref{discreteLifitingdefn}. Then there exist positive constants $C_{\boldsymbol{u}}^{nc}, C_{\boldsymbol{y}}^{nc}$ independent of the mesh parameter $h$, such that
	$$\norm{\boldsymbol{u}_h}_{1,\mathcal{T}_h^{nc}} \leq C_{\boldsymbol{u}}^{nc} \norm{\boldsymbol{y}^D}_{1/2,\Gamma} := M_{\boldsymbol{u}}^{nc}, \;\;\;\; \norm{\boldsymbol{y}_{h,0}}_{1,\mathcal{T}_h^{nc}} \leq C_{\boldsymbol{y}}^{nc} \norm{\boldsymbol{y}^D}_{1/2,\Gamma}, \mbox{and}$$
	$$\norm{\boldsymbol{y}_h}_{1,\mathcal{T}_h^{nc}} \leq C_{\boldsymbol{y}}^{nc} \norm{\boldsymbol{y}^D}_{1/2,\Gamma} := M_{\boldsymbol{y}}^{nc}.$$
\end{lemma}
\begin{proof}
	Choose $(\boldsymbol{v}_h, \boldsymbol{s}_h) = (\boldsymbol{u}_h, \boldsymbol{y}_{h,0})$ in \eqref{ncPh:Sred} and use \eqref{Positivitync:ch} and \eqref{Positivitync:chy} to rewrite it as follows:
	\begin{align*}
		a^h(\boldsymbol{y}_h;\boldsymbol{u}_h,\boldsymbol{u}_h) &=  d^h(\boldsymbol{y}_h, \boldsymbol{u}_h) \;\; \forall \;\; \boldsymbol{u}_h \in \boldsymbol{X}_h^{cr}, \\
		a^h_{\boldsymbol{y}}(\boldsymbol{y}_{h,0},\boldsymbol{y}_{h,0}) &=  - a^h_{\boldsymbol{y}}(\boldsymbol{y}_{h,1},\boldsymbol{y}_{h,0}) - c^h_{\boldsymbol{y}}(\boldsymbol{u}_h, \boldsymbol{y}_{h,1}, \boldsymbol{y}_{h,0}) \;\; \forall \;\; \boldsymbol{y}_{h,0} \in \mathcal{M}_h^{cr_0}.		
	\end{align*} 
	Due to the coercivity of the bilinear forms in \eqref{nc:Coercivity}, continuity properties stated in Lemma \ref{ncStateCty}, we have
	\begin{align}
		\alpha_a^{nc} \norm{\boldsymbol{u}_h}_{1,\mathcal{T}_h^{nc}} &\leq C_F C_e^{nc}  \sum_{K \in {\mathcal{T}_h}} \big(\norm{\boldsymbol{y}_{h,0}}_{0,K} + \norm{\boldsymbol{y}_{h,1}}_{0,K}\big)  \label{ncEnergyEstEq1}\\
		&\leq C_F C_e^{nc}  \big( \norm{ \boldsymbol{y}_{h,0}}_{1,\mathcal{T}_h^{nc}} + \norm{\boldsymbol{y}_{h,1}}_{1,\mathcal{T}_h^{nc}}\big), \nonumber\\
		\alpha_2^{nc} \norm{\boldsymbol{y}_{h,0}}_{1,\mathcal{T}_h^{nc}} &\leq C^{nc}_e   \norm{\boldsymbol{u}_h}_{1,\mathcal{T}_h^{nc}} \norm{\boldsymbol{y}_{h,1}}_{[L^4(\Omega)]^2} + \hat{C}_a^{nc} \norm{\boldsymbol{y}_{h,1}}_{1,\mathcal{T}_h^{nc}}. \label{ncEnergyEstEq2}
	\end{align}
	Put the bounds of \eqref{ncEnergyEstEq2} in \eqref{ncEnergyEstEq1} and apply Lemma \ref{IntermediateLemma} which still holds for the discrete case with
	\begin{align}
		\frac{C_F (C_e^{nc})^2}{\alpha_a^{nc} \alpha_2^{nc}} \norm{\boldsymbol{y}_{h,1}}_{[L^4(\Omega)]^2} < a < 1, \;\; \mbox{to get} \label{lifting_dataSmallnessCondition}
	\end{align} 
	\begin{align}
		\norm{\boldsymbol{u}_h}_{1,\mathcal{T}_h^{nc}} \leq \frac{C_F C_e^{nc}}{(1-a)\alpha_a^{nc}} \left(1 + \frac{C_a^{nc}}{\alpha_2^{nc}}\right) \|\boldsymbol{y}_{h,1}\|_{1, \mathcal{T}_h^{nc}}. \label{ncEnergyEstEq3}
	\end{align}
	Now an application of Lemma \ref{discreteLifiting} on \eqref{ncEnergyEstEq3} and substituting the resulting bound in \eqref{ncEnergyEstEq2} leads to the desired bound for $\boldsymbol{y}_{h,0}$.
\end{proof}
We make use of the following proposition which is a consequence of the Brouwer fixed point theorem to show the existence of discrete solutions \cite{temam2001navier}:
\begin{proposition}\label{ConseqOfBFPT}
	Let H be a finite dimensional Hilbert space with norm $|\cdot|_H$ and inner product $(\cdot,\cdot)_H$. Let $R: H \longrightarrow H$ be a continuous map such that there exists a positive constant $\rho$ satisfying
	$$(R(v),v)_H > 0 \;\; \forall \; v \in H \;\; \mbox{with} \;\; |v|_H = \rho.$$
	Then there exists a $w \in H$ such that $R(w) = 0, |w|_H \leq \rho$. 
\end{proposition}
\begin{theorem}\label{ncExistenceState}
	Let $\boldsymbol{y}_{h,1}$ be the discrete lifting of $\boldsymbol{y}_h$ as defined in Section \ref{Sec:ncExistenceState}. Then there exists a discrete solution $(\boldsymbol{u}_h, \boldsymbol{y}_h) \in \boldsymbol{X}_h^{cr} \times \mathcal{M}_h^{cr}$ to \eqref{ncPh:Sred} satisfying the bounds of Lemma \ref{ncEnergyEstState}.
\end{theorem}
\begin{proof}
	Define the maps $Q^{\boldsymbol{u}}$ and $Q^{\boldsymbol{y}}$ on $\boldsymbol{X}_h^{cr}$ and $\mathcal{M}_h^{cr_0}$ for all $\boldsymbol{v}_h \in \boldsymbol{X}_h^{cr}$ and $\boldsymbol{s}_h \in \mathcal{M}_h^{cr_0}$, respectively, as follows:
	\begin{align*}
		(Q^{\boldsymbol{u}}(\boldsymbol{u}_h), \boldsymbol{v}_h)
		&:= a^h(\boldsymbol{y}_h; \boldsymbol{u}_h, \boldsymbol{v}_h)
		+ c^h(\boldsymbol{u}_h, \boldsymbol{u}_h, \boldsymbol{v}_h)
		- d^h(\boldsymbol{y}_h, \boldsymbol{v}_h),\\
		(Q^{\boldsymbol{y}}(\boldsymbol{y}_{h,0}), \boldsymbol{s}_h)
		&:= a^h_{\boldsymbol{y}}(\boldsymbol{y}_{h,0}, \boldsymbol{s}_h)
		+ c^h_{\boldsymbol{y}}(\boldsymbol{u}_h, \boldsymbol{y}_{h,0}, \boldsymbol{s}_h)
		- a^h_{\boldsymbol{y}}(\boldsymbol{y}_{h,1}, \boldsymbol{s}_h)
		- c^h_{\boldsymbol{y}}(\boldsymbol{u}_h, \boldsymbol{y}_{h,1}, \boldsymbol{s}_h).
	\end{align*}
	The Continuity of $Q^{\boldsymbol{u}}$ and $Q^{\boldsymbol{y}}$ readily follows by using the boundedness properties of Lemma \ref{ncStateCty}. Let $Q^{\boldsymbol{u},\boldsymbol{y}} := Q^{\boldsymbol{u}} \times Q^{\boldsymbol{y}}$ with $\norm{(\boldsymbol{u}_h, \boldsymbol{y}_{h,0})} = \norm{\boldsymbol{u}_h}_{\boldsymbol{X}_h^{cr}} + \norm{\boldsymbol{y}_{h,0}}_{\mathcal{M}_h^{cr_0}}.$ Then the continuity of $Q^{\boldsymbol{u}, \boldsymbol{y}}$ follows from the continuity of maps  $Q^{\boldsymbol{u}}$ and $Q^{\boldsymbol{y}}$ and the definition of $Q^{\boldsymbol{u}, \boldsymbol{y}}$. In order to make use of Proposition \ref{ConseqOfBFPT}, which would imply that $Q^{\boldsymbol{u,\boldsymbol{y}}}$ has a fixed point, it remains to show the coercivity of $Q^{\boldsymbol{u},\boldsymbol{y}}$. Since $\boldsymbol{u}_h \in \boldsymbol{X}_h^{cr}$ and thanks to \eqref{nc:Coercivity} and \eqref{Positivitync:ch}, the form $a^h(\cdot; \cdot, \cdot) + c^h(\cdot, \cdot, \cdot)$ is coercive, and on using the boundedness properties and Lemma \ref{ncEnergyEstState}, we get
	\begin{align*}
		(Q^{\boldsymbol{u}}(\boldsymbol{u}_h), \boldsymbol{u}_h) &\geq \alpha_a^{nc} \norm{\boldsymbol{u}_h}^2_{1,\mathcal{T}_h^{nc}} - C_F C_e^{nc} \norm{\boldsymbol{y}_h}_{1,\mathcal{T}_h^{nc}} \norm{\boldsymbol{u}_h}_{1,\mathcal{T}_h^{nc}} \nonumber\\
		&\geq \alpha_a^{nc} \norm{\boldsymbol{u}_h}^2_{1,\mathcal{T}_h^{nc}} - C_F C_e^{nc} M_{\boldsymbol{y}}^{nc} M_{\boldsymbol{u}}^{nc}. 
	\end{align*}
	It follows that $(Q^{\boldsymbol{u}}(\boldsymbol{u}_h), \boldsymbol{u}_h) > 0$ for $\norm{\boldsymbol{u}_h}_{1,\mathcal{T}_h^{nc}}^2 = \kappa_1^{nc}$, and $\kappa_1^{nc}$ is sufficiently large: more precisely
	$$\kappa_1^{nc} > \left(\frac{1}{\alpha_a^{nc}} \left(C_e^{nc} M_{\boldsymbol{u}}^{nc} C_F M_{\boldsymbol{y}}^{nc}\right)\right).$$
	Similarly, invoking \eqref{nc:Coercivity} and \eqref{Positivitync:chy}, the form $a^h_{\boldsymbol{y}}(\cdot;\cdot,\cdot) + c^h_{\boldsymbol{y}}(\cdot,\cdot,\cdot)$ is coercive, yielding
	\begin{align*}
		(Q^{\boldsymbol{y}}(\boldsymbol{y}_{h,0}), \boldsymbol{y}_{h,0}) &\geq \alpha_2^{nc} \norm{\boldsymbol{y}_{h,0}}^2_{1,\mathcal{T}_h^{nc}} - \hat{C}_a^{nc} \norm{\boldsymbol{y}_{h,1}}_{1,\mathcal{T}_h^{nc}} \norm{\boldsymbol{y}_{h,0}}_{1,\mathcal{T}_h^{nc}} - C_e^{nc} \norm{\boldsymbol{u}_h}_{1,\mathcal{T}_h^{nc}} \norm{\boldsymbol{y}_{h,1}}_{1,\mathcal{T}_h^{nc}} \norm{\boldsymbol{y}_{h,0}}_{1,\mathcal{T}_h^{nc}} \\
		&\geq \alpha_2^{nc} \norm{\boldsymbol{y}_{h,0}}^2_{1,\mathcal{T}_h^{nc}} -  \hat{C}_a^{nc} C_{\boldsymbol{y}}^{nc} \norm{\boldsymbol{y}_{h,1}}_{1,\mathcal{T}_h^{nc}}^2-   C_e^{nc} M_{\boldsymbol{u}}^{nc}  C_{\boldsymbol{y}}^{nc} \norm{\boldsymbol{y}_{h,1}}_{1,\mathcal{T}_h^{nc}}^2.
	\end{align*}
	After an application of Lemma \ref{discreteLifiting}, we conclude that $(Q^{\boldsymbol{y}}(\boldsymbol{y}_{h,0}), \boldsymbol{y}_{h,0}) > 0$ for $\norm{\boldsymbol{y}_{h,0}}_{1,\mathcal{T}_h^{nc}}^2 = \kappa_2^{nc}$, and $\kappa_2^{nc}$ is sufficiently large: more precisely
	$$\kappa_2^{nc} > \left\{\frac{1}{\alpha_2^{nc}} \left(\hat{C}_a^{nc} C_{\boldsymbol{y}}^{nc}(1+C_e^{nc} M_{\boldsymbol{u}}^{nc})\norm{\boldsymbol{y}^D}_{1/2,\Gamma}^2\right)\right\}.$$
	From the coercivity of $Q^{\boldsymbol{y}}$ and $Q^{\boldsymbol{u}}$ and the definition of $Q^{\boldsymbol{u}, \boldsymbol{y}}$, we deduce that $(Q^{\boldsymbol{u}, \boldsymbol{y}}(\boldsymbol{u}_h, \boldsymbol{y}_{h,0}),(\boldsymbol{u}_h, \boldsymbol{y}_{h,0})) > 0$ for $\norm{(\boldsymbol{u}_h, \boldsymbol{y}_{h,0})} = \kappa_1^{nc} + \kappa_2^{nc} > 0.$
\end{proof}

\subsection{Uniqueness of discrete solutions}
\begin{theorem}\label{ncUniqueness}
	Let $(\boldsymbol{u}_h,\boldsymbol{y}_h) \in \boldsymbol{X}_h^{cr} \cap \boldsymbol{W}^{1,3}(\mathcal{T}_h) \times \mathcal{M}_h^{cr}$ be a solution of the reduced problem \eqref{ncPh:Sred}, and assume that
	\begin{align}\label{ncUniquenessAssumption}
		\norm{\nabla_h \boldsymbol{u}_h}_{L^3(\Omega)} \leq \tilde{M}, 
	\end{align}
	where $\tilde{M}>0$ is chosen such that following smallness assumption (see Remark \ref{smalldata})
	\begin{align}\label{Disc_StateUniquenessCondition}
		\alpha_a^{nc} > \frac{C_e^{nc} M^{nc}_{\boldsymbol{y}}}{\alpha_2^{nc}} \left(\gamma_{\nu} \tilde{M} + C_{F} \right) + C_e^{nc}  M_{\boldsymbol{u}}^{nc}
	\end{align}
	holds. If the condition \eqref{Disc_StateUniquenessCondition} is satisfied, then the solution of  \eqref{ncPh:Sred} is unique. 
\end{theorem}
\begin{proof}
	Let $(\boldsymbol{u}_h^a, \boldsymbol{y}_h^a)$ and $(\boldsymbol{u}_h^b,\boldsymbol{y}_h^b)$ solve \eqref{ncPh:Sred} and let $(\boldsymbol{\alpha}_h,\boldsymbol{\beta}_h) := (\boldsymbol{u}_h^a-\boldsymbol{u}_h^b, \boldsymbol{y}_h^a-\boldsymbol{y}_h^b)$. Then subtracting the corresponding variational formulations for all $\boldsymbol{v}_h \in \boldsymbol{X}_h^{cr}$, $\boldsymbol{s}_h \in \mathcal{M}_h^{cr_0}$, we have	
	\begin{align}\label{ncStateUniquenessEq1}
		\begin{aligned}
			a^h(\boldsymbol{y}_h^a;\boldsymbol{u}_h^a,\boldsymbol{v}_h) - a^h(\boldsymbol{y}_h^b;\boldsymbol{u}_h^b,\boldsymbol{v}_h) + c^h(\boldsymbol{u}_h^a,\boldsymbol{u}_h^a,\boldsymbol{v}_h) - c^h(\boldsymbol{u}_h^b, \boldsymbol{u}_h^b,\boldsymbol{v}_h) &= d^h(\boldsymbol{y}_h^a,\boldsymbol{v}_h) - d(\boldsymbol{y}_h^b,\boldsymbol{v}_h), \\
			a^h_{\boldsymbol{y}}(\boldsymbol{y}_h^a,\boldsymbol{s}_h) - a^h_{\boldsymbol{y}}(\boldsymbol{y}_h^b,\boldsymbol{s}_h) + c^h_{\boldsymbol{y}}(\boldsymbol{u}_h^a,\boldsymbol{y}_h^a,\boldsymbol{s}_h) - c^h_{\boldsymbol{y}}(\boldsymbol{u}_h^b,\boldsymbol{y}_h^b,\boldsymbol{s}_h) &= 0.
		\end{aligned}
	\end{align}	
	We can rewrite the terms in \eqref{ncStateUniquenessEq1}  as
	\begin{align*}
		a^h(\boldsymbol{y}_h^a;\boldsymbol{u}_h^a,\boldsymbol{v}_h) - a^h(\boldsymbol{y}_h^b;\boldsymbol{u}_h^b,\boldsymbol{v}_h) &= a^h(\boldsymbol{y}_h^a;\boldsymbol{\alpha}_h,\boldsymbol{v}_h) + a^h(\boldsymbol{y}_h^a;\boldsymbol{u}_h^b,\boldsymbol{v}_h) - a^h(\boldsymbol{y}_h^b;\boldsymbol{u}_h^b,\boldsymbol{v}_h), \\
		c^h(\boldsymbol{u}_h^a,\boldsymbol{u}_h^a,\boldsymbol{v}_h) - c^h(\boldsymbol{u}_h^b,\boldsymbol{u}_h^b,\boldsymbol{v}_h) &= c^h(\boldsymbol{u}_h^a,\boldsymbol{\alpha}_h,\boldsymbol{v}_h) - c^h(\boldsymbol{u}_h^a,\boldsymbol{u}_h^b,\boldsymbol{v}_h) + c^h(\boldsymbol{u}_h^b,\boldsymbol{u}_h^b,\boldsymbol{v}_h), \\
		c^h_{\boldsymbol{y}}(\boldsymbol{u}_h^a,\boldsymbol{y}_h^a,\boldsymbol{s}_h) - c^h_{\boldsymbol{y}}(\boldsymbol{u}_h^b,\boldsymbol{y}_h^b,\boldsymbol{s}_h) &= c^h_{\boldsymbol{y}}(\boldsymbol{u}_h^a,\boldsymbol{\beta}_h,\boldsymbol{s}_h) - c^h_{\boldsymbol{y}}(\boldsymbol{u}_h^a,\boldsymbol{y}_h^b,\boldsymbol{s}_h) +
		c^h_{\boldsymbol{y}}(\boldsymbol{u}_h^b,\boldsymbol{y}_h^b,\boldsymbol{s}_h).
	\end{align*}	
	Utilizing H\"older's inequality, stability properties discussed in Section \ref{Sec:DSP}, Lemma \ref{ncEnergyEstState}, discrete Sobolev embedding \eqref{DiscSobEmbnc} and the assumption \eqref{ncUniquenessAssumption}, we can establish the following bounds:
	\begin{align}
		|a^h(\boldsymbol{y}_h^a;\boldsymbol{u}_h^b,\boldsymbol{v}_h) - a^h(\boldsymbol{y}_h^b;\boldsymbol{u}_h^b,\boldsymbol{v}_h)| &\leq \sum_{K \in {\mathcal{T}_h}} \norm{\nu(\boldsymbol{y}_h^a) - \nu(\boldsymbol{y}_h^b)}_{L^6(K)} \norm{\nabla_h \boldsymbol{u}_h^a}_{L^3(K)} \norm{\nabla_h \boldsymbol{v}_h}_{L^2(K)} \nonumber\\
		&\leq \gamma_{\nu}  \norm{\boldsymbol{y}_h^a - \boldsymbol{y}_h^b}_{L^6(\Omega)} \norm{\nabla_h \boldsymbol{u}_h^a}_{L^3(\Omega)} \norm{\boldsymbol{v}_h}_{1, \mathcal{T}_h^{nc}} \nonumber\\
		&\leq \gamma_{\nu} \; \tilde{M} \norm{\boldsymbol{\beta}_h}_{1,\mathcal{T}_h^{nc}} \norm{ \boldsymbol{v}_h}_{1,\mathcal{T}_h^{nc}}, \label{ahDifference}\\
		|c^h(\boldsymbol{u}_h^b,\boldsymbol{u}_h^b,\boldsymbol{v}_h) - c^h(\boldsymbol{u}_h^a,\boldsymbol{u}_h^b,\boldsymbol{v}_h)| &\leq C_e^{nc} \norm{\boldsymbol{\alpha}_h}_{1,\mathcal{T}_h^{nc}} \norm{\boldsymbol{u}_h^b}_{1,\mathcal{T}_h^{nc}} \norm{\boldsymbol{v}_h}_{1,\mathcal{T}_h^{nc}} \nonumber\\ 
		&\leq C_e^{nc} M_{\boldsymbol{u}}^{nc} \norm{\boldsymbol{\alpha}_h}_{1,\mathcal{T}_h^{nc}} \norm{\boldsymbol{v}_h}_{1,\mathcal{T}_h^{nc}}, \label{chDifference}\\
		|c^h_{\boldsymbol{y}}(\boldsymbol{u}_h^b,\boldsymbol{y}_h^b,\boldsymbol{s}_h) -
		c^h_{\boldsymbol{y}}(\boldsymbol{u}_h^a,\boldsymbol{y}_h^b,\boldsymbol{s}_h)| &\leq C_e^{nc} \norm{\boldsymbol{\alpha}_h}_{1,\mathcal{T}_h^{nc}} \norm{\boldsymbol{y}_h^b}_{1,\mathcal{T}_h^{nc}} \norm{\boldsymbol{s}_h}_{1,\mathcal{T}_h^{nc}} \nonumber\\
		&\leq C_e^{nc} M_{\boldsymbol{y}}^{nc} \norm{\boldsymbol{\alpha}_h}_{1,\mathcal{T}_h^{nc}} \norm{\boldsymbol{s}_h}_{1,\mathcal{T}_h^{nc}}, \label{chyDiff}\\
		|d^h(\boldsymbol{y}_h^a, \boldsymbol{v}_h) - d^h(\boldsymbol{y}_h^b, \boldsymbol{v}_h)| &\leq C_F \norm{\boldsymbol{\beta}_h}_{1,\mathcal{T}_h^{nc}} \norm{\boldsymbol{v}_h}_{1,\mathcal{T}_h^{nc}}. \label{dhDifference}
	\end{align}  
	Choosing $\boldsymbol{v}_h = \boldsymbol{\alpha}_h \in \boldsymbol{X}_h^{cr}$ and $\boldsymbol{s}_h = \boldsymbol{\beta}_{h}$ in \eqref{ncStateUniquenessEq1} and exploiting \eqref{nc:Coercivity}, \eqref{Positivitync:ch} and \eqref{Positivitync:chy} gives
	\begin{align}
		\alpha_a^{nc} \norm{\boldsymbol{\alpha}_h}^2_{1,\mathcal{T}_h^{nc}} &\leq  |a^h(\boldsymbol{y}_h^a;\boldsymbol{u}_h^a,\boldsymbol{\alpha}_h) - a^h(\boldsymbol{y}_h^a;\boldsymbol{u}_h^b,\boldsymbol{\alpha}_h)| + |c^h(\boldsymbol{u}_h^b,\boldsymbol{u}_h^b,\boldsymbol{\alpha}_h) - c^h(\boldsymbol{u}_h^a,\boldsymbol{u}_h^b,\boldsymbol{\alpha}_h)|  \label{ncStateUniquenessEq2}\\
		&~~~~~+ |d^h(\boldsymbol{y}_h^a,\boldsymbol{\alpha}_h) - d^h(\boldsymbol{y}_h^b,\boldsymbol{\alpha}_h)|, \nonumber\\
		\alpha_2^{nc} \norm{\boldsymbol{\beta}_h}^2_{1,\mathcal{T}_h^{nc}} &\leq |c^h_{\boldsymbol{y}}(\boldsymbol{u}_h^b,\boldsymbol{y}_h^b,\boldsymbol{\beta}_h) -
		c^h_{\boldsymbol{y}}(\boldsymbol{u}_h^a,\boldsymbol{y}_h^b,\boldsymbol{\beta}_h)|. \label{ncStateUniquenessEq3}
	\end{align}
	Invoking the above established bounds in \eqref{ncStateUniquenessEq2} and \eqref{ncStateUniquenessEq3}, we get
	\begin{align*}
		\alpha_a^{nc} \norm{\boldsymbol{\alpha}_h}_{1,\mathcal{T}_h^{nc}} &\leq \big(\gamma_{\nu} \; \tilde{M} + C_F \big) \norm{\boldsymbol{\beta}_h}_{1, \mathcal{T}_h^{nc}} + C_e^{nc} \; M_{\boldsymbol{u}}^{nc} \norm{\boldsymbol{\alpha}_h}_{1,\mathcal{T}_h^{nc}},\\
		\alpha_2^{nc} \norm{\boldsymbol{\beta}_h}_{1,\mathcal{T}_h^{nc}} &\leq C_e^{nc} M_{\boldsymbol{y}}^{nc} \norm{\boldsymbol{\alpha}_h}_{1,\mathcal{T}_h^{nc}}.
	\end{align*}	
	Adding the above two inequalities and simplifying gives the desired uniqueness of the discrete solutions.
\end{proof}
\section{A priori error analysis}\label{NA:apriori:State}
Let $\Pi_h: Q \longrightarrow Q_h^0$ be such that $(\Pi_h p - p, q_h) = 0$ for all $p \in Q$ and $q_h \in Q_h^0$ denote the pressure interpolation operator, $\boldsymbol{\Pi}_{nc}$ defined in \eqref{nonConf_interpolation} denote the nonconforming interpolation operator and without loss of generality let $\boldsymbol{P}_e$ denote the standard edge projection on to piecewise constants on $e$. Then under adequate regularity assumptions, the following approximation properties hold (cf. \cite{MR3213584,Ern_Book,EdgeProjection}):
\begin{align}\label{nc:approximation_properties}
	\begin{aligned}
		&\norm{\boldsymbol{u} - \boldsymbol{\Pi}_{nc} \boldsymbol{u}}_{1,\mathcal{T}_h^{nc}} \leq C (h^{3/2+\delta} \sqrt{\sigma} + h^{1/2+\delta} \sqrt{\nu_2}) \norm{\boldsymbol{u}}_{3/2+\delta,\Omega}, \\
		&\norm{\boldsymbol{y} - \boldsymbol{\Pi}_{nc} \boldsymbol{y}}_{1,\mathcal{T}_h^{nc}} \leq C h^{3/2+\delta} \sqrt{\bar{\sigma}} \norm{\boldsymbol{y}}_{3/2+\delta,\Omega}, \;\; \norm{p - \Pi_h p}_{0,\Omega} \leq C h^{1/2+\delta} \norm{p}_{1/2+\delta,\Omega}, \\
		& \norm{\boldsymbol{v} - \boldsymbol{P}_e \boldsymbol{v}}_{0,e} \leq C h^{1/2}_{K} |\boldsymbol{v}|_{1,K}, \; \forall \; \boldsymbol{v} \in \boldsymbol{H}^1(K) \;\left(\mbox{or,}\; \boldsymbol{v} \in \left[H^1(K)\right]^2\right),  \; e \in \mathcal{E}(\mathcal{T}_h).
	\end{aligned}
\end{align}
\begin{theorem}\label{NA_State}
Let $(\boldsymbol{u}, p, \boldsymbol{y})$ and $(\boldsymbol{u}_h, p_h, \boldsymbol{y}_h)$ be the solutions of \eqref{P:S} and \eqref{ncPh:S}, respectively. Furthermore, for $\delta \in (0,\frac{1}{2})$ and $\boldsymbol{u} \in \boldsymbol{H}_0^1(\Omega) \cap \boldsymbol{H}^{3/2 + \delta}(\Omega),\;p \in L_0^2(\Omega) \cap H^{1/2 + \delta}(\Omega),\;\boldsymbol{y} \in \left[H^{3/2 + \delta}(\Omega) \right]^2$. Then under the smallness assumption \eqref{Disc_StateUniquenessCondition}, there exists a positive constant $C$ independent of the mesh parameter $h$ such that
	\begin{align}
		\norm{\boldsymbol{u} - \boldsymbol{u}_h}_{1,\mathcal{T}_h^{nc}} + \norm{\boldsymbol{y} - \boldsymbol{y}_h}_{1,\mathcal{T}_h^{nc}} &\leq C h^{1/2+\delta} \big(\norm{\boldsymbol{u}}_{3/2+\delta,\Omega} + \norm{\boldsymbol{y}}_{3/2+\delta,\Omega} \big), \label{aPriori_uy} \\
		\norm{p - p_h}_{0,\Omega} &\leq C h^{1/2+\delta} \big(\norm{\boldsymbol{u}}_{3/2+\delta,\Omega} + \norm{\boldsymbol{y}}_{3/2+\delta,\Omega} + \norm{p}_{1/2+\delta,\Omega}\big). \label{aPriori_p}
	\end{align} 	
\end{theorem}

\begin{proof}
	Let us decompose the errors into
	\begin{align}\label{aPriori_eq0}
		\begin{aligned}
			e_{\boldsymbol{u}} := (\boldsymbol{u} - \boldsymbol{\chi}) + (\boldsymbol{\chi} - \boldsymbol{u}_h) &= \hat{e}_{\boldsymbol{u}} + \tilde{e}_{\boldsymbol{u}}, \;\; e_{\boldsymbol{y}} := (\boldsymbol{y} - \boldsymbol{\theta}) + (\boldsymbol{\theta} - \boldsymbol{y}_h) = \hat{e}_{\boldsymbol{y}} + \tilde{e}_{\boldsymbol{y}} \\ 
			&\hspace{-1cm}e_{p} := (p - \psi) + (\psi - p_h) = \hat{e}_{p} + \tilde{e}_{p}, \\ 
		\end{aligned}
	\end{align}
	where, $(\boldsymbol{\chi}, \psi, \boldsymbol{\theta}) \in \boldsymbol{V}_h^{cr} \times Q_h^0 \times \mathcal{M}_h^{cr}$. Now an application of integration by parts implies that $(\boldsymbol{u}, p, \boldsymbol{y})$ satisfies the following equation:
	\begin{align}\label{aPriori_eq1}
		a^{h}(\boldsymbol{y}; \boldsymbol{u}, \boldsymbol{v}_h) + c^{h}(\boldsymbol{u}, \boldsymbol{u}, \boldsymbol{v}_h) + b^{h}(\boldsymbol{v}_h,p) -  d^h(\boldsymbol{y},\boldsymbol{v}_h) &= \sum_{K \in {\mathcal{T}_h}} \int_{\partial K} \nu(T) (\nabla \boldsymbol{u} \boldsymbol{n}_K) \cdot \boldsymbol{v}_h \; ds \;\;
		 \forall  \boldsymbol{v}_h \in \boldsymbol{V}^{cr_0}_h.
	\end{align}
	Writing a discrete analogue of \eqref{aPriori_eq1} and subtracting the result, yields the following:
	\begin{align}\label{aPriori_eq2}
		&a^{h}(\boldsymbol{y}; \boldsymbol{u}, \boldsymbol{v}_h) - a^{h}(\boldsymbol{y}_h; \boldsymbol{u}_h, \boldsymbol{v}_h) + c^{h}(\boldsymbol{u}, \boldsymbol{u}, \boldsymbol{v}_h) - c^{h}(\boldsymbol{u}_h, \boldsymbol{u}_h, \boldsymbol{v}_h) + b^{h}(\boldsymbol{v}_h,p - p_h)  \\
		&\hspace{4cm}- d^h(\boldsymbol{y},\boldsymbol{v}_h) + d^h(\boldsymbol{y}_h,\boldsymbol{v}_h) = \sum_{K \in {\mathcal{T}_h}} \int_{\partial K} \nu(T) (\nabla \boldsymbol{u}  \boldsymbol{n}_K) \cdot \boldsymbol{v}_h \; ds. \nonumber
	\end{align}
	Furthermore, we can obtain the following:
	\begin{align}
		&b^h(\boldsymbol{u} - \boldsymbol{u}_h, q_h) = 0, \label{aPriori_eq3}\\
		&a^h_{\boldsymbol{y}}(\boldsymbol{y},\boldsymbol{s}_h) - a^h_{\boldsymbol{y}}(\boldsymbol{y}_h,\boldsymbol{s}_h) + c^h_{\boldsymbol{y}}(\boldsymbol{u},\boldsymbol{y},\boldsymbol{s}_h) - c^h_{\boldsymbol{y}}(\boldsymbol{u}_h,\boldsymbol{y}_h,\boldsymbol{s}_h) =  \sum_{K \in {\mathcal{T}_h}} \int_{\partial K} \boldsymbol{D} (\nabla \boldsymbol{y} \boldsymbol{n}_K) \cdot \boldsymbol{s}_h \; ds. \label{aPriori_eq4}
	\end{align}
	Now, using the error decomposition in \eqref{aPriori_eq0}, testing \eqref{aPriori_eq2} with $\boldsymbol{v}_h = \tilde{e}_{\boldsymbol{u}}$ and \eqref{aPriori_eq4} with $\boldsymbol{s}_h = \tilde{e}_{\boldsymbol{y}}$ and rearranging the terms, we end up with the following error equations:
	\begin{align}
		a^h(\boldsymbol{y}; \tilde{e}_{\boldsymbol{u}}, \tilde{e}_{\boldsymbol{u}}) + c^h(\boldsymbol{u}_h, \tilde{e}_{\boldsymbol{u}}, \tilde{e}_{\boldsymbol{u}}) &= T_1 + T_2 + T_3 + T_4, \label{aPriori_eq5}\\
		a^h(\tilde{e}_{\boldsymbol{y}}, \tilde{e}_{\boldsymbol{y}}) + c^h_{\boldsymbol{y}}(\boldsymbol{u}_h, \tilde{e}_{\boldsymbol{y}},\tilde{e}_{\boldsymbol{y}}) &= I_1 + I_2 + I_3, \label{aPriori_eq6}
	\end{align}
	where
	\begin{align*}
		&T_1 := a^h(\boldsymbol{y}_h; \boldsymbol{u}_h, \tilde{e}_{\boldsymbol{u}}) -  a^h(\boldsymbol{y}; \boldsymbol{u}_h, \tilde{e}_{\boldsymbol{u}}), \\
		&T_2 := \left[c^h(\boldsymbol{\chi}, \boldsymbol{u},  \tilde{e}_{\boldsymbol{u}}) - c^h(\boldsymbol{u}, \boldsymbol{u},  \tilde{e}_{\boldsymbol{u}})\right] + \left[c^h(\boldsymbol{u}_h, \boldsymbol{u},  \tilde{e}_{\boldsymbol{u}}) - c^h(\boldsymbol{\chi}, \boldsymbol{u},  \tilde{e}_{\boldsymbol{u}})\right]  - c^h(\boldsymbol{u}_h; \hat{e}_{\boldsymbol{u}}, \tilde{e}_{\boldsymbol{u}}), \\
		&T_3 := d^h(\boldsymbol{y}_h, \tilde{e}_{\boldsymbol{u}}) - d^h(\boldsymbol{y},\tilde{e}_{\boldsymbol{u}}), \;\; T_4 := \sum_{K \in {\mathcal{T}_h}} \int_{\partial K} \nu(T)( \nabla \boldsymbol{u}  \boldsymbol{n}_K) \cdot \tilde{e}_{\boldsymbol{u}} \; ds, \\
		& I_1 := - a^h_{\boldsymbol{y}}(\hat{e}_{\boldsymbol{y}}, \tilde{e}_{\boldsymbol{y}}), \;\; I_2 := - c^h_{\boldsymbol{y}}(\hat{e}_{\boldsymbol{u}}, \boldsymbol{y}, \tilde{e}_{\boldsymbol{y}}) -  c^h_{\boldsymbol{y}}(\tilde{e}_{\boldsymbol{u}}, \boldsymbol{y}, \tilde{e}_{\boldsymbol{y}}) - c^h_{\boldsymbol{y}}(\boldsymbol{u}_h,  \hat{e}_{\boldsymbol{y}}, \tilde{e}_{\boldsymbol{y}}),\\
		& I_3 :=  \sum_{K \in {\mathcal{T}_h}} \int_{\partial K} \boldsymbol{D} (\nabla \boldsymbol{y} \boldsymbol{n}_K) \cdot \tilde{e}_{\boldsymbol{y}} \; ds.
	\end{align*}
	Now we focus our attention on finding appropriate bounds for these terms. Following steps analogous to \eqref{ahDifference},  \eqref{chDifference} and \eqref{dhDifference}, we can get the following bounds:
	\begin{align*}
		&|T_1| \leq \gamma_{\nu} \tilde{M} \norm{\boldsymbol{y} - \boldsymbol{y}_h}_{1,\mathcal{T}_h^{nc}} \norm{\tilde{e}_{\boldsymbol{u}}}_{1,\mathcal{T}_h^{nc}}, \\
		&|T_2| \leq C_e^{nc} M_{\boldsymbol{u}} (\norm{\hat{e}_{\boldsymbol{u}}}_{1,\mathcal{T}_h^{nc}} + \norm{\tilde{e}_{\boldsymbol{u}}}_{1,\mathcal{T}_h^{nc}}) \norm{\tilde{e}_{\boldsymbol{u}}}_{1,\mathcal{T}_h^{nc}} + C_{e}^{nc} M_{\boldsymbol{u}}^{nc} \norm{\hat{e}_{\boldsymbol{u}}}_{1,\mathcal{T}_h^{nc}} \norm{\tilde{e}_{\boldsymbol{u}}}_{1,\mathcal{T}_h^{nc}}, \\
		&|T_3| \leq C_F (\norm{\hat{e}_{\boldsymbol{y}}}_{1,\mathcal{T}_h^{nc}} + \norm{\tilde{e}_{\boldsymbol{y}}}_{1,\mathcal{T}_h^{nc}}) \norm{\tilde{e}_{\boldsymbol{u}}}_{1,\mathcal{T}_h^{nc}}, \\
		&|I_1| \leq \hat{C}^{nc}_a \norm{\hat{e}_{\boldsymbol{y}}}_{1,\mathcal{T}_h^{nc}} \norm{\tilde{e}_{\boldsymbol{y}}}_{1,\mathcal{T}_h^{nc}}, \\
		&|I_2| \leq C_e^{nc} M_{\boldsymbol{y}} (\norm{\hat{e}_{\boldsymbol{u}}}_{1,\mathcal{T}_h^{nc}} + \norm{\tilde{e}_{\boldsymbol{u}}}_{1,\mathcal{T}_h^{nc}}) \norm{\tilde{e}_{\boldsymbol{y}}}_{1,\mathcal{T}_h^{nc}} + C_e^{nc} M_{\boldsymbol{u}}^{nc} \norm{\hat{e}_{\boldsymbol{y}}}_{1,\mathcal{T}_h^{nc}} \norm{\tilde{e}_{\boldsymbol{y}}}_{1,\mathcal{T}_h^{nc}}.
	\end{align*}
	The consistency error terms $T_4$ and $I_3$ which arise due to the nonconforming approximation are bounded by using the edge projection $\boldsymbol{P}_e$ and \cite[(10.3.9)]{brennerBook}, as follows:
\begin{align*}
	|T_4|
	&= \left| \sum_{e \in \mathcal{E}(\mathcal{T}_h)} \int_{e} \nu(T)\,(\nabla \boldsymbol{u}\,\boldsymbol{n}_{e}) \cdot \big(\tilde{\boldsymbol{e}}_{\boldsymbol{u}}^{+}-\tilde{\boldsymbol{e}}_{\boldsymbol{u}}^{-}\big)\; ds\right|  
	= \left| \sum_{e \in \mathcal{E}(\mathcal{T}_h)} \int_{e} \nu(T)\,\left((\nabla \boldsymbol{u}\,\boldsymbol{n}_{e})
	-\boldsymbol{P}_e(\nabla \boldsymbol{u}\,\boldsymbol{n}_{e})\right)\cdot \left(\tilde{\boldsymbol{e}}_{\boldsymbol{u}}^{+}-\tilde{\boldsymbol{e}}_{\boldsymbol{u}}^{-}\right)\; ds\right| \\
	&\le \nu_2 \Big(\sum_{e \in \mathcal{E}(\mathcal{T}_h)}  |e|^{-1} 
	\big\|(\nabla \boldsymbol{u}\,\boldsymbol{n}_{e})
	-\boldsymbol{P}_e(\nabla \boldsymbol{u}\,\boldsymbol{n}_{e})\big\|^2_{0,e}\Big)^{1/2}
	\Big(\sum_{e \in \mathcal{E}(\mathcal{T}_h)} |e| \,
	\big\|\tilde{\boldsymbol{e}}_{\boldsymbol{u}}^{+}-\tilde{\boldsymbol{e}}_{\boldsymbol{u}}^{-}\big\|^2_{0,e}\Big)^{1/2} \\
	&\le C \nu_2 \Big(\sum_{e \in \mathcal{E}(\mathcal{T}_h)}  |e|^{-1} 
	\big\|(\nabla \boldsymbol{u}\,\boldsymbol{n}_{e})
	-\boldsymbol{P}_e(\nabla \boldsymbol{u}\,\boldsymbol{n}_{e})\big\|^2_{0,e}\Big)^{1/2}
	\Big(\sum_{K \in \mathcal{T}_h} h^2_K \,|\tilde{\boldsymbol{e}}_{\boldsymbol{u}}|^2_{1,K}\Big)^{1/2} \\
	&\le C \nu_2\, h \Big(\sum_{e \in \mathcal{E}(\mathcal{T}_h)}  |e|^{-1} 
	\big\|(\nabla \boldsymbol{u}\,\boldsymbol{n}_{e})
	-\boldsymbol{P}_e(\nabla \boldsymbol{u}\,\boldsymbol{n}_{e})\big\|^2_{0,e}\Big)^{1/2}
	\norm{\tilde{\boldsymbol{e}}_{\boldsymbol{u}}}_{1,\mathcal{T}_h^{nc}} .
	\\
	|I_3|
	&= \Big| \sum_{e \in \mathcal{E}(\mathcal{T}_h)} \int_{e} \boldsymbol{D}\,(\nabla \boldsymbol{y}\,\boldsymbol{n}_{e}) \cdot \big(\tilde{\boldsymbol{e}}_{\boldsymbol{y}}^{+}-\tilde{\boldsymbol{e}}_{\boldsymbol{y}}^{-}\big)\; ds\Big|  
  \le C\, \hat{C}_a\, h \Big(\sum_{e \in \mathcal{E}(\mathcal{T}_h)} |e|^{-1}\,
	\big\|(\nabla \boldsymbol{y}\,\boldsymbol{n}_{e})
	-\boldsymbol{P}_e(\nabla \boldsymbol{y}\,\boldsymbol{n}_{e})\big\|^2_{0,e}\Big)^{1/2}
	\norm{\tilde{\boldsymbol{e}}_{\boldsymbol{y}}}_{1,\mathcal{T}_h^{nc}} .
\end{align*}

	Inserting the bounds into \eqref{aPriori_eq5} and \eqref{aPriori_eq6} and applying \eqref{Positivitync:ch}, \eqref{Positivitync:chy} and \eqref{nc:Coercivity} yields
	\begin{align}
		\alpha_a^{nc} \norm{\tilde{e}_{\boldsymbol{u}}}_{1,\mathcal{T}_h^{nc}} &\leq  \gamma_{\nu} \tilde{M} (\norm{\hat{e}_{\boldsymbol{y}}}_{1,\mathcal{T}_h^{nc}} + \norm{\tilde{e}_{\boldsymbol{y}}}_{1,\mathcal{T}_h^{nc}}) + C_e^{nc} M_{\boldsymbol{u}} (\norm{\hat{e}_{\boldsymbol{u}}}_{1,\mathcal{T}_h^{nc}} + \norm{\tilde{e}_{\boldsymbol{u}}}_{1,\mathcal{T}_h^{nc}}) + C_{e}^{nc} M_{\boldsymbol{u}}^{nc} \norm{\hat{e}_{\boldsymbol{u}}}_{1,\mathcal{T}_h^{nc}} \nonumber\\ 
		&\quad + C_F (\norm{\hat{e}_{\boldsymbol{y}}}_{1,\mathcal{T}_h^{nc}} + \norm{\tilde{e}_{\boldsymbol{y}}}_{1,\mathcal{T}_h^{nc}}) + C \nu_2 h \Big(\sum_{e \in {\mathcal{E}_{\mathcal{T}_h}}}  |e|^{-1} \|{\nabla \boldsymbol{u}  \boldsymbol{n}_{e} - \boldsymbol{P}_e \left(\nabla \boldsymbol{u}  \boldsymbol{n}_{e}\right)\|}^2_{0,e}\Big)^{1/2}, \label{aPriori_eq7} \\
		\alpha_2^{nc} \norm{\tilde{e}_{\boldsymbol{y}}}_{1,\mathcal{T}_h^{nc}} &\leq \hat{C}^{nc}_a \norm{\hat{e}_{\boldsymbol{y}}}_{1,\mathcal{T}_h^{nc}} +  C_e^{nc} M_{\boldsymbol{y}} \big(\norm{\hat{e}_{\boldsymbol{u}}}_{1,\mathcal{T}_h^{nc}} + \norm{\tilde{e}_{\boldsymbol{u}}}_{1,\mathcal{T}_h^{nc}}\big) + C_e^{nc} M_{\boldsymbol{u}}^{nc} \norm{\hat{e}_{\boldsymbol{y}}}_{1,\mathcal{T}_h^{nc}} \nonumber\\
		&\quad + C \hat{C}_a h \Big(\sum_{e \in {\mathcal{E}_{\mathcal{T}_h}}} |e|^{-1} \; \|{\nabla \boldsymbol{y}  \boldsymbol{n}_{e} - \boldsymbol{P}_e \left(\nabla \boldsymbol{y} \boldsymbol{n}_{e}\right)\|}^2_{0,e}\Big)^{1/2}. \label{aPriori_eq8}
	\end{align}
	Substituting \eqref{aPriori_eq8} in \eqref{aPriori_eq7} gives
	\begin{align}
		&\Big(\alpha_a^{nc} - \frac{C_e^{nc} M^{nc}_{\boldsymbol{y}}}{\alpha_2^{nc}} \big(\gamma_{\nu} \tilde{M} + C_{F} \big) + C_e^{nc}  M_{\boldsymbol{u}}^{nc}\Big) \norm{\tilde{e}_{\boldsymbol{u}}}_{1,\mathcal{T}_h^{nc}} \nonumber\\&\leq C \Big[\norm{\hat{e}_{\boldsymbol{u}}}_{1,\mathcal{T}_h^{nc}} + \norm{\hat{e}_{\boldsymbol{y}}}_{1,\mathcal{T}_h^{nc}} + h \Big(\sum_{e \in {\mathcal{E}_{\mathcal{T}_h}}}  |e|^{-1} \|{\nabla \boldsymbol{u}  \boldsymbol{n}_{e} - \boldsymbol{P}_e \left(\nabla \boldsymbol{u}  \boldsymbol{n}_{e}\right)\|}^2_{0,e}\Big)^{1/2}  \nonumber\\
		&\qquad+ h \Big(\sum_{e \in {\mathcal{E}_{\mathcal{T}_h}}} |e|^{-1} \; \|{\nabla \boldsymbol{y}  \boldsymbol{n}_{e} - \boldsymbol{P}_e \left(\nabla \boldsymbol{y}  \boldsymbol{n}_{e}\right)\|}^2_{0,e}\Big)^{1/2} \Big]. \nonumber
	\end{align}
	Choosing $\boldsymbol{\chi} = \boldsymbol{\Pi}_{nc} \boldsymbol{u}, \boldsymbol{\theta} = \boldsymbol{\Pi}_{nc} \boldsymbol{y}$ and invoking the approximation properties given in \eqref{nc:approximation_properties} straightforwardly leads to the estimate \eqref{aPriori_uy}.  To derive the pressure estimate, we exploit the discrete inf-sup condition \eqref{discreteinfsup} and Lemma \ref{ncStateCty} as follows:
	\begin{align}\label{aPriori_eq9}
		\norm{\tilde{e}_p}_{0,\Omega} \leq \frac{1}{\tilde{\beta}} \sup_{\boldsymbol{v}_h \in \boldsymbol{V}_h \backslash \{\boldsymbol{0}\}} \left(\frac{b^h(\boldsymbol{v}_h, e_p)}{\norm{\boldsymbol{v}_h}_{1,\mathcal{T}_h^{nc}}} + \frac{b^h(\boldsymbol{v}_h, \hat{e}_p)}{\norm{\boldsymbol{v}_h}_{1,\mathcal{T}_h^{nc}}}\right) \leq \frac{1}{\tilde{\beta}} \sup_{\boldsymbol{v}_h \in \boldsymbol{V}_h \backslash \{\boldsymbol{0}\}} \frac{b^h(\boldsymbol{v}_h, e_p)}{\norm{\boldsymbol{v}_h}_{1,\mathcal{T}_h^{nc}}} + \frac{1}{\tilde{\beta}} \norm{\hat{e}_p}_{0,\Omega}.
	\end{align}
	For any $\boldsymbol{v}_h \in \boldsymbol{V}_h^{cr}$, \eqref{aPriori_eq2} implies
	\begin{align}
		b^h(\boldsymbol{v}_h, p_h - p) &= a^{h}(\boldsymbol{y}; \boldsymbol{u}, \boldsymbol{v}_h) - a^{h}(\boldsymbol{y}_h; \boldsymbol{u}_h, \boldsymbol{v}_h) + c^{h}(\boldsymbol{u}, \boldsymbol{u}, \boldsymbol{v}_h) - c^{h}(\boldsymbol{u}_h, \boldsymbol{u}_h, \boldsymbol{v}_h) \nonumber\\
		&\qquad - d^h(\boldsymbol{y},\boldsymbol{v}_h) + d^h(\boldsymbol{y}_h,\boldsymbol{v}_h) - \sum_{K \in {\mathcal{T}_h}} \int_{\partial K} \nu(T)( \nabla \boldsymbol{u}  \boldsymbol{n}_K) \cdot \boldsymbol{v}_h \; ds. 
	\end{align}
	Utilizing \eqref{ahDifference}, \eqref{chDifference}, \eqref{dhDifference} and analogous steps used to bound $|T_4|$, we get
	\begin{align}\label{aPriori_eq10}
		|b^h(\boldsymbol{v}_h, p_h-p)| &\leq \norm{\boldsymbol{v}_h}_{1,\mathcal{T}_h^{nc}} \Big( \gamma_{\nu} \tilde{M} \norm{e_{\boldsymbol{y}}}_{1,\mathcal{T}_h^{nc}}  + C_e^{nc} M_{\boldsymbol{u}}^{nc} \norm{e_{\boldsymbol{u}}}_{1,\mathcal{T}_h^{nc}}  + C_F \norm{e_{\boldsymbol{y}}}_{1,\mathcal{T}_h^{nc}} \nonumber\\
		&\quad  + C \nu_2 h \Big(\sum_{e \in {\mathcal{E}_{\mathcal{T}_h}}}  |e|^{-1} \|{\nabla \boldsymbol{u}  \boldsymbol{n}_{e} - \boldsymbol{P}_e \left(\nabla \boldsymbol{u} \boldsymbol{n}_{e}\right)\|}^2_{0,e}\Big)^{1/2} \Big).
	\end{align}
	Thus the estimate \eqref{aPriori_p} follows by using \eqref{aPriori_eq10} in \eqref{aPriori_eq9}, choosing $\psi = \Pi_h p$ and using \eqref{aPriori_uy} and \eqref{nc:approximation_properties}.
\end{proof}

	\begin{remark}\label{NA_State:convex.domains}
			For convex domains $\delta = 1/2$. Then using Theorem \ref{H2regularity} and following same steps as in the above proof, we obtain the following optimal error bounds on the discrete velocity and pressure,
			\begin{align*}
				\norm{\boldsymbol{u} - \boldsymbol{u}_h}_{1,\mathcal{T}_h^{nc}} + \norm{\boldsymbol{y} - \boldsymbol{y}_h}_{1,\mathcal{T}_h^{nc}} &\leq C h \big(\norm{\boldsymbol{u}}_{2,\Omega} + \norm{\boldsymbol{y}}_{2,\Omega} \big),  \\
				\norm{p - p_h}_{0,\Omega} &\leq C h \big(\norm{\boldsymbol{u}}_{2,\Omega} + \norm{\boldsymbol{y}}_{2,\Omega} + \norm{p}_{1,\Omega}\big). 
			\end{align*} 	
	\end{remark}
	
\begin{remark}[Pressure robustness]\label{rem:pressure-robustness}
		In the literature, pressure robustness is often understood in the sense of \cite{john2017divergence}, namely that the discrete velocity is invariant under the addition of gradient forces to the momentum equation. Equivalently, the discrete velocity depends only on the divergence-free part of the forcing in the Helmholtz--Hodge decomposition \cite{linke2014role}. It is shown that CR-DG pair does not automatically enjoy this invariance property, despite yielding locally divergence-free discrete velocities. This is because nonconformity leads to additional interface contributions and a reconstruction is required to recover full pressure robustness as shown in \cite{john2017divergence,linke2014role}. 
		
	    Since our error analysis is based on reduced model problem in which pressure is elimiated and recovered later, our velocity error estimates are independent of pressure variable. However, our scheme is not pressure robust in the sense of \cite{john2017divergence}, as is evident from velocity and pressure errors in the Darcy regime in Experiment \ref{exp:acc.convex}.
\end{remark}	

\section{Numerical Experiments}\label{NE} 
In this section we provide a set of examples which verify the convergence rates proved in  Theorems \ref{NA_State} (see Remark \ref{NA_State:convex.domains} for convex domains) and verify the proposed method on the porous-cavity benchmark problem by comparing with data from the literature. We employ the Newton's method to solve the discretised problem, resulting in a linearised system which is solved using the linear solver \textit{MUMPS} in open source finite element library FEniCS \cite{alnaes2015fenics}. The essential boundary conditions are implemented using the Dirichlet boundary class feature of FEniCS and the zero mean value for the pressure approximation is handled using a Lagrange multiplier approach.


\subsection{Accuracy test on convex domain $(\delta = 1/2)$}\label{exp:acc.convex} Let the computational domain be $\Omega := [0,1] \times [0,1]$ and consider a sequence of uniformly refined meshes on $\Omega$ with mesh parameter $h$.  We take the buoyancy term  of the form $\boldsymbol{F}(\boldsymbol{y}) = (T + N_{r} S) \boldsymbol{g},$ where $N_r$ is the solutal to the thermal buoyancy ratio, and the viscosity is chosen to be of exponential form, that is, $\nu(T) = \nu_2 \exp(-T)$. The other parameters are defined as follows: $\boldsymbol{g} = (0,1)^{\top}, \boldsymbol{K}^{-1} = \sigma \boldsymbol{I}, \boldsymbol{D} = 1000 \boldsymbol{I}.$ We use the following smooth functions as the closed form solutions of \eqref{P:GE} in our accuracy test:
\begin{align*}
	&\boldsymbol{u}(x,y) = \left(\sin(\pi x) \cos(\pi y), - \cos(\pi x) \sin(\pi y)\right)^{\top}, \quad  p(x,y) = \cos(\pi x) \exp(y), \\
	&T(x,y) = 0.5 + 0.5 \cos(xy), \quad S(x,y) = 0.1+0.3\exp(xy).
\end{align*}
The boundary conditions, additional external forces and source terms are prescribed according to the above prescribed manufactured solutions. The errors for all the variable in their respective discrete norms and the corresponding convergence rates are denoted as follows:
\begin{align*}
	&\boldsymbol{e}_{\boldsymbol{u}} = \norm{\boldsymbol{u} - \boldsymbol{u}_h}_{1,\mathcal{T}_h^{nc}},\quad \boldsymbol{e}_{p} = \norm{p - p_h}_{0,\Omega},\quad \boldsymbol{e}_{T} = \norm{T-T_h}_{1,\mathcal{T}_h^{nc}},\\
	&\boldsymbol{e}_{S} = \norm{S - S_h}_{1,\mathcal{T}_h^{nc}},\quad \boldsymbol{rate} = \log(\boldsymbol{e}_{(\cdot)} / \tilde{\boldsymbol{e}}_{(\cdot)}) (\log(h / \tilde{h}))^{-1},
\end{align*}
where $\boldsymbol{e}, \tilde{\boldsymbol{e}}$ denote the errors generated on two consecutive mesh refinements of mesh size $h$ and $\tilde{h},$ respectively. The errors and convergence rates  focusing on cases where the viscosity and permeability coefficients scale differently, changing from Flow to Darcy regimes are reported in Table \ref{table1}; where $\boldsymbol{Itr}$ denotes the number of iterations the \textit{Newton} algorithm takes to reach the absolute tolerance of $10^{-8}$ and $\boldsymbol{\chi}_{(\cdot)}$ denotes the degrees of freedom associated with the space $(\cdot)$. The tabulated errors confirm the theoretical bounds of Remark \ref{NA_State:convex.domains} and that the discrete velocity is divergence free. 
 
\begin{longtable}{@{}|l|l|l|l|l|l|l|l|l|l|l|l|l|@{}}
	
	\caption{Experimental errors and order of convergence in the discrete norms under flow, Stokes and Darcy regimes.} \label{table1}\\
	\hline
	\multicolumn{13}{|c|}{\textbf{Flow regime} ($\nu_2 = 1, \; \sigma = 0$)} \\ \hline
	$\boldsymbol{\chi}_{\boldsymbol{V}_h^{cr}}$ & $\boldsymbol{e}_{\boldsymbol{u}}$ & $\boldsymbol{rate}$ & $\|\bdiv_h \boldsymbol{u}_h\|_{\infty}$ & $\boldsymbol{\chi}_{\mathcal{M}_h^{cr}}$ & $\boldsymbol{e}_T$ & $\boldsymbol{rate}$ & $\boldsymbol{e}_S$ & $\boldsymbol{rate}$ & $\boldsymbol{\chi}_{Q_h^0}$ & $\boldsymbol{e}_p$ & $\boldsymbol{rate}$ & $\boldsymbol{Itr}$ \\ \hline
	112   & 0.71816  & -    & 4.44E-16 & 56    & 2.54079  & -    & 2.31212  & -    & 32    & 1.41664  & -    & 4 \\
	416   & 0.36977  & 0.96 & 1.33E-15 & 208   & 1.27796  & 1.00 & 1.16021  & 1.00 & 128   & 0.82215  & 0.78 & 4 \\
	1600  & 0.18633  & 0.99 & 2.66E-15 & 800   & 0.63784  & 1.00 & 0.58045  & 1.00 & 512   & 0.44617  & 0.88 & 4 \\
	6272  & 0.09334  & 1.00 & 5.33E-15 & 3136  & 0.31892  & 1.00 & 0.29045  & 1.00 & 2048  & 0.23265  & 0.94 & 4 \\
	24832 & 0.04665  & 1.00 & 1.07E-14 & 12416 & 0.15957  & 1.00 & 0.14539  & 1.00 & 8192  & 0.11800  & 0.98 & 4 \\
	98816 & 0.02332  & 1.00 & 2.13E-14 & 49408 & 0.07985  & 1.00 & 0.07269  & 1.00 & 32768 & 0.05931  & 1.00 & 4 \\ \hline
	
	\multicolumn{13}{|c|}{\textbf{Stokes regime} ($\nu_2 = 10, \; \sigma = 0$)} \\ \hline
	112   & 2.05517  & -    & 8.88E-16 & 56    & 3.94857  & -    & 6.50500  & -    & 32    & 1.47301  & -    & 6 \\
	416   & 1.10270  & 0.90 & 1.78E-15 & 208   & 2.00340  & 0.98 & 3.29944  & 0.98 & 128   & 0.86344  & 0.77 & 6 \\
	1600  & 0.56386  & 0.97 & 4.00E-15 & 800   & 1.00983  & 0.99 & 1.66376  & 0.99 & 512   & 0.45687  & 0.92 & 5 \\
	6272  & 0.28416  & 0.99 & 7.11E-15 & 3136  & 0.50703  & 0.99 & 0.83577  & 1.00 & 2048  & 0.23322  & 0.97 & 5 \\
	24832 & 0.14253  & 1.00 & 1.42E-14 & 12416 & 0.25405  & 1.00 & 0.41891  & 1.00 & 8192  & 0.11757  & 0.99 & 5 \\
	98816 & 0.07137  & 1.00 & 3.55E-14 & 49408 & 0.12716  & 1.00 & 0.20972  & 1.00 & 32768 & 0.05899  & 1.00 & 5 \\ \hline
	
	\multicolumn{13}{|c|}{\textbf{Darcy regime} ($\nu_2 = 1, \; \sigma = 10000$)} \\ \hline
	112   & 16.330   & -    & 1.33E-15 & 56    & 3.94857  & -    & 6.50501  & -    & 32    & 171.156  & -    & 4 \\
	416   & 7.9157   & 1.04 & 2.22E-15 & 208   & 2.00340  & 0.98 & 3.29944  & 0.98 & 128   & 99.6558  & 0.78 & 4 \\
	1600  & 3.84327  & 1.04 & 3.55E-15 & 800   & 1.00983  & 0.99 & 1.66377  & 0.99 & 512   & 52.5293  & 0.92 & 4 \\
	6272  & 1.88428  & 1.03 & 7.11E-15 & 3136  & 0.50703  & 0.99 & 0.83577  & 0.99 & 2048  & 26.7997  & 0.97 & 4 \\
	24832 & 0.93138  & 1.02 & 1.42E-14 & 12416 & 0.25405  & 1.00 & 0.41891  & 1.00 & 8192  & 13.5129  & 0.99 & 4 \\
	98816 & 0.46281  & 1.01 & 2.84E-14 & 49408 & 0.12716  & 1.00 & 0.20972  & 1.00 & 32768 & 6.78184  & 0.99 & 4 \\
	\hline
\end{longtable}

\begin{remark}
	In the Darcy regime to handle the reaction domination we modify our scheme by adding the following penalty in the discrete bilinear form $a^h(\cdot; \cdot, \cdot)$:
	$\sum_{e \in \mathcal{E}(\mathcal{T}_h)} \int_{e}\frac{a_0}{h_e} \nu_2 \jump{\boldsymbol{u}_h} : \jump{\boldsymbol{v}_h} \; ds,$
	with the stabilization parameter $a_0 = 10 \sqrt{\sigma}$ and calculate the error in the modified natural norm
	$$\norm{\boldsymbol{v}}_{1,\mathcal{T}_h^{nc}} =  \sum_{K \in {\mathcal{T}_h}} \sigma \norm{\boldsymbol{v}}^2_{0,K} + \nu_2 \norm{\nabla \boldsymbol{v}}^2_{0,K} + \sum_{e \in \mathcal{E}(\mathcal{T}_h)} \frac{1}{h_e} \|\jump{\boldsymbol{v}}\|^2_{0,e}, \; \forall \; \boldsymbol{v} \in \boldsymbol{H}^1({\mathcal{T}_h}).$$
\end{remark}

	\subsection{Accuracy test on non-convex domain $(\delta < 1/2)$}\label{exp:L}  Consider a sequence of uniformly refined meshes discretizing the L-shaped domain $\Omega = (-1,1)^2 \backslash [0,1]^2$. We take same data and parameters as in  Experiment \ref{exp:acc.convex}. We use the following closed form solutions of \eqref{P:GE} for this test,
	
	\begin{align*}
		&\boldsymbol{u}(x,y) = \left(\sin(\pi x) \cos(\pi y), - \cos(\pi x) \sin(\pi y)\right)^{\top}, \quad  p(x,y) = r^{2/3} \cos\left(\frac{2}{3} \theta\right), \\
		&T(x,y) = 0.5 + 0.5  \left(r^{2/3} \cos\left(\frac{2}{3} \theta\right)\right) , \quad S(x,y) = 0.1+0.3\left( r^{2/3} \sin\left(\frac{2}{3} \theta\right)\right).
	\end{align*}
	
	The regularity of solutions $T, S$ and $p$ is $H^{5/6 - \epsilon}(\Omega) (\epsilon > 0)$. So, we expect the convergence rate of $\mathcal{O}(h^{1/2 + 1/6})$ in the $\| \cdot \|_{1,\mathcal{T}_h^{nc}}$-norm for $T$ and $S$, which we observe for all regimes in Table \ref{table:Lshape} for $T$ and $S$, which verifies the results of Theorem \ref{NA_State}. The convergence rate for $p$ in the $L^2$-norm is still optimal which is expected given the regularity of the manufactured $p$.      

\begin{longtable}{@{}|l|l|l|l|l|l|l|l|l|l|l|l|l|@{}}
	
	\caption{Experimental errors and order of convergence in the discrete norms for Experiment \ref{exp:L} under flow, Stokes and Darcy regimes.}\label{table:Lshape}\\
	\hline
	\multicolumn{13}{|c|}{\textbf{Flow regime} ($\nu_2 = 1,\; \sigma = 0$)} \\ \hline
	$\boldsymbol{\chi}_{\boldsymbol{V}_h^{cr}}$ & $\boldsymbol{e}_{\boldsymbol{u}}$ & $\boldsymbol{rate}$ &
	$\|\bdiv_h \boldsymbol{u}_h\|_{\infty}$ &
	$\boldsymbol{\chi}_{\mathcal{M}_h^{cr}}$ & $\boldsymbol{e}_T$ & $\boldsymbol{rate}$ &
	$\boldsymbol{e}_S$ & $\boldsymbol{rate}$ &
	$\boldsymbol{\chi}_{Q_h^0}$ & $\boldsymbol{e}_p$ & $\boldsymbol{rate}$ & $\boldsymbol{Itr}$ \\ \hline
	608    & 1.271  & -    & 1.33E-15 & 304    & 8.543 & -    & 6.793 & -    & 192    & 0.2827  & -    & 7 \\
	2368   & 0.6453 & 0.98 & 2.89E-15 & 1184   & 5.461 & 0.65 & 4.331 & 0.65 & 768    & 0.1516  & 0.90 & 6 \\
	9344   & 0.3247 & 0.99 & 7.11E-15 & 4672   & 3.473 & 0.65 & 2.746 & 0.66 & 3072   & 0.0781 & 0.96 & 6 \\
	37120  & 0.1628 & 1.00 & 1.42E-14 & 18560  & 2.200 & 0.66 & 1.736 & 0.66 & 12288  & 0.0396 & 0.98 & 6 \\
	147968 & 0.0815 & 1.00 & 2.84E-14 & 73984  & 1.391 & 0.66 & 1.096 & 0.66 & 49152  & 0.0199 & 0.99 & 6 \\
	\hline
	
	\multicolumn{13}{|c|}{\textbf{Stokes regime} ($\nu_2 = 10,\; \sigma = 0$)} \\ \hline
	608    & 4.041  & -    & 1.33E-15 & 304    & 8.543 & -    & 6.793 & -    & 192    & 3.0293  & -    & 6 \\
	2368   & 2.045  & 0.98 & 3.11E-15 & 1184   & 5.461 & 0.65 & 4.331 & 0.65 & 768    & 1.5602  & 0.96 & 6 \\
	9344   & 1.026  & 0.99 & 6.66E-15 & 4672   & 3.473 & 0.65 & 2.746 & 0.66 & 3072   & 0.7868   & 0.99 & 6 \\
	37120  & 0.5140 & 1.00 & 1.42E-14 & 18560  & 2.200 & 0.66 & 1.736 & 0.66 & 12288  & 0.3945 & 1.00 & 6 \\
	147968 & 0.2571 & 1.00 & 2.84E-14 & 73984  & 1.391 & 0.66 & 1.096 & 0.66 & 49152  & 0.1974 & 1.00 & 6 \\
	\hline
	
	\multicolumn{13}{|c|}{\textbf{Darcy regime} ($\nu_2 = 1,\; \sigma = 10000$)} \\ \hline
	160     & 24.00  & -    & 8.88E-16 & 80      & 13.27  & -    & 10.53  & -    & 48      & 95.924     & -    & 4 \\
	608     & 5.619  & 2.09 & 1.78E-15 & 304     & 8.543  & 0.64 & 6.793  & 0.63 & 192     & 17.204    & 2.48 & 4 \\
	2368    & 1.530  & 1.88 & 3.11E-15 & 1184    & 5.461  & 0.65 & 4.331  & 0.65 & 768     & 2.4061   & 2.84 & 4 \\
	9344    & 0.5185 & 1.56 & 7.11E-15 & 4672    & 3.473 & 0.65 & 2.746  & 0.66 & 3072    & 0.3643    & 2.72 & 4 \\
	37120   & 0.2068 & 1.33 & 1.42E-14 & 18560   & 2.200  & 0.66 & 1.736  & 0.66 & 12288   & 0.0829  & 2.14 & 4 \\
	147968  & 0.09150& 1.18 & 2.84E-14 & 73984   & 1.391  & 0.66 & 1.096  & 0.66 & 49152   & 0.0317  & 1.39 & 4 \\
	590848  & 0.04291& 1.09 & 5.68E-14 & 295424  & 0.878 & 0.66 & 0.691 & 0.66 & 19660  & 0.0149  & 1.08 & 4 \\
	2361344 & 0.02075& 1.05 & 1.42E-13 & 1180672 & 0.554 & 0.66 & 0.436 & 0.67 & 78643  & 0.0074 & 1.02 & 4 \\
	\hline
	
\end{longtable}

	\subsection{Soret and Dufour effects in a porous cavity benchmark} 

We consider the porous-cavity benchmark problem on $\Omega=(0,1)^2$ in the Darcy regime
(see \cite[Example 2]{burger2019h} and the references \cite{MR2999596,goyeau1996numerical}). The physical
parameters are fixed as $R_k=1, Da=10^{-7}, Le=10, Pr=10, Sr=0, Du=0, N=0$,
with $\mathbf{g}=(0,-1)^\top$, $\nu(T)=1$, and $\mathbf{K}=Da\,\mathbf{I}$. The buoyancy forcing is given by
$\mathbf{F}(\boldsymbol{y})=(Gr_T\,T+Gr_C\,C)\mathbf{g}$, where $Gr_T=Ra/(Pr\,Da)$ and $Gr_C=N\,Gr_T$.
The temperature and concentration are coupled through the diffusion operator
$$
-\nabla\cdot\big(\mathbf{D}\nabla \boldsymbol{y}\big),\qquad \boldsymbol{y}=(T,C)^\top,
\qquad 
\mathbf{D}=
\begin{pmatrix}
	\frac{R_k}{Pr} & Du \\
	Sr & \frac{1}{Sc}
\end{pmatrix},
\qquad Sc=Le\,Pr.
$$
On the left and right walls $\Gamma_L$ and $\Gamma_R$ we prescribe Dirichlet data
$T=1,\ C=1 \ \text{on }\Gamma_L,  T=0,\ C=0 \ \text{on }\Gamma_R$,
while the top and bottom walls $\Gamma_T$ and $\Gamma_B$ are impermeable and adiabatic.
The wall-averaged Nusselt and Sherwood numbers on $\Gamma_L$ are defined by
\begin{align*}
	Nu = \int_0^1 \left.\frac{\partial T}{\partial x}\right|_{x=0}\,dy,
	\qquad
	Sh = \int_0^1 \left.\frac{\partial C}{\partial x}\right|_{x=0}\,dy.
\end{align*}
In our implementation these are evaluated via boundary integration. Since the outward unit normal on
$\Gamma_L$ is $\mathbf{n}=(-1,0)^\top$, we compute
\[
Nu=\int_{\Gamma_L}-\nabla T_h\cdot \mathbf{n}\,ds,\qquad
Sh=\int_{\Gamma_L}-\nabla C_h\cdot \mathbf{n}\,ds,
\]
and we report $|Nu|$ and $|Sh|$ for different values of the Rayliegh number $Ra$ in Tables \ref{table:Nu} and \ref{table:Sh}, respectively. 

\begin{longtable}{|l|l|l|l|l|l|}
	\caption{Nusselt numbers on $\Gamma_L$ for the proposed scheme in the Darcy regime on $150\times 150$ structured grid.}
	\label{table:Nu}\\
	\hline
	\multicolumn{6}{|c|}{\textbf{Darcy regime}} \\ \hline
	$Ra$ &
	$|Nu|$ & $\boldsymbol{Itr}$ &
	$Nu$ \cite{burger2019h} &
	$Nu$ \cite{MR2999596} &
	$Nu$ \cite{goyeau1996numerical} \\
	\hline
	100  & 3.0992  & 7  & 3.10 & 3.15 & 3.11 \\
	200  & 4.9256  & 7  & 4.97 & 5.02 & 4.96 \\
	400  & 7.6641  & 9  & 7.84 & 7.83 & 7.77 \\
	1000 & 13.0682 & 12 & 13.72 & 14.01 & 13.47 \\
	\hline
\end{longtable}

\begin{longtable}{|l|l|l|l|l|}
	\caption{Sherwood numbers on $\Gamma_L$ for the proposed scheme in the Darcy regime on $150\times 150$ structured grid.}
	\label{table:Sh}\\
	\hline
	\multicolumn{5}{|c|}{\textbf{Darcy regime}} \\ \hline
	$Ra$ &
	$|Sh|$ &
	$Sh$ \cite{burger2019h} &
	$Sh$ \cite{MR2999596} &
	$Sh$ \cite{goyeau1996numerical} \\
	\hline
	100  & 13.2135 & 13.58 & 13.54 & 13.25 \\
	200  & 19.4361 & 20.73 & 20.11 & 19.86 \\
	400  & 27.7142 & 30.91 & 27.96 & 28.41 \\
	1000 & 42.1791 & 49.42 & 48.01 & 48.32 \\
	\hline
\end{longtable}

The proposed scheme converges robustly for $Ra\le 1000$ using $7$--$12$ Newton 
iterations to reach the relative tolerance of $10^{-9}$. The computed $|Nu|$ agrees well with the benchmark values at Rayleigh numbers $Ra\leq1000$.
For $|Sh|$, the discrepancy increases with $Ra$, which we suspect is due to development of thin boundary layers that are challenging to resolve using lowest-order discretizations on uniform grids.

\section*{Acknowledgements}
The authors greatly acknowledge the funding from SERB-CRG India (Grant Number : CRG/2021/002569). The first author also gratefully acknowledges helpful discussions with Dr. Kush Kinra and thanks Indian Institute of Technology Roorkee, where majority of this works was done.

\addcontentsline{toc}{section}{References}
\bibliographystyle{ieeetr} 
\bibliography{References}
.\end{document}